\documentclass[a4paper, 11 pt]{article}
\usepackage{a4wide, amsmath, amssymb, mathtools, yfonts}
\usepackage{comment}
\usepackage{tikz}
\usepackage{tikz-cd}
\usepackage[all]{xy}
\usepackage[utf8]{inputenc}
\usepackage{amsthm, mathrsfs}
\usepackage[english]{babel}
\usepackage{hyperref}
\usepackage{authblk}
\usepackage[OT2,T1]{fontenc}
\DeclareSymbolFont{cyrletters}{OT2}{wncyr}{m}{n}
\DeclareMathSymbol{\Sha}{\mathalpha}{cyrletters}{"58}

\numberwithin{equation}{section}
\newtheorem{lemma}{Lemma}[section]
\newtheorem{theorem}[lemma]{Theorem}
\newtheorem{proposition}[lemma]{Proposition}

\newtheorem{conjecture}{Conjecture}
\newtheorem{heuristic}[lemma]{Heuristic}

\theoremstyle{definition}
\newtheorem{mydef}[lemma]{Definition}
\newtheorem{example}[lemma]{Example}

\newtheorem{remark}[lemma]{Remark}

\newcommand{\Aff}{\mathbb{A}}
\newcommand{\Spec}{\textup{Spec}}

\newcommand{\Z}{\mathbb{Z}}
\newcommand{\Q}{\mathbb{Q}}
\newcommand{\C}{\mathbb{C}}
\newcommand{\QQ}{\Q}
\newcommand{\FFF}{\mathbb{F}}
\newcommand{\Qbar}{\overline{\Q}}
\newcommand{\R}{\mathbb{R}}
\newcommand{\FF}{\mathbb{F}}

\newcommand{\Hom}{\mathrm{Hom}}
\newcommand{\Frob}{\textup{Frob}}

\newcommand\Gal{\mathrm{Gal}}

\newcommand{\Sel}{\textup{Sel}}
\newcommand{\Vplac}{\mathscr{V}}

\newcommand{\mfp}{\mathfrak{p}}

\newcommand{\mfq}{\mathfrak{q}}

\newcommand{\mfB}{\mathfrak{B}}
\newcommand{\ovp}{\overline{\mfp}}
\newcommand{\ovQQ}{\Qbar}

\newcommand{\symb}[2]{\left[#1,\, #2\right]}
\newcommand{\isoarrow}{\xrightarrow{\,\,\sim\,\,}}
\newcommand{\res}{\textup{res}}

\newcommand{\msL}{\mathscr{L}}

\newcommand{\Ssm}{S_{\textup{sm}}}
\newcommand{\Smed}{S_{\textup{med}}}
\newcommand{\Slg}{S_{\textup{lg}}}
\newcommand{\ovq}{\overline{\mfq}}

\newcommand{\mcA}{\mathcal{A}}
\newcommand{\mcX}{\mathcal{X}}
\newcommand{\mcM}{\mathcal{M}}
\newcommand{\mcO}{\mathcal{O}}

\newcommand{\mcB}{\mathcal{B}}

\newcommand{\mbb}{\mathbf{b}}

\newcommand{\mcT}{\mathcal{T}}
\newcommand{\mcTbnd}{\mathcal{T}_{\textup{bnd}}}
\title{\sffamily \bfseries \textbf{Tamagawa ratios and unbounded Selmer moments}}

\author[1]{Peter Koymans\thanks{Mathematisch Instituut, Universiteit Utrecht, Postbus 80.010, 3508 TA Utrecht, The Netherlands, p.h.koymans@uu.nl}}
\author[2]{Alexander Smith\thanks{Department of Mathematics, 2033 Sheridan Road, Evanston, IL 60208, asmith@northwestern.edu}}
\affil[1]{Utrecht University}
\affil[2]{Northwestern University}

\begin{document}
\maketitle

\begin{abstract}
We develop a framework to predict whether a family of Selmer groups has average size that is bounded or unbounded. Applying this framework to certain geometric families of abelian varieties over $\QQ$, we give a conjectural characterization of which such families have $\ell$-Selmer groups of unbounded average size for a given prime $\ell$. In the case that the $\ell$-torsion Galois module is constant across the family, we show that our characterization is correct.

The key tool of our technique is the Greenberg--Wiles' formula, which expresses the ratio of the sizes of a Selmer group and the corresponding dual Selmer group as a product of local factors. This formula gives a purely local lower bound for the size of a Selmer group that we conjecture is close to sharp most of the time.
\end{abstract}

\section{Introduction}
\subsection{Selmer groups in families of elliptic curves}
Given polynomials $a_1, a_2, a_3, a_4, a_6$ in the polynomial ring $\Z[u_1, \dots, u_n]$, we may consider the equation
\begin{equation}
    \label{eq:generic_ell_fib}
\mcA: y^2 + a_1xy + a_3y = x^3 + a_2x^2 + a_4x + a_6.
\end{equation}
We will assume this equation defines an elliptic curve over $\QQ(u_1, \dots, u_n)$. Taking the discriminant of this elliptic curve defines a polynomial $\Delta$ in $\QQ[u_1, \dots, u_n]$ such that, for any tuple of integers $\mathbf{b} = (b_1, \dots, b_n)$, if $\Delta(b_1, \dots, b_n)$ is nonzero, then
\[\mcA_{\mbb} : y^2 + a_1(\mbb)xy + a_3(\mbb)y = x^3 + a_2(\mbb)x^2 + a_4(\mbb)x + a_6(\mbb)\]
defines an elliptic curve over $\QQ$. Fixing positive integers $\gamma_1, \dots, \gamma_n$, we define a height function
\[h(b_1, \dots, b_n) = \max\big(|b_1|^{\gamma_1}, \dots, |b_n|^{\gamma_n}\big)\]
and use this height to define a family
\[\mcA_{\le H }= \big\{\mcA_{\mbb}\,:\,\, \mbb \in \Z^n \,\text{ with }\,\Delta(\mbb) \ne 0 \,\text{ and }\, h(\mbb) \le H\big\}\]
for any positive $H$. This defines a geometric family of elliptic curves.

A major goal of arithmetic statistics is to understand how Selmer groups are distributed in geometric families of elliptic curves. For the family
\[\mcA: y^2 = x^3 + u_1x + u_2\]
with height $h(b_1, b_2) = \max(|b_1|^3, |b_2|^2)$, this goal is encapsulated by the Poonen--Rains conjectures \cite{PR} and their generalizations \cite{BKLPR}. For a given prime $\ell$, Poonen and Rains  give a conjectural distribution for the groups $\Sel_{\ell}\, \mcA_{(b_1, b_2)}$. They also conjecture values for the moments of these Selmer groups; for any nonnegative integer $\kappa$, they conjecture
\begin{equation}
\label{eq:PR}
\lim_{H \to \infty} \frac{1}{\#\mcA_{\le H}} \sum_{E \in \mcA_{\le H}} \left(\#\Sel_{\ell}\, E\right)^{\kappa} = \prod_{i = 1}^{\kappa} (\ell^i + 1).
\end{equation}
Up to slight modifications to the height function $h$, this identity is known to hold for $\ell = 2, 3, 5$ and $\kappa = 1$ from work of Bhargava and Shankar \cite{BS1, BS2, BS4}, and the right hand side is known to be an upper bound for the left hand side when $\ell = 2$ and $\kappa = 2$ by work of Bhargava--Shankar--Swaminathan \cite{BSS}. Although not strictly falling under the purview of \eqref{eq:PR} as stated above, there is also work of Bhargava and Shankar \cite{BS3} on the average size of the $4$-Selmer group.

Given any rational numbers $a_4, a_6$ such that $x^3 + a_4x + a_6$ is irreducible in $\Q[x]$, the Poonen--Rains heuristic is known to describe the $2$-Selmer groups in the quadratic twist family of curves
\[\mcA: y^2 = x^3 +u_1^2a_4x + u_1^3a_6;\]
in particular, \eqref{eq:PR} is known to hold for any nonnegative integer $\kappa$ and $\ell = 2$ \cite{Smi22b}. With a modification to account for the impact of $2$-torsion on the $2$-Selmer rank, these conjectures are also known to describe the $2$-Selmer groups in quadratic twist families of elliptic curves with full rational $2$-torsion and no cyclic $4$-isogeny \cite{HB, Kane}.

But, in other quadratic twist families, the behavior is very different. This was first observed by Xiong and Zaharescu, who noted that the average $2$-Selmer rank in the family $\mcA: y^2 = x^3 + u_1^2x$ is unbounded in \cite{XZCongruent}. This unboundedness result was later extended to other quadratic twist families \cite{Feng Xiong,  Klagsbrun, Xiong, XZ}. The work of Xiong and Zaharescu appeared a couple years after \cite{YuII}, where Yu showed that, for the family
\[
\mcA: y^2  = x(x -1 )(x-u_1),
\]
we have the unboundedness result
\[
\frac{1}{\# \mcA_{\le H}}\sum_{E \in \mcA_{\le H}} \# \Sel_2 \, E \asymp \sqrt{\log H}.
\]
See Example \ref{ex:YuII} for more on Yu's work. Since this work, unbounded Selmer group moments have been found in many natural geometric families \cite{ABS, Chan, Chan Hanselman Li, Chan Verzobio, KLOTama, KLO, Phillips}.

\subsection{Our conjecture and results}
All these unboundedness results can be explained by \emph{Tamagawa ratios}, a term introduced by Klagsbrun in his work on $2$-Selmer groups in quadratic twist families \cite{Klagsbrun} which has roots in work of Cassels \cite{Cassels}. Given a rational degree $\ell$-isogeny $\varphi: E \to E_0$ of elliptic curves over $\QQ$ with dual isogeny $\varphi_0: E_0 \to E$, we have an approximate relation
\begin{equation}
    \label{eq:Tama_approx}
\frac{\# \Sel_{\varphi} \, E}{\# \Sel_{\varphi_0}\, E_0} \approx \prod_p \frac{c_p(E_0)}{c_p(E)}
\end{equation}
of the sizes of the isogeny Selmer groups, where $c_p(E)$ denotes the Tamagawa number of $E$ at $p$. The product of the ratio of Tamagawa numbers then yields a purely local lower bound for the size of the isogeny Selmer group $\Sel_{\varphi}\, E$, which in turn yields a  purely local lower bound for the Selmer group $\Sel_{\ell}\, E$. 

The approximate relationship \eqref{eq:Tama_approx} is an instance of the Greenberg--Wiles' formula, which applies to more general families of Galois modules decorated with local conditions; see Definition \ref{defn:decorated_modules} for details. The central slogan of this paper, encapsulated in Heuristic \ref{heur:main} and Conjecture \ref{conj:AV}, is that large Selmer group moments in ``natural'' families of decorated Galois modules over $\QQ$ can always be explained by the Greenberg--Wiles' formula. For geometric families of elliptic curves, this conjecture takes the following form: 
\begin{conjecture}
\label{conj:ell}
Take $\mcA$ to be the family of elliptic curves \eqref{eq:generic_ell_fib}, and choose a prime $\ell$. Take $K = \QQ(u_1, \dots, u_n)$, so \eqref{eq:generic_ell_fib} defines an elliptic curve $\mcA_K$ over $K$. For each $K$-subgroup scheme $T$ of $\mcA_K[\ell]$, define the function $\beta(T, \text{--}) : \R_{\ge 0} \to \R$ as in Definition \ref{defn:Tate_alg}. Take $\beta: \R_{\ge 0} \to \R_{\ge 0}$ to be the function defined by
\[\beta(\kappa) = \max_{T} \beta(T, \kappa).\]
Then, for all $\kappa \ge 0$,
\[\frac{1}{\#\mcA_{\le H}}\sum_{E \in \mcA_{\le H}} (\# \Sel_{\ell}\, E)^{\kappa} \asymp (\log H)^{\beta(\kappa)}.\]
\end{conjecture}

Methods from the geometry of numbers give many interesting cases of this conjecture. Besides the papers mentioned above, we also mention work of Bhargava and Ho for Selmer groups in families with certain marked rational points \cite{BH} and work of Bhargava, Elkies, and Shnidman in the family $y^2 = x^3 +  u_1$ \cite{BES}. The families considered in these papers have $\beta(\kappa) = 0$ for all $\kappa$.

This conjecture is also known for $\ell = 2$ in a number of quadratic twist families. For this paper, the two most salient prior results are the following:

\begin{example}[{Koymans--Pagano--Sofos \cite{KPS}}]
\label{ex:KPS}
Given a nonzero polynomial $P \in \Q[u_1, \dots, u_n]$ and distinct rational numbers $r_1, r_2, r_3$,  Conjecture \ref{conj:ell} holds for the $2$-Selmer groups of the family 
\[\mcA/\Q(u_1, \dots, u_n) : y^2 = (x - r_1P)(x- r_2P)(x-r_3 P)\]
with height $h(b_1, \dots, b_n) = \max |b_i|$.

Like work of Kane \cite{Kane}, Swinnerton-Dyer \cite{SD}, and Heath-Brown \cite{HB} before it, this result takes advantage of the fact that the $2$-Selmer group of a quadratic twist of an elliptic curve with full rational $2$-torsion can be calculated in terms of a matrix of Legendre symbols of the primes dividing the twist. The novelty of this result is its use of sieve methods, which are used to control most but not all of this matrix; specifically, the entries invoking the very largest primes dividing the twist are left uncontrolled. The use of sieve methods precludes determining the exact size of the Selmer moments, but is strong enough to prove Conjecture \ref{conj:ell} for this family.
\end{example}

Another general result for quadratic twist families is the following:

\begin{example}[{Smith \cite{Smi22b}}]
\label{ex:Smith}
Given integers $a_4, a_6$ satisfying $4a_4^3 + 27a_6^2 \ne 0$, Conjecture \ref{conj:ell} holds for the $2$-Selmer groups of the family 
\[
\mcA: y^2 = x^3 + u_1^2a_4 x + u_1^3 a_6.
\]
Work on this example starts by generalizing the matrix-of-Legendre-symbols approach of \cite{Kane, SD} to handle elliptic curves of other torsion types. 
\end{example}

By combining the sieve techniques of Example \ref{ex:KPS} with the generalized Legendre symbol matrices from Example \ref{ex:Smith}, we are able to prove  Conjecture \ref{conj:ell} in the case that $\mcA_K[\ell]$ is constant throughout the family.

\begin{theorem}
\label{tGeom1}
Take $\mcA$ to be the family of elliptic curves \eqref{eq:generic_ell_fib}, and choose a prime $\ell$. Take $K = \QQ(u_1, \dots, u_n)$. Suppose there is a number field $L$ such that every point in $\mcA_K[\ell]$ is rational over $L(u_1, \dots, u_n)$.

Then Conjecture \ref{conj:ell} holds for the family $\mcA$ and the prime $\ell$.
\end{theorem}

In this theorem, as $\mbb$ changes, the Galois module $\mcA_{\mbb}[\ell]$ is fixed, but the local conditions that decorate it change. The most general result of this paper, Theorem \ref{thm:main}, proves the analogue of Conjecture \ref{conj:ell} in families of decorated Galois modules where the underlying module over $\Q$ is fixed. To make this work, we assume that the local conditions are \emph{effectively equidistributed}; see Definition \ref{defn:effective-constant} for details.

To closely approximate the distribution of the $\ell$-Selmer groups in $\mcA_{\le H}$, we would need to have a handle on the local conditions of an elliptic curve drawn from this family at all of its places of bad reduction. This is possible in some cases, as in Example \ref{ex:Smith}, but is usually infeasible. Instead, we follow the approach of Example \ref{ex:KPS} and ignore the primes of bad reduction above a certain threshold relative to $H$. As $E$ varies up to height $H$, the product of primes of bad reduction below this threshold has a distribution approximately given by a calculable multiplicative function. These results are proved by employing a general form of a sieve originally considered by Erd\H{o}s \cite{Erdos} and later considered by Shiu \cite{Shiu}, Wolke \cite{Wolke} and Nair--Tenenbuam \cite{NT}.

Our next step is to consider the set of objects in this family whose product of bad primes below the threshold lies in a set of the form
\[\{p_1 \dots p_k\,:\,\, p_i \in X_i \text{ for all } i\},\]
where the $X_i$ are judiciously chosen sets of primes. This step is known as \emph{gridding}, and is a modification of a similar step in \cite{Smi22b}.

With the grids set up, the final step in the proof of Theorem \ref{thm:main} is to estimate the character sums that give the average size of the Selmer groups over these grids. This again uses techniques from \cite{Smi22b}, but with complications due to the more general setup of Theorem \ref{thm:main}.

Once the proof of Theorem \ref{thm:main} is finished, the main work that remains is to show that it implies our results for elliptic curves and abelian varieties, including Theorem \ref{tGeom1} and its generalization Theorem \ref{thm:AV}. This means that we need to show that the families of decorated modules in these theorems have local conditions that are effectively equidistributed. This step, which is trivial in the quadratic twist case, is much more complicated in full generality. We end up relying on a quantifier elimination result of Denef for $p$-adic fields \cite{Denef, Denef2} and its uniform generalization due to Pas \cite{PasCrelle}. With this input from model theory, we are able to translate our questions about local conditions into solvable problems in algebraic geometry.

\subsection{Other applications of Theorem \ref{thm:main}}
The setup of Theorem \ref{thm:main} is very general. We go through a number of applications of this result in Section \ref{sExa}; we highlight some of these here, starting with a generalization of Examples \ref{ex:KPS} and \ref{ex:Smith}.

\begin{theorem}[Example \ref{ex:abelian_twist}]
\label{thm:quadtwist}
Choose a number field $F$ and an abelian variety $A/F$. Given $d$ in $F^{\times}$, take $A^d$ to be the quadratic twist of $A$ corresponding to $F(\sqrt{d})/F$.

Choose a nonconstant polynomial $P$ in $F[u_1, \dots, u_n]$. Given a tuple $\mbb = (b_1, \dots, b_n)$ of integers, define $h(\mbb) = \max_i |b_i|$. Then there is $C > 0$ such that, for all $\kappa > 0$, we have
\begin{equation}
\label{eq:quadtwist_moments}
\frac{1}{(2H)^n} \sum_{\substack{\mbb \in \Z^n \\ h(\mbb) \le H \textup{ and } P(\mbb) \ne 0}} \exp\left(\kappa \cdot \textup{rank}\big(A^{P(\mbb)}/F\big) \right) \le \exp \exp(C \kappa)
\end{equation}
for all sufficiently large $H$ in terms of $\kappa$.

In particular, there is some $c > 0$ so that, for all $r \ge 1$,
\begin{equation}
    \label{eq:quadtwist_ranks}
\limsup_{H \to \infty} \frac{\#\big\{\mbb\in \Z^n \,:\,\, h(\mbb) \le H,\,\, P(\mbb) \ne 0, \, \textup{ and } \, \textup{rank}\left(A^{P(\mbb)}\right) \ge r\big\}}{(2H)^n} \le r^{-cr}.
\end{equation}
\end{theorem}

In this theorem, as in Theorem \ref{tGeom1}, the relevant local conditions are shown to be effectively equidistributed using geometric methods. But this is not a requirement. For a $3$-isogeny $\phi: A \rightarrow A'$ and for $t \in \Q^\times$, denote by $\phi_t: A^t \rightarrow (A')^t$ the $3$-isogeny attached to the quadratic twist. We prove the following: 

\begin{theorem}
\label{t3Selmer}
Let $\phi: A \rightarrow A'$ be a $3$-isogeny of abelian varieties over $\Q$ and assume that $A[2]$ is irreducible as a $G_\Q$-module. Then there is $C > 1$ such that for all $H \geq 10$
$$
\sum_{\substack{|t| \leq H \\ t \textup{ sqf.}}} \# \Sel_{2\phi_t} \, A^t \leq CH.
$$
\end{theorem}

To prove this, we apply Theorem \ref{thm:main} to decorated modules indexed by tuples of the form $(t, \psi)$, where $\psi$ is a $\phi_t$-Selmer element of $A^t$. In this case, the necessary effective equidistribution results are a consequence of work in the geometry of numbers \cite{BKLOS, BTT}. This theorem is proved in Section \ref{s3Selmer}.

An obvious example of an application of our framework outside the context of abelian varieties over number fields is to class groups. In this context, an analogous result to Theorem \ref{thm:quadtwist} is the following:

\begin{theorem}
\label{tClass}
Let $k \in \Z_{\geq 2}$ and let $P(u) \in \Z[u]$ be squarefree. Then there is $C > 0$ such that for all real numbers $H \geq 10$
$$
\sum_{\substack{1 \leq b \leq H \\ P(b) \neq 0}} \# \mathrm{Cl}(\Q(\sqrt{P(b)}))[2^k] \leq CH (\log H)^r,
$$
where $r$ denotes the number of irreducible factors of $P(u)$.
\end{theorem}

We prove this in Section \ref{sClass}. 

The final example we highlight is one of many families of elliptic curves considered in Section \ref{sGeom}. Compared to the others, it is special because of its apparent ``paradox''; the distribution of $3$-Selmer ranks shows superexponential decay, but the expected size of the $3$-Selmer group is unbounded.

\begin{theorem}[Example \ref{ex:full_3tor}]
\label{tGeom2}
Consider the family
\[\mcA: y^2 + 3u_1xy + (u_1^3 + u_2^3)y = x^3.\]
This is the family of rational elliptic curves with $3$-torsion isomorphic to $\Z/3\Z \oplus \mu_3$.

Then there is some $c > 0$ so that, for all $r \ge 1$,
\[\limsup_{H \to \infty} \frac{\# \{E \in \mcA_{\le H}\,:\,\, \dim \, \Sel_3\, E \ge r\}}{\# \mcA_{\le H}} \le r^{-cr}.\]
At the same time,
\[\frac{1}{\# \mcA_{\le H}} \sum_{E \in \mcA_{\le H}} \# \Sel_3\, E \asymp \log H.\]
\end{theorem}

Using our framework, the odd behavior of $3$-Selmer groups seen in this theorem is a consequence of odd behavior of Tamagawa ratios.

\subsection{Overview of the paper}
In Section \ref{sec:heuristic}, we review the Greenberg--Wiles' formula and use it to give the general heuristic underpinning Conjecture \ref{conj:ell}, which is Heuristic \ref{heur:main}. This heuristic concerns the behavior of Selmer groups of families of decorated modules over $\Q$. In this section, we also state our main result, Theorem \ref{thm:main}, which shows that the heuristic is correct in constant module families with effectively equidistributed local conditions.

Over Sections \ref{sSieve}, \ref{sGrid}, and \ref{sChar}, we prove Theorem \ref{thm:main}. This starts in Section \ref{sSieve} with the development of a suitable sieve, both for an upper bound (Theorem \ref{tSieve}) and a lower bound (Theorem \ref{tSieve2}). In Section \ref{sGrid}, we adapt the theory of grids \cite[Section 8]{Smi22b} to this setup. This reduces the Selmer group moments we are bounding to a character sum that we estimate in Section \ref{sChar}. This approach is based on \cite[Sections 5-7]{Smi22b}, but requires several novel arguments as our character sum involves a complicated multiplicative weight coming from the sieve. With the character sums estimated, we then prove Theorem \ref{thm:main} in Section \ref{ssec:proof_main}. 

In Section \ref{sGeometry}, we state a generalization of Theorem \ref{tGeom1} for abelian varieties, Theorem \ref{thm:AV}, and show that it is a consequence of Theorem \ref{thm:main}. This involves proving that the associated family of decorated modules has effectively equidistributed local conditions.

To this point, our main results have shown that sums of Selmer group sizes are of a similar magnitude to the corresponding sums of Tamagawa ratios. In Section \ref{ssec:Tamagawa_AV}, we give tools for estimating sums of Tamagawa ratios, in the process giving a closed form for our estimate of the sums of Selmer group sizes. These tools are particularly concrete in the case where we are looking at a family of elliptic curves, and we prove Theorem \ref{tGeom1} at the end of Section \ref{sGeometry}.

Finally, in Section \ref{sExa}, we use the theory developed throughout the paper to study a plethora of examples, including 8 families of elliptic curves. This section has the proofs of Theorem \ref{tGeom2} and Theorem \ref{thm:quadtwist} as consequences of Theorem \ref{thm:AV}. Finally, we derive Theorem \ref{t3Selmer} and Theorem \ref{tClass} directly from Theorem \ref{thm:main} at the end of this section.

\subsection*{Acknowledgments}
Conjecture \ref{conj:ell} is an attempt to answer a question posed by Manjul Bhargava to the second author in September 2024. We are grateful to him for this question. We would also like to thank Ken Willyard for useful conversations.

The first author gratefully acknowledges the support of the Dutch Research Council (NWO) through the Veni grant ``New methods in arithmetic statistics''. This research was partially conducted during the period the second author served as a Clay Research Fellow.
\section{A general heuristic for Selmer moments}
\label{sec:heuristic}
In this section we explain the heuristic underpinning our conjecture and state our main result in its most general form (see Theorem \ref{thm:main}).

\subsection{Tamagawa ratios}
We begin with the general definition of a Selmer group and Tamagawa ratio.

\begin{mydef}
\label{defn:decorated_modules}
Take $F$ to be a global field, and take $M$ to be a finite $G_F$-module whose order is indivisible by the characteristic of $F$. For each place $v$ of $F$, choose a subgroup
\[
\msL_v \subseteq H^1(G_v, M).
\]
We assume that $\msL_v$ is given by the unramified cohomology group $H^1_{\text{ur}}(G_v, M)$ at all but finitely many places. We refer to the tuple $(M, (\msL_v)_v)$ as a finite $G_F$-module decorated with local conditions, and we usually just refer to this object by $M$.

Given such a module decorated with local conditions, we then define the Selmer group of $M$ by
\[
\Sel\, M :=  \ker\left(H^1(G_F, M) \to \prod_v H^1(G_v, M)/\msL_v\right).
\]
This group is always finite by our assumption that almost every $\msL_v$ is the unramified set of local conditions.

We define the dual $M^{\vee}$ of $M$ as the tuple 
\[
M^{\vee} := \left(M^\ast(1), \,(\msL_v^{\perp})_{v}\right).
\]
Here, $M^\ast(1)$ is the Tate twist of the Pontryagin dual to $M$, and $\msL_v^{\perp}$ is the orthogonal complement to $\msL_v$ under the usual local Tate pairing between $H^1(G_v, M^\ast(1))$ and $H^1(G_v, M)$.

We define the \emph{Tamagawa ratio} for the decorated module $(M, (\msL_v)_v)$ by
\[
\mcT(M): = \frac{\# H^0(G_F, M)}{\# H^0(G_F, M^{\vee})} \cdot \prod_v \mcT_v(M) \quad\text{where}\quad \mcT_v(M):= \frac{\# \msL_v}{\# H^0(G_v, M)}.
\]
The \emph{Greenberg--Wiles' formula} \cite[Theorem 8.7.9]{NSW} then states
\begin{equation}
    \label{eq:Wiles}
\frac{\# \Sel\, M}{\# \Sel\, M^{\vee}} = \mcT(M).
\end{equation}

\begin{remark}
In the case that $M$ is the kernel of some isogeny of elliptic curves over a number field with the usual isogeny Selmer group local conditions, the local terms of the Tamagawa ratio may be written as ratios of Tamagawa numbers, giving the name. This notation is essentially due to Cassels, who also proved the Greenberg--Wiles' formula for $M$ in this special case  \cite[Theorem 1.1]{Cassels}.

Wiles proved the formula in the case $F = \Q$ \cite[Proposition 1.6]{Wiles} by applying an argument of Greenberg \cite[eq.~(22)]{Greenberg}. No modification is needed to Wiles' proof to extend it to other global fields.
\end{remark}

Now, given any $G_F$-submodule $T$ of $M$, we may endow $T$ with the local conditions $\left(\iota^{-1}(\msL_v)\right)_v$, where $\iota$ is the inclusion of $T$ in $M$. We then have an exact sequence
\[
0 \to H^0(G_F, T) \to H^0(G_F, M) \to H^0(G_F, M/T) \to \Sel\, T \to \Sel\, M,
\]
so
\begin{align*}
\# \Sel\, M  &\,\ge\, \# \Sel\, T\cdot \frac{   \# H^0(G_F, M)}{\# H^0(G_F, T) \cdot \# H^0(G_F, M/T)} \\
&\,\ge\,  \mcT(T) \cdot \frac{ \# H^0(G_F, M)}{\# H^0(G_F, T) \cdot \# H^0(G_F, M/T)}\\
&\,=\,  \frac{ \# H^0(G_F, M)}{\# H^0(G_F, T^{\vee}) \cdot \# H^0(G_F, M/T)} \cdot \prod_v \frac{\# \iota^{-1}(\msL_v)}{\# H^0(G_v, T)}.
\end{align*}
Calling this final expression $\mcT(M, T)$, we define the \emph{Tamagawa lower bound} for $M$ by
\[\mcTbnd(M) := \max_{T \subseteq M} \mcT(M, T),\]
where the maximum is taken over all $G_F$-submodules of $M$. 
\end{mydef}
We always have
\begin{equation}
\label{eq:yes_its_a_lower_bound}
\# \Sel\, M \ge \mcTbnd(M).
\end{equation}
Our starting expectation is that, if $M$ varies in a ``natural'' family of $G_{\Q}$-modules, then \eqref{eq:yes_its_a_lower_bound} is not too far from being sharp.

\begin{heuristic}
\label{heur:main}
Choose a ``natural'' infinite family $M_1, M_2, \dots$ of finite $G_{\Q}$-modules decorated with local conditions. We assume that the $M_i$ all have the same cardinality.

Then, for any $\kappa > 0$, we have
$$
\limsup_{n \to \infty}\, \frac{1}{n} \cdot \sum_{i = 1}^n \left(\frac{\# \Sel\, M_i}{\mcTbnd(M_i)}\right)^{\kappa} < \infty
$$
and
$$
\frac{1}{n}\cdot \sum_{i = 1}^n \left(\# \Sel\, M_i\right)^{\kappa} \,\asymp\, \frac{1}{n} \cdot \sum_{i = 1}^n \mcTbnd(M_i)^{\kappa}.
$$
In particular, the terms $(\# \Sel \, M_i)^{\kappa}$ have unbounded average if and only if the terms $\mcTbnd(M_i)^{\kappa}$ have unbounded average.
\end{heuristic}

Outside the cases of this heuristic needed to formulate Conjectures \ref{conj:ell} and \ref{conj:AV}, we will not make any guess for what property should take the place of ``natural'' in this heuristic. It is straightforward to find infinite families where the heuristic is not true.

It may seem odd that we have restricted our attention to modules over $\QQ$. But this is not really a restriction. Given a number field $F$ and a sequence of decorated $G_F$-modules $M_1, M_2, \dots$, we can decorate the induced $G_{\QQ}$-modules $\text{Ind}_{F/\QQ} \,M_1, \text{Ind}_{F/\QQ}\, M_2, \dots$ with local conditions so that we have a natural isomorphism
\[
\Sel_{\QQ}\, \text{Ind}_{F/\QQ} M_i \cong \Sel_F \,M_i
\]
for all $i$; see \cite[Definition 4.2]{MS}.

Furthermore, the modules $\text{Ind}_{F/\QQ} M_i$ can have submodules not corresponding to submodules of $M_i$, and these can affect the Tamagawa lower bound. This can give some natural families over number fields besides $\QQ$ where the heuristic does not hold.

\begin{example}
Take $F/\QQ$ to be an imaginary quadratic field. Given a squarefree product $d$ of rational primes that are inert in $F/\QQ$, take $M_d$ to be the module $(\FFF_2, (\msL_p)_p)$ over $F$, where $\msL_p$ is generated by the image of the quadratic character $\chi_d: G_{F_p} \to \FFF_2$ for $p$ dividing $d$, and where $\msL_p$ is otherwise the unramified local conditions.

Then, over $F$, $\mcTbnd(M_d)$ equals $1$. But $\text{Ind}_{F/\QQ} \,\FFF_2$ contains the submodule $\FFF_2$, and the Tamagawa ratio associated to this decorated module over $\QQ$ is no smaller than $2^{\omega(d) - 1}$. This means that the $\Sel\, M_d$ have unbounded average size compared with their Tamagawa bounds over $F$.
\end{example}

\subsection{Our results for constant modules}
In the special case where the $G_{\QQ}$-structure of our decorated module does not change, and assuming some additional technical hypotheses, we can prove that our heuristic is correct.

\begin{mydef}
Take $M$ to be a finite $G_{\QQ}$-module. Fix a finite set of rational places $\Vplac_0$ containing $\infty$, the places dividing $|M|$, and the places such that $I_p$ has nontrivial action on $M$. Given a rational place $v$, a \emph{local conditions quasi-subgroup} for $M$ will either be the empty set (which will take the role of a placeholder symbol throughout the paper) or a subgroup of $H^1(G_v, M)$. We call a tuple $(M, (\msL_v)_v)$ a \emph{quasi-decorated module} if $\msL_v$ is a local conditions quasi-subgroup for all $v$, and if $\msL_v$ is $H^1_{\text{ur}}(G_v, M)$ for all but finitely many $v$. A \emph{realization} of $(\msL_v)_v$ will be a tuple of local conditions subgroups $(\msL'_v)_v$ such that $\msL'_v = \msL_v$ whenever $\msL_v$ is not the empty set.

We then define
\[
\Sel\left(M, (\msL_v)_v\right) = \Sel\left(M, (\msL'_v)_v\right) \quad \text{with} \quad \msL'_v = 
\begin{cases}
H^1(G_v, M) &\text{ if } \msL_v = \emptyset \text{ or } v \in \Vplac_0 \\ 
\msL_v &\text{ otherwise}
\end{cases}
\]
and
\[
\mcTbnd(M, (\msL_v)_v) = \mcTbnd(M, (\msL'_v)_v) \quad \text{with} \quad \msL'_v = 
\begin{cases}
0 &\text{ if } \msL_v = \emptyset \text{ or } v \in \Vplac_0 \\ 
\msL_v &\text{ otherwise.}
\end{cases}
\]
This is defined so that
\begin{equation}
\label{eq:quasi_realized}
\frac{\#\Sel\, (M, (\msL'_v)_v) }{\mcTbnd\left(M, (\msL'_v)_v\right)} \le \frac{\#\Sel\, (M, (\msL_v)_v) }{\mcTbnd\left(M, (\msL_v)_v\right)}
\end{equation}
for any realization $(\msL'_v)_v$ of $(\msL_v)_v$.
\end{mydef}

\begin{mydef}
\label{defn:effective-constant}
Fix a finite $G_{\Q}$-module $M$. Choose a set $X$, and for every $x \in X$, choose a quasi-decorated module 
\[
M_x = \left(M, (\msL_{xv})_v\right).
\]
Choose a height function $h: X \to \R_{\ge 0}$ such that for all $x \in X$
$$
h(x) \geq \prod_{p \in \Vplac_0 - \{\infty\}} p \cdot \prod_{\substack{p \not \in \Vplac_0 \\ \msL_{xp} \neq H^1_{\text{ur}}(G_p, M)}} p.
$$
Given $H > 1$, take $X_H$ to be the subset of $x \in X$ with $h(x) \leq H$. We assume this set is finite for all $H$.

We then say that $\{M_x\,:\,\, x \in X\}$ is a \emph{constant-module family with effectively equidistributed local conditions} if the following conditions are satisfied:
\begin{enumerate}
\item[(1)] \emph{(Effective equidistribution}, cf.~\cite{El-Baz Loughran Sofos}\emph{)} Given a rational prime $p$, take $P_p$ to be the set of subgroups of $H^1(G_p, M)$ together with the empty set. There is $c > 0$ and, for each $p$ outside $\Vplac_0$, there is a function
\[
\mu_p: P_p \to \R_{\ge 0}
\]
such that, for any $H > 10$, for any squarefree product $Q$ of rational primes outside $\Vplac_0$ satisfying $Q \le H^c$, and for any  $(\msL_p)_{p | Q}$ in $\prod_{p | Q} P_p$, we have
\[
\left|\frac{\# \{x \in X_H\,:\,\, \msL_{xp} = \msL_p \,\text{ for all }\, p \mid Q\}}{\# X_H}  - \prod_{p \mid Q} \mu_p(\msL_p) \right| \le H^{-c}.
\]
We further assume
$$
\sum_p \mu_p(\emptyset) < \infty.
$$

\item[(2)] \emph{(Sparsity of bad primes)} There is $C > 0$ such that, for any prime $p$ outside $\Vplac_0$ and any $\msL_p$ in $P_p$ besides $H^1_{\text{ur}}(G_p, M)$, we have
\[
\mu_p(\msL_p) \le Cp^{-1}.
\]

\item[(3)] \emph{(No moderately rare local conditions)}
Choose a weak equivalence class $\mathcal{L}$ of local conditions, in the sense of Definition \ref{defn:equiv_local} below; this is a collection of pairs of the form $(p, \msL_p)$, where $p$ is a prime and $\msL_p$ is a subgroup of $H^1(G_p, M)$.

Then either
\[
\sum_{(p, \msL_p) \in \mathcal{L}} \mu_p(\msL_p) < \infty
\]
or
\[ 
\sum_{\substack{(p, \msL_p) \in \mathcal{L}\\ p \le H}} \mu_p(\msL_p) \ge -1+  c \log \log H
\]
for some fixed $c > 0$ and all $H > 10$.
\end{enumerate}
For each prime $p$ outside $\Vplac_0$, also fix a function $\tilde{g}_p: P_p \to \R_{\ge 0}$ such that $\tilde{g}_p$ is $1$ on the unramified local conditions. We will assume there is some $C > 1$ such that $\tilde{g}_p(\msL_p)$ is bounded between $C^{-1}$ and $C$ for all $p$ and all $\msL_p$ in $P_p$. We will then take
\[
\tilde{g}((\msL_p)_p) = \prod_{p \not \in \Vplac_0} \tilde{g}_p(\msL_p).
\]
\end{mydef}

\begin{theorem}
\label{thm:main}
Take $\{M_x\,:\,\, x \in X\}$ to be a constant-module family with effectively equidistributed local conditions, and choose $\tilde{g}: \prod' P_p \to \R_{> 0}$ as above. Then there is $C > 0$ depending on the family and on $\tilde{g}$ such that, for all $\kappa \ge 0$,  $H > 100$, and $\nu \in [0, \kappa]$
\[
\sum_{x \in X_H} \left(\frac{\# \Sel\, M_x}{\mcTbnd(M_x)}\right)^{\kappa} \tilde{g}(\msL_x)^{\nu} \,\le\, \exp \exp (C\kappa) \sum_{x \in X_H} \tilde{g}(\msL_x)^{\nu}.
\]
\end{theorem}

\begin{remark}
With a little more work, our methods could be used to give effective forms for the constant $C$. To simplify the calculations, we have opted for a largely ineffective form for this estimate. However, we will keep track of the impact of $\kappa$ on our estimates.

We note that the bounds above, which are doubly exponential in $\kappa$, grow more quickly than $e^{C\kappa^2}$, which is the general expectation for how these Selmer moments grow with $\kappa$; see e.g.~\cite{Smi22b}. This doubly exponential growth seems to be an unavoidable consequence of our sieve theoretic methods.
\end{remark}

\begin{remark}
\label{rmk:two_mult_weights}
By taking $\tilde{g} = 1$ in this theorem, we get
\[\frac{1}{\#X_H} \sum_{x \in X_H} \left(\frac{\# \Sel\, M_x}{ \mcTbnd(M_x)}\right)^{\kappa}  \,\le\, \exp \exp (C\kappa).\]
For any fixed submodule $T$ of $M$, we may instead take $\tilde{g}(\msL_x)$ to be proportional to the Tamagawa ratio $\mathcal{T}(M_x, T)$. Applying the theorem to every such $T$ and then summing over all submodules $T \subseteq M$ gives
\[\sum_{x \in X_H} \left(\# \Sel\, M_x \right)^{\kappa}  \,\le\, \exp \exp (C\kappa) \sum_{x \in X_H} \mcTbnd(M_x)^{\kappa}.\]
This second application is the reason we have included the extra weight $\tilde{g}$ in Theorem \ref{thm:main}.
\end{remark}

This definition of effectively equidistributed local conditions relies on a notion of weak equivalence of local conditions, which we now define.

\begin{mydef}
\label{defn:equiv_local}
Take $M$ and $\Vplac_0$ as above. Take $m$ to be the exponent of $M$. For a prime $p$ outside $\Vplac_0$, the group $mI_p$ is a normal subgroup of the inertia group $I_p$ and decomposition group $G_p$. The group $G_p/mI_p$ then takes the form $\widehat{\Z} \ltimes \Z/m\Z$, where the action defining the semidirect product depends on $p$ mod $m$. Any subgroup $\msL_p$ of $H^1(G_p, M)$ is the inflation of some subgroup $\msL'_p$ of $H^1(G_p/mI_p, M)$.

Given another prime $q$ outside $\Vplac_0$ and a subgroup $\msL_q$ of $H^1(G_q, M)$ corresponding to the inflation of the subgroup $\msL'_q$ of $H^1(G_q/mI_q, M)$, we say $(p, \msL_p)$ and $(q, \msL_q)$ are \emph{weakly equivalent} if there is a continuous group isomorphism $\iota: G_p/mI_p \isoarrow G_q/mI_q$ and a choice of $\sigma \in G_{\QQ}$ such that
\[
\sigma(\tau m) = \iota(\tau) \sigma(m)\quad\text{for all } \tau \in G_p/mI_p \text{ and } m \in M
\]
and such that the corresponding group change operation
\[
(\iota, \sigma): H^1(G_q/mI_q, M) \isoarrow H^1(G_p/mI_p, M)
\]
identifies $\msL'_q$ with $\msL'_p$.

This splits the tuples $(p, \msL_p)$ with $p$ outside $\Vplac_0$ into finitely many weak equivalence classes.
\end{mydef}
\section{Linearizing the sum}
\label{sSieve}

\begin{mydef}
\label{dSieve}
Let $C_1 > 1$ and $0 < c < 1$ be real numbers. For each prime $p$, we let $\mathcal{C}_p$ be a finite set together with a designated element $u_p \in \mathcal{C}_p$. We assume that $|\mathcal{C}_p| \leq C_1$ for all $p$. We let $\mu$ be a function defined on pairs $(p, a)$ with $a \in \mathcal{C}_p$ such that
$$
\mu(p, a) \leq \frac{C_1}{p}
$$
for all $a \neq u_p$ and such that
$$
\sum_{a \in \mathcal{C}_p} \mu(p, a) = 1.
$$
This allows us to introduce $m(p) = \sum_{a \in \mathcal{C}_p - \{u_p\}} \mu(p, a)$, which we then extend to a multiplicative function supported on squarefree integers.

Let $\mathcal{A}$ be the set of pairs $(n, \mathcal{D})$ consisting of a positive squarefree integer $n$ and a tuple $\mathcal{D} = (a_p)_{p \mid n}$ with $a_p \in \mathcal{C}_p - \{u_p\}$. Sometimes we will abuse notation and implicitly identify a pair $(n, (a_p)_{p \mid n})$ with the pair $(n, (b_p)_p)$, where $b_p = a_p$ for $p \mid n$ and $b_p = u_p$ for $p \nmid n$. Given a pair $\mathbf{t} = (n, (a_p)_p)$, we define $n(\mathbf{t}) = n$ and $\mathbf{t}_p = a_p$.

We let $H > 10$ be a real number and we let $X_H$ be a finite set. We say that a pair of maps $(\lambda, S)$ is Selmer compatible if
\begin{itemize}
\item $\lambda$ and $S$ are maps of type $\lambda: X_H \rightarrow \mathcal{A}$ and $S: \mathcal{A} \rightarrow \mathbb{R}_{\geq 0}$;
\item for all $x \in X_H$, we have that $n(\lambda(x)) \leq H$;
\item the map $S$ satisfies
$$
S(nm, \mathcal{D} \times \mathcal{E}) \leq S(n, \mathcal{D}) C_1^{\omega(m)}
$$
for all $(n, \mathcal{D}), (m, \mathcal{E}) \in \mathcal{A}$;
\item given a squarefree integer $1 \leq Q \leq H^c$, and given a choice of $a_p \in \mathcal{C}_p - \{u_p\}$ for $p \mid Q$, we have
\begin{align}
\label{eLevel}
\left|\frac{\# \{x \in X_H\,:\,\, \lambda(x)_p = a_p \,\text{ for all }\, p \mid Q\}}{\# X_H} \,-\, \prod_{p \mid Q} \mu(p, a_p)\right| \le H^{-c}.
\end{align}
\end{itemize}
 The previous inequality gives the existence of other constants $\kappa > 0$ and $K \geq 1$ satisfying
\begin{align}
\label{eSieveAssu}
\prod_{w \leq p < z} (1 - m(p))^{-1} \leq K \left(\frac{\log z}{\log w}\right)^\kappa 
\end{align}
for all $2C_1^2 \leq w < z$. For later use, we remark that valid choices of $\kappa$ and $K$ are $\kappa = C_1^2$ and $K = C^{C_1^2}$ for some absolute constant $C > 2$.
\end{mydef}

To shorten some notation, we introduce the quantity
\begin{align}
\label{egdD}
\rho(d, (a_p)_{p \mid d}) := \prod_{p \mid d} \mu(p, a_p) \prod_{\substack{p \mid d \\ p > 2C_1^2}} (1 - m(p))^{-1}.
\end{align}
Our next result gives a uniform upper bound, provided that $H$ is sufficiently large. Our proof follows fairly standard arguments available in the literature, see for example \cite{CKPS, Erdos, Henriot, Nair, NT, Shiu, Wolke}. We must redo these arguments as the dependence of the implied constant on $C_1$ is needed to bound Selmer moments uniformly in $\kappa$. Fortunately, we are able to introduce some simplifications compared to the literature as we are only interested in multiplicative functions supported on squarefree integers.

\begin{theorem}
\label{tSieve}
There exists an absolute constant $C > 0$ such that the following holds. Let $S, \lambda$ be Selmer compatible with parameters $C_1 > 1$ and $0 < c < 1$. Set
\begin{gather*}
B(\kappa, K) := K \left(1 + 4 (9\kappa + 1)^\kappa e^{9\kappa} K^{11}\right) \\
S := (1 + 2^\kappa) C_1^{6/c} + \sum_{s = 2}^\infty \frac{(C C_1^3)^{2C_1^3} \cdot C_1^{\frac{3 (s + 1)}{c}} \cdot (s + 1)^\kappa \cdot \exp\left(4 C_1^3 s^{1/2}\right)}{e^{\frac{s \log s}{4}}}.
\end{gather*}
Then we have for all $H > \exp\left((2C_1)^{C/c^2}\right)$
\begin{align}
\label{eXHSum}
\sum_{x \in X_H} S(\lambda(x)) \leq C \cdot B(\kappa, K) S \#X_H \prod_{2C_1^2 < p \leq H^{c/3}} (1 - m(p)) \sum_{\substack{(d, \mathcal{D}) \in \mathcal{A} \\ d \leq H^{c/3}}} S(d, \mathcal{D}) \rho(d, \mathcal{D}).
\end{align}
\end{theorem}

In certain favorable situations, it is possible to also prove a matching lower bound. This will be done in the next theorem.

\begin{theorem}
\label{tSieve2}
There exists an absolute constant $C > 0$ such that the following holds. Let $S, \lambda$ be Selmer compatible with parameters $C_1 > 1$ and $0 < c < 1$. Moreover, assume that we have, for all pairs $(n, \mathcal{D}), (m, \mathcal{E}) \in \mathcal{A}$ with $\gcd(n, m) = 1$,
\begin{align}
\label{eLowerS}
S(nm, \mathcal{D} \times \mathcal{E}) \geq S(n, \mathcal{D}) C_1^{-\omega(m)}.
\end{align}
Set $v := \frac{c}{27\kappa + 3\log(3K^{10})}$. Then we have for all $H > \exp\left((2C_1)^{C/c^2}\right)$
$$
\sum_{x \in X_H} S(\lambda(x)) \geq \frac{\# X_H}{2 C_1^{\lfloor 1/v \rfloor}} \prod_{2C_1^2 < p \leq H^v} (1 - m(p)) \sum_{\substack{(d, \mathcal{D}) \in \mathcal{A} \\ d \leq H^v}} S(d, \mathcal{D}) \rho(d, \mathcal{D}).
$$
\end{theorem}

\subsection{Results on multiplicative functions and sieves}
We start with an upper bound on smooth numbers following \cite[Lemma 1]{Shiu}.

\begin{lemma}
\label{lSmooth}
Let $F: \Z_{\geq 1} \rightarrow \R_{\geq 0}$ be a multiplicative function supported on squarefree integers. Assume that there exists a positive real number $c_0$ such that $F(p) \leq c_0/p$ for every prime $p$. Then we have for all real numbers $x, z \geq 2$
$$
\sum_{\substack{z \leq n \leq x \\ p \mid n \Rightarrow p \leq \log x}} F(n) \leq \frac{\exp\left(2c_0 (\log x)^{1/2}\right)}{z^{1/2}}.
$$
\end{lemma}

\begin{proof}
Set $y := \log x$, and let $0 < c < 1$ be a real number to be chosen later. By Rankin's trick, we have the estimates
$$
\sum_{\substack{z \leq n \leq x \\ p \mid n \Rightarrow p \leq \log x}} F(n) \leq \frac{1}{z^c} \sum_{\substack{1 \leq n \leq x \\ p \mid n \Rightarrow p \leq y}} F(n) n^c \leq \frac{1}{z^c} \prod_{p \leq y} \left(1 + F(p) p^c\right).
$$
The product may be estimated by
\begin{align*}
\prod_{p \leq y} \left(1 + F(p) p^c\right) &\leq \prod_{p \leq y} \left(1 + c_0 p^{c - 1}\right) \leq \exp\left(\sum_{p \leq y} c_0 p^{c - 1}\right) \\
&\leq \exp\left(\sum_{2 \leq n \leq y} c_0 n^{c - 1}\right) \leq \exp\left(\frac{c_0 y^c}{c} - \frac{c_0}{c}\right) \leq \exp\left(\frac{c_0 y^c}{c}\right)
\end{align*}
by viewing $\sum_{2 \leq n \leq y} n^{c - 1}$ as a right Riemann sum for $\int_1^y t^{c - 1} dt$ with spacing $1$. We pick $c := 1/2$ to end the proof of the lemma.
\end{proof}

\begin{lemma}
\label{lNT}
Let $F: \Z_{\geq 1} \rightarrow \mathbb{R}_{\geq 0}$ and let $C_1 > 1$ be a real number. Suppose that $F$ is supported on squarefree integers and satisfies
\begin{equation}
\label{eFsub}
F(ab) \leq \frac{F(a) C_1^{\omega(b)}}{b}
\end{equation}
for all $a, b \in \Z_{\geq 1}$ with $\gcd(a, b) = 1$. Then there exists an absolute constant $C > 0$ such that for all $x \geq z \geq C$ and all real numbers $\delta > 0$
$$
\sum_{\substack{n \geq x \\ P^+(n) \leq z}} F(n) \leq \frac{(CC_1)^{C_1} \exp\left(2 C_1 e^\delta\right)}{\exp\left(\delta \log x/\log z\right)} \sum_{n \leq z} F(n).
$$
\end{lemma}

\begin{proof}
We let $\beta := \delta/\log z > 0$. Then we have
$$
\sum_{\substack{n \geq x \\ P^+(n) \leq z}} F(n) \leq \frac{1}{x^\beta} \sum_{P^+(n) \leq z} F(n) n^\beta.
$$
Let $\psi$ be the multiplicative function with $(\psi \ast 1)(n) = n^\beta$, so $n^\beta = \sum_{d \mid n} \psi(d)$. Therefore we obtain the sum
$$
\sum_{P^+(n) \leq z} F(n) n^\beta = \sum_{P^+(d) \leq z} \sum_{P^+(e) \leq z} F(de) \psi(d).
$$
Since $F$ is supported on squarefrees, we may restrict the sum to pairs $d, e$ with $\gcd(d, e) = 1$. The submultiplicativity assumption on $F$ gives the upper bound
\begin{align}
\label{eUpper1}
\sum_{\substack{n \geq x \\ P^+(n) \leq z}} F(n) \leq \frac{1}{x^\beta} \sum_{P^+(e) \leq z} F(e) \sum_{P^+(d) \leq z} \frac{C_1^{\omega(d)} \psi(d)}{d}.
\end{align}
Then the inner sum equals an Euler product, which we estimate via
$$
\prod_{p \leq z} \left(1 + \frac{C_1 \psi(p)}{p} \right) = \prod_{p \leq z} \left(1 + \frac{C_1 (p^\beta - 1)}{p} \right) \leq \exp\left(C_1 \sum_{p \leq z} \frac{p^\beta - 1}{p}\right).
$$
Our aim is now to understand the inner sum. Observe that the inequality
\begin{equation}
\label{eIneqExp}
e^x - 1 \leq \frac{x}{\delta} e^\delta
\end{equation}
holds for all $0 \leq x \leq \delta$. Indeed, it is readily verified that this inequality holds at $x = 0$ and $x = \delta$, and thus the general inequality follows from the fact that $e^x - 1 - \alpha x$ has positive second derivative for any fixed real number $\alpha \in \R$. Applying \eqref{eIneqExp} with $x = \beta \log p = \frac{\delta \log p}{\log z}$ gives
$$
\sum_{p \leq z} \frac{p^\beta - 1}{p} = \sum_{p \leq z} \frac{e^{\beta \log p} - 1}{p} \le \frac{\beta e^\delta}{\delta} \sum_{p \leq z} \frac{\log p}{p} = \frac{e^\delta}{\log z} \sum_{p \leq z} \frac{\log p}{p} \leq 2e^\delta
$$
for $z$ larger than an absolute constant, where the last inequality follows by Mertens' theorem. Returning to equation \eqref{eUpper1} and recognizing $x^\beta = \exp(\delta \log x/\log z)$, we have thus far shown that
\begin{equation}
\label{eUpper2}
\sum_{\substack{n \geq x \\ P^+(n) \leq z}} F(n) \leq \frac{\exp\left(2 C_1 e^\delta\right)}{\exp\left(\delta \log x/\log z\right)} \sum_{P^+(e) \leq z} F(e).
\end{equation}
We shall eventually apply equation \eqref{eUpper2} twice. For now, we set $W := 2e C_1 + \log(2)$ and continue to estimate the right hand side of equation \eqref{eUpper2} as
\begin{equation}
\label{eLoweringTrick}
\sum_{P^+(e) \leq z} F(e) \leq \prod_{z^{1/W} < p \leq z} \left(1 + \frac{C_1}{p}\right) \sum_{P^+(e) \leq z^{1/W}} F(e) \leq (C W)^{C_1} \sum_{P^+(e) \leq z^{1/W}} F(e)
\end{equation}
for $z$ larger than some absolute constant $C > 0$. Finally, we observe that
$$
\sum_{P^+(e) \leq z^{1/W}} F(e) \leq \sum_{e \leq z} F(e) + \sum_{\substack{e > z \\ P^+(e) \leq z^{1/W}}} F(e) \leq \sum_{e \leq z} F(e) + \frac{1}{2} \sum_{P^+(e) \leq z^{1/W}} F(e),
$$
where we applied equation \eqref{eUpper2} with $\delta = 1$ in our last inequality. Rearranging the above inequality yields
\begin{align}
\label{eRearrangeTrick}
\sum_{P^+(e) \leq z^{1/W}} F(e) \leq 2 \sum_{e \leq z} F(e).
\end{align}
Inserting the two inequalities \eqref{eLoweringTrick} and \eqref{eRearrangeTrick} into \eqref{eUpper2} ends the proof of the lemma.
\end{proof}

We now turn to the relevant sieve result for our arguments. As is traditional in sieving, we write $P(z) := \prod_{p < z} p$.

\begin{theorem}
\label{tComb}
Let $(a_n)_{n \in \Z}$ be a finitely supported sequence of nonnegative real numbers. Let $D \geq z \geq 2$, and let $X$ be a real number. Let $\rho$ be a multiplicative function with $0 \leq \rho(p) < 1$ for all $p$. Define
$$
r_d := \rho(d) X - \sum_{n \equiv 0 \bmod d} a_n
$$
for all squarefree $d < D$. Let $K > 1$ and $\kappa > 0$ be such that
$$
\prod_{w \leq p < z} (1 - \rho(p))^{-1} \leq K\left(\frac{\log z}{\log w}\right)^\kappa
$$
for all $z > w \geq 2$. Then we have
$$
\sum_{\gcd(n, P(z)) = 1} a_n \leq X \left(1 + 4 (9\kappa + 1)^\kappa e^{9\kappa - s} K^{11}\right) \prod_{p < z} (1 - \rho(p)) + \sum_{\substack{d < D \\ d \mid P(z)}} |r_d|,
$$
where $s = \log D/\log z$.
\end{theorem}

\begin{proof}
See \cite[Corollary 6.10]{FI}.
\end{proof}

\subsection{The Erd\H{o}s--Nair--Tenenbaum sieve}
We now prove Theorem \ref{tSieve}.

\begin{proof}
If $x \in X_H$, then $n(\lambda(x))$ is a positive squarefree integer of size at most $H$. We set $Z := H^{c/3}$. Then we factor
$$
n(\lambda(x)) = \prod_{i = 1}^r p_i
$$
with $r \in \Z_{\geq 0}$ and primes $p_1 < \dots < p_r$. We let $0 \leq j < r$ be the unique index such that
\begin{equation}
\label{ePrimeFactorsSpltting}
\prod_{i = 1}^j p_i \leq Z, \quad \quad \prod_{i = 1}^{j + 1} p_i > Z,
\end{equation}
taking $j := r$ if no such index exists. Then we define
$$
a(x) := \prod_{i = 1}^j p_i, \quad \quad b(x) := \prod_{i = j + 1}^r p_i.
$$
Our estimation of \eqref{eXHSum} will be split in two cases
\begin{enumerate}
\item[(i)] $p_{j + 1} \leq \log Z$,
\item[(ii)] $p_{j + 1} > \log Z$,
\end{enumerate}
where we define $p_{r + 1} := \infty$.

\paragraph{Case (i).}
Since we set $p_{r + 1} := \infty$ by definition, we get $j < r$ in this case. Moreover, by the inequalities $p_{j + 1} \leq \log Z$ and \eqref{ePrimeFactorsSpltting}, we obtain 
$$
a(x) \geq \frac{Z}{\log Z} > Z^{1/2}
$$
for $H^{c/3}$ larger than an absolute constant. Because $n(\lambda(x)) \leq H$, there exists an absolute constant $C > 0$ such that the trivial bound 
$$
S(\lambda(x)) \leq S(1, (u_p)_p) \cdot C_1^{\frac{C \log H}{\log \log H}}
$$ 
holds for $H > 10$. This yields
\begin{equation}
    \label{eq:CaseIBnd}
\sum_{\substack{x \in X_H \\ \text{case (i)}}} S(\lambda(x)) \leq S(1, (u_p)_p) \cdot C_1^{\frac{C \log H}{\log \log H}} \sum_{\substack{x \in X_H \\ \text{case (i)}}} 1.
\end{equation}
Setting $d := a(x)$, we get $Z^{1/2} < d \leq Z$. Hence we have
$$
\sum_{\substack{x \in X_H \\ \text{case (i)}}} 1 \leq \sum_{\substack{Z^{1/2} < d \leq Z \\ p \mid d \Rightarrow p \leq \log Z}} \sum_{\substack{x \in X_H \\ a(x) = d}} 1 \leq \sum_{\substack{Z^{1/2} < d \leq Z \\ p \mid d \Rightarrow p \leq \log Z}} \sum_{\substack{x \in X_H \\ d \mid n(\lambda(x))}} 1.
$$
We employ equation \eqref{eLevel} to show that the inner sum is at most 
$$
\sum_{\substack{x \in X_H \\ d \mid n(\lambda(x))}} 1 \leq \#X_H \left(m(d) + \frac{C_1^{\omega(d)}}{H^c}\right) \leq \#X_H \left(m(d) + \frac{C_1^{\frac{C \log H}{\log \log H}}}{H^c}\right).
$$
Inserting this estimate back in and applying Lemma \ref{lSmooth} shows that the left hand side of \eqref{eq:CaseIBnd} is at most
\begin{align}
\label{eUpperBound1}
S(1, (u_p)_p) C_1^{\frac{C \log H}{\log \log H}} \#X_H \left(\frac{Z C_1^{\frac{C \log H}{\log \log H}}}{H^c} + \frac{\exp\left(2 C_1^2 (\log Z)^{1/2}\right)}{Z^{1/4}}\right).
\end{align}
Indeed, Lemma \ref{lSmooth} applies, since $m(d)$ is a multiplicative function supported on squarefrees and bounded on primes $p$ by $C_1^2/p$.

\paragraph{Case (ii).}
In this case we have
\begin{equation}
\label{eq:CaseIIBnd}
\sum_{\substack{x \in X_H \\ \text{case (ii)}}} S(\lambda(x)) \leq \sum_{\substack{x \in X_H \\ \text{case (ii)}}} S(a(x), \mathcal{D}(x)) C_1^{\omega(b(x))}
\end{equation}
by Definition \ref{dSieve}, where we have set $\mathcal{D}(x) := (\lambda(x)_p)_{p \mid a(x)}$. Set
$$
s := \left \lfloor \frac{\log Z}{\log p_{j + 1}} \right \rfloor,
$$
which is to be interpreted as $s = 0$ in case $p_{j + 1} = \infty$. By construction we have $Z^{1/(s + 1)} < p_{j + 1} \leq Z^{1/s}$. Moreover, we have
$$
s \in \Z \cap \left[0, \frac{\log Z}{\log \log Z}\right]
$$
by virtue of being in case (ii). By definition of $Z$, we get the upper bound $\omega(b(x)) \leq 3 (s + 1)/c$. Setting $d := a(x)$ again and splitting the sum over $s$, we obtain
\begin{align}
\label{eSplitOvers}
\sum_{\substack{x \in X_H \\ \text{case (ii)}}} S(a(x), \mathcal{D}(x)) C_1^{\omega(b(x))} \leq 
\sum_{0 \leq s \leq \frac{\log Z}{\log \log Z}} C_1^{\frac{3(s + 1)}{c}} \sum_{\substack{(d, \mathcal{D}) \\ d \leq Z}} S(d, \mathcal{D}) \sum_{\substack{x \in X_H \\ (a(x), \mathcal{D}(x)) = (d, \mathcal{D}) \\ Z^{1/(s + 1)} < P^-(b(x)) \leq Z^{1/s}}} 1.
\end{align}
Let us now record the observation that $s \geq 2$ implies $d \geq Z^{1/2}$. 

At this stage our goal will be to upper bound the inner sum by an application of the fundamental lemma of sieve theory. We define
$$
P_{s, d} := \prod_{\substack{2C_1^2 < p \leq Z^{1/(s + 1)} \\ p \nmid d}} p,
$$
so $\gcd(n(\lambda(x)), P_{s, d}) = 1$. For a fixed pair $(d, \mathcal{D}) = (d, (c_p)_p)$, we have
$$
\sum_{\substack{x \in X_H \\ (a(x), \mathcal{D}(x)) = (d, \mathcal{D}) \\ Z^{1/(s + 1)} < P^-(b(x)) \leq Z^{1/s}}} 1 \leq \sum_{\substack{x \in X_H \\ \forall p \mid d \, : \,  \lambda(x)_p = c_p \\ \gcd(n(\lambda(x)), P_{s, d}) = 1}} 1.
$$
We now apply Theorem \ref{tComb}. In the notation of that theorem, we take
\begin{gather*}
a_m := \# \left\{x \in X_H : p \mid d \Rightarrow \lambda(x)_p = c_p \text{ and } dm = \frac{n(\lambda(x))}{\prod_{\substack{p \leq 2C_1^2 \\ p \mid n(\lambda(x))}} p}\right\}, \quad X := \# X_H \prod_{p \mid d} \mu(p, c_p), \\
D := H^{c/2} = Z^{3/2}, \quad \quad z := Z^{\min(1/(s + 1), 1/2)}.
\end{gather*}
With this notation set, we have
$$
|r_e| = \left|\sum_{m \equiv 0 \bmod e} a_m - \#X_H m(e) \prod_{p \mid d} \mu(p, c_p)\right| \leq \frac{\#X_H C_1^{\omega(e)}}{H^c}
$$
for all $e \mid P_{s, d}$ with $e \leq H^{c/2}$ by equation \eqref{eLevel} (note that, by construction, $r_e = 0$ if $e$ is not coprime with $d \prod_{p \leq 2C_1^2} p$). Hence Theorem \ref{tComb}, whose assumptions are satisfied thanks to the bound \eqref{eSieveAssu}, gives
$$
\sum_{\substack{x \in X_H \\ \forall p \mid d \, : \,  \lambda(x)_p = c_p \\ \gcd(n(\lambda(x)), P_{s, d}) = 1}} 1 \leq B'(\kappa, K) \# X_H \prod_{p \mid d} \mu(p, c_p) \hspace{-0.4cm} \prod_{\substack{2C_1^2 < p \leq Z^{\frac{1}{s + 1}} \\ p \nmid d}} \hspace{-0.4cm} \left(1 - m(p)\right) + \frac{\# X_H}{H^c} \sum_{\substack{e \leq H^{c/2} \\ e \mid P_{s, d}}} C_1^{\omega(e)}
$$
with $B'(\kappa, K) := 1 + 4 (9\kappa + 1)^\kappa e^{9\kappa} K^{11}$. 

After summing over $s, d, \mathcal{D}$ in equation \eqref{eSplitOvers}, a direct calculation shows that the second term above can be absorbed in the main term of Theorem \ref{tSieve} for $H > \exp\big((2C_1)^{C/c^2}\big)$. Returning to equation \eqref{eSplitOvers} and recalling that $s \geq 2$ implies $d \geq Z^{1/2}$, we find that the left hand side of \eqref{eq:CaseIIBnd} is at most
$$
B(\kappa, K) \#X_H \prod_{2C_1^2 < p \leq Z} (1 - m(p)) \sum_{0 \leq s \leq \frac{\log Z}{\log \log Z}} C_1^{\frac{3(s + 1)}{c}} (s + 1)^\kappa \hspace{-0.7cm} \sum_{\substack{(d, \mathcal{D}) \\ d \leq Z, P^+(d) \leq Z^{1/s} \\ s \geq 2 \Rightarrow d \geq Z^{1/2}}} \hspace{-0.6cm} S(d, \mathcal{D}) \rho(d, \mathcal{D}),
$$
where we applied the bound \eqref{eSieveAssu} and where we recall the notations \eqref{egdD} and $B(\kappa, K) = K B'(\kappa, K)$. For $s \geq 2$, we allude to Lemma \ref{lNT} with 
$$
F(d) = \sum_{\substack{\mathcal{D} \\ (d, \mathcal{D}) \text{ a pair}}} S(d, \mathcal{D}) \cdot \rho(d, \mathcal{D}).
$$
To check that the condition \eqref{eFsub} of Lemma \ref{lNT} is satisfied, we prove that for all squarefree integers $d, e$ with $\gcd(d, e) = 1$
$$
F(de) \leq \frac{F(d) \cdot (2C_1^3)^{\omega(e)}}{e}.
$$
Indeed, this follows from
\begin{align*}
F(de) &= \sum_{\substack{\mathcal{D}, \mathcal{E} \\ (de, \mathcal{D} \times \mathcal{E}) \text{ a pair}}} S(d e, \mathcal{D} \times \mathcal{E}) \cdot \rho(de, \mathcal{D} \times \mathcal{E}) \nonumber \\
&\leq C_1^{\omega(e)} \sum_{\substack{\mathcal{D} \\ (d, \mathcal{D}) \text{ a pair}}} S(d, \mathcal{D}) \cdot \rho(d, \mathcal{D}) \times \sum_{\substack{\mathcal{E} \\ (e, \mathcal{E}) \text{ a pair}}} \rho(e, \mathcal{E}) \nonumber \\
&\leq 2^{\omega(e)} \cdot C_1^{\omega(e)} \cdot m(e) \sum_{\substack{\mathcal{D} \\ (d, \mathcal{D}) \text{ a pair}}} S(d, \mathcal{D}) \cdot \rho(d, \mathcal{D}) \leq \frac{F(d) \cdot (2C_1^3)^{\omega(e)}}{e}
\end{align*}
by submultiplicativity of $S$, the inequality $m(p) \leq 1/2$ for $p \geq 2C_1^2$ and the inequality $m(p) \leq C_1^2/p$. Inserting the bound from Lemma \ref{lNT} with $\delta = \frac{1}{2} \log s$ and simply keeping the terms with $s \in \{0, 1\}$, we get
\begin{align}
\label{eUpperBound2}
\leq B(\kappa, K) S \#X_H \prod_{2C_1^2 < p \leq Z} (1 - m(p)) \sum_{\substack{(d, \mathcal{D}) \\ d \leq Z}} S(d, \mathcal{D}) \rho(d, \mathcal{D})
\end{align}
with 
$$
S := (1 + 2^\kappa) C_1^{6/c} + \sum_{s = 2}^\infty \frac{(C C_1^3)^{2C_1^3} \cdot C_1^{\frac{3 (s + 1)}{c}} \cdot (s + 1)^\kappa \cdot \exp\left(4 C_1^3 s^{1/2}\right)}{e^{\frac{s \log s}{4}}}
$$
for an absolute constant $C > 0$.

\paragraph{Conclusion of proof.}
In case (i) we have proven the upper bound in \eqref{eUpperBound1}, while in case (ii) we have proven the upper bound in \eqref{eUpperBound2}. The upper bound \eqref{eUpperBound1} can be absorbed in \eqref{eUpperBound2} for $H > \exp((2C_1)^{C/c^2})$. Examining \eqref{eUpperBound2}, we get the main term in Theorem \ref{tSieve}, and the result follows.
\end{proof}

\subsection{The lower bound}
We now prove Theorem \ref{tSieve2}.

\begin{proof}
Let $v > 0$ be a small real number to be chosen later, and set $Z := H^v$. Given an integer $n$, we define
$$
n^\flat := \prod_{p \leq Z} p^{v_p(n)}.
$$
Fixing all pairs $(d, \mathcal{D}) = (d, (c_p)_p)$ with $P^+(d) \leq Z$, we have
\begin{align*}
\sum_{x \in X_H} S(\lambda(x)) &= \sum_{\substack{(d, \mathcal{D}) \\ P^+(d) \leq Z}} \sum_{\substack{x \in X_H \\ n(\lambda(x))^\flat = d \\ \forall p \mid d \, : \, \lambda(x)_p = c_p}} S(\lambda(x)) = \sum_{\substack{(d, \mathcal{D}) \\ P^+(d) \leq Z}} \sum_{\substack{x \in X_H \\ P^-(n(\lambda(x))/d) > Z \\ \forall p \mid d \, : \, \lambda(x)_p = c_p}} S(\lambda(x)) \\
&\geq \sum_{\substack{(d, \mathcal{D}) \\ 1 \leq d \leq Z}} \sum_{\substack{x \in X_H \\ P^-(n(\lambda(x))/d) > Z \\ \forall p \mid d \, : \, \lambda(x)_p = c_p}} S(\lambda(x)),
\end{align*}
where we used that $S(\lambda(x)) \geq 0$ and that all $d \leq Z$ satisfy $P^+(d) \leq d \leq Z$.  Since the number of prime divisors of $n(\lambda(x))/d$ is bounded by $\lfloor 1/v \rfloor$, we get the estimate
$$
\sum_{\substack{(d, \mathcal{D}) \\ 1 \leq d \leq Z}} \sum_{\substack{x \in X_H \\ P^-(n(\lambda(x))/d) > Z \\ \forall p \mid d \, : \, \lambda(x)_p = c_p}} S(\lambda(x)) \geq C_1^{- \lfloor 1/v \rfloor} \sum_{\substack{(d, \mathcal{D}) \\ 1 \leq d \leq Z}} S(d, \mathcal{D}) \sum_{\substack{x \in X_H \\ P^-(n(\lambda(x))/d) > Z \\ \forall p \mid d \, : \, \lambda(x)_p = c_p}} 1
$$
from our assumption \eqref{eLowerS}.

We apply \cite[Theorem 6.9]{FI} with $z \leftarrow Z$ and $D \leftarrow H^\gamma$ with $\gamma$ to be chosen soon. In particular, we must choose $\gamma \geq (9 \kappa + 1) v$ in order to apply this result. Then we get
$$
\sum_{\substack{x \in X_H \\ P^-(n(\lambda(x))/d) > Z \\ \forall p \mid d \, : \, \lambda(x)_p = c_p}} 1 \geq \# X_H \prod_{p \mid d} \mu(p, c_p) \prod_{\substack{2C_1^2 < p \leq Z \\ p \nmid d}} (1 - m(p)) \left(1 - e^{9 \kappa - \frac{\gamma}{v}} K^{10}\right) - \frac{\#X_H}{H^c} \sum_{e \leq H^\gamma} C_1^{\omega(e)}.
$$
This inspires us to choose $\gamma := \frac{c}{3}$, and we want to pick $v$ such that
$$
v \leq \frac{c}{27 \kappa + 3}, \quad \quad e^{9 \kappa - \frac{c}{3v}} K^{10} \leq \frac{1}{3}.
$$
Recalling that $K \geq 1$, we pick $v := \frac{c}{27\kappa + 3\log(3K^{10})}$, which satisfies both conditions above. Finally, observe that 
$$
\frac{C_1^{- \lfloor 1/v \rfloor} \#X_H}{H^c} \sum_{\substack{(d, \mathcal{D}) \\ 1 \leq d \leq Z}} S(d, \mathcal{D}) \sum_{e \leq H^\gamma} C_1^{\omega(e)} \leq C_1^{- \lfloor 1/v \rfloor} \cdot \#X_H \cdot S(1, (u_p)_p)
$$
for $H > \exp\left((2C_1)^{C/c^2}\right)$, completing the proof.
\end{proof}
\section{Gridding}
\label{sGrid}
Let $k \in \Z_{\geq 1}$ and let $g_1, \dots, g_k$ be nonnegative multiplicative functions supported on subsets of the positive squarefree integers. Assume that there is $C_2 > 100$ such that
\[
g_i(p) \le C_2p^{-1}
\]
for all primes $p$ and $1 \le i \le k$. A choice of $1 \le i \le k$ will be called \emph{rare} if $\sum_p g_i(p)$ is finite. We take $C_4$ to be the supremum of this sum over all choices of rare $i$. Fixing some $C_3 > 1$, let
\[
S: \Z_{\geq 1}^k \to \R_{\ge 0}
\]
be a function satisfying
\[
C_3^{-1} \cdot S(n_1, \dots, n_k) \le S(n_1, \dots, pn_i, \dots, n_k) \le C_3 \cdot S(n_1, \dots, n_k)
\]
for all $1 \le i \le k$, all primes $p$, and all tuples $(n_1, \dots, n_k)$ of nonnegative integers indivisible by $p$. It will be convenient to assume $S(n_1, \dots, n_k) = 0$ if $n_1 \cdots n_k$ is not squarefree.

We are aiming to control the order of growth of
$$
\sum_{\substack{(n_1, \dots, n_k) \in \Z_{\geq 1}^k \\ \prod_{i \le k} n_i \le H}} g_1(n_1) \cdots g_k(n_k) S(n_1, \dots, n_k).
$$
Following \cite{Smi22b}, we will do this by cutting the space of tuples $(n_1, \dots, n_k)$ into subsets parameterized by certain product spaces known as \emph{grids}. 

\begin{mydef}[{\cite[Definition 8.3]{Smi22b}}]
Given $H \ge 25$, take
\[
\alpha_0(H) = \exp^{(3)}\left(\tfrac{1}{3}\log^{(3)} H\right) \quad\text{and}\quad \alpha(H) = 1 + \exp^{(3)}\left(\tfrac{1}{4}\log^{(3)} H\right)^{-1}.
\]
A grid of height $H$ is a tuple $\left(r, (X_i)_{i \le r}, (k_i)_{i \le r}\right)$ satisfying:
\begin{itemize}
    \item $r \in \Z_{\geq 0}$, $X_i$ is a set of primes for each integer $1 \le i \le r$, and $k_i$ is an integer $1 \leq k_i \leq k$, 
    \item each $X_i$ is either a singleton prime smaller than $\alpha_0(H)$, or the set of all primes in an interval of the form
\begin{equation}
\label{eq:grid_axis}
\big[\alpha_0(H) \alpha(H)^m,\, \alpha_0(H) \alpha(H)^{m+1}\big)
\end{equation}
for some integer $m$, 
    \item the sets $X_i$ and $X_j$ are pairwise disjoint, and 
    \item $\prod_{1 \le i \le r} \sup(X_i)$ is no larger than $H$.
\end{itemize}
Given a grid of height $H$, and given a point $x = (p_1, \dots, p_r)$ in $X = \prod_{1 \le j \le r} X_j$ and $1 \le i \le k$, we define
\[
n_i(x) = \prod_{\substack{1 \le j \le r \\ k_j = i}} p_j.
\]
We then say that $\overline{n}(x) := (n_1(x), n_2(x), \dots, n_k(x))$ is contained in a grid of height $H$. We also define
$$
g(\overline{n}(x)) := g_1(n_1(x)) \cdots g_k(n_k(x)).
$$
Given a grid of height $H$, take $S_{\text{sm}}$ to be the set of $1 \le i \le r$ such that $X_i$ is a singleton, take $S_{\text{med}}$ to be the set of $i$ such that $X_i$ is not a singleton but 
$$
\sup(X_i) \le \exp^{(3)}\left(\tfrac{1}{2} \log^{(3)} H\right),
$$
and take $S_{\text{lg}}$ to be the set of all remaining indices.
\end{mydef}

As in \cite{Smi22b}, we will need a notion of a good grid of height $H$. This requires some adjustment, since we are putting very few restrictions on the multiplicative functions $g_i$.

\begin{mydef}
\label{defn:goodgrid}
Given a grid $(r, (X_i)_{i \le r}, (k_i)_{i \le r})$ of height $H$, we call the grid \emph{okay} if the following properties all hold:
\begin{enumerate}
    \item We have $r \le (\log \log H)^2$.
    \item For all $1 \le i \le r$ not in $\Ssm$,
    \[\sum_{p \in X_i} g_{k_i}(p) \ge \exp^{(3)}\left(\tfrac{2}{7}\log^{(3)} H\right)^{-1}.\]
    \item For all $1 \le i \le r$, $k_i$ is not rare.
\end{enumerate}
Fix $0 < c_1 < 1$. Assuming that the grid $(r, (X_i)_{i \le r}, (k_i)_{i \le r})$ is okay, call it $c_1$-\emph{good} if
\begin{enumerate}
    \item[4.] Defining $S_{\text{sm}}$ and $S_{\text{med}}$ as above, we have
    \[
    \# S_{\text{sm}} \le (\log^{(2)} H)^{\frac{1}{3} + \frac{1}{100}} \quad \text{and}\quad \# S_{\text{med}} \le (\log^{(2)} H)^{\frac{1}{2} + \frac{1}{100}}.
    \]
    \item[5.] For all $1 \le j \le k$ such that
    \[
    \sum_{p \le H} g_j(p) \ge c_1 \log \log H,
    \]
    there are at least $(\log^{(2)} H)^{\frac{5}{6}}$ choices of $i$ such that $k_i = j$ and
    \[\sum_{p \in X_i} g_j(p) \ge \sum_{p \in X_i} \tfrac{1}{3}c_1p^{-1}.\]
\end{enumerate}
\end{mydef}

\begin{proposition}
\label{prop:gridding}
There exists an absolute constant $C > 0$ such that the following holds. Let $H \geq 25$ be a real number such that 
\begin{equation}
    \label{eq:H_grid_bnd}
H \geq \exp^{(2)} \left((k C_2 C_3/c_1)^C\right).
\end{equation}
Let $A > 1$ be such that, for all okay grids $X$ of height $H$, we have 
\begin{equation}
\label{eq:okay_assu}
\sum_{x \in X} g\left(\overline{n}(x)\right) S\left(\overline{n}(x)\right) \le A C_3^{|S_{\textup{sm}}|} \sum_{x \in X} g\left(\overline{n}(x)\right),
\end{equation}
and for all $c_1$-good grids $X$ of height $H$, we have
\begin{equation}
\label{eq:good_assu}
\sum_{x \in X} g\left(\overline{n}(x)\right) S\left(\overline{n}(x)\right) \le A \sum_{x \in X} g\left(\overline{n}(x)\right).
\end{equation}
Then we have
\begin{equation}
\label{eq:gridding_prop}
\sum_{n_1 \cdots n_k \le H} g_1(n_1) \cdots g_k(n_k) S(n_1, \dots, n_k) \le 2\exp(kC_3C_4) A \sum_{\substack{n_1 \cdots n_k \le H \\ \mu^2(n_1 \cdots n_k) = 1}} g_1(n_1) \cdots g_k(n_k).
\end{equation}
\end{proposition}

We first state a lemma to handle rare indices.

\begin{lemma}
Suppose $1$ is rare, so $\sum_p g_1(p) \le C_4$. Then 
\[
\sum_{n_1 \cdots n_k \le H} g_1(n_1) \cdots g_k(n_k) S(n_1, \dots, n_k) \leq \exp(C_3C_4) \hspace{-0.1cm} \sum_{n_2 \cdots n_k \le H} \hspace{-0.1cm} g_2(n_2) \cdots g_k(n_k) S(1, n_2, \dots, n_k).
\]
\end{lemma}

\begin{proof}
We may bound the left hand side by
\[
\sum_{\substack{(n_1, \dots, n_k) \in \Z_{\geq 1}^k \\ n_1 \cdots n_k \le H}} g_1(n_1) \cdots g_k(n_k) C_3^{\omega(n_1)} S(1, n_2, \dots, n_k),
\]
which in turn is at most
\[
\left(\sum_{n_1 \ge 1} g_1(n_1) C_3^{\omega(n_1)}\right) \cdot \left(\sum_{\substack{(n_2, \dots, n_k) \in \Z_{\geq 1}^{k-1} \\ n_2 \cdots n_k \le H}} g_2(n_2) \cdots g_k(n_k) S(1, n_2, \dots, n_k) \right).
\]
The left term in this product is given by
\[
\prod_p \left(1 + C_3 g_1(p)\right),
\]
and the logarithm of this product is bounded by $\sum_p C_3 g_1(p)$, which is at most $C_3C_4$.
\end{proof}

We now prove the proposition.

\begin{proof}[Proof of Proposition \ref{prop:gridding}]
By applying the above lemma after the necessary permutation of $\{1, \dots, k\}$ to every rare index, of which there are at most $k$, we find that it suffices to show that
\[
\sum_{\substack{n_1 \cdots n_k \le H \\ n_i = 1 \text{ if } i \text{ is rare}}} g_1(n_1) \cdots g_k(n_k) S(n_1, \dots, n_k) \le 2A \sum_{\substack{n_1 \cdots n_k \le H \\ \mu^2(n_1 \cdots n_k) = 1}} g_1(n_1) \cdots g_k(n_k),
\]
where the left sum is over all $(n_1, \dots, n_k)$ of product at most $H$ such that $n_i$ is $1$ for all rare indices $i$. We also remark that
\begin{equation}
\label{eq:trivial_lower}
\sum_{\substack{n_1 \cdots n_k \le H \\ \mu^2(n_1 \cdots n_k) = 1}} g_1(n_1) \cdots g_k(n_k) \geq 1
\end{equation}
by looking at the contribution from $n_1 = \dots = n_k = 1$. For all squarefree $n$, we have the bound
\begin{equation}
\label{eq:onen_gridding}
\sum_{n_1 \cdots n_k = n} g_1(n_1) \cdots g_k(n_k) S(n_1, \dots, n_k) \le \frac{(k C_2 C_3)^{\omega(n)} S(1, \dots, 1)}{n} \le \frac{A (k C_2 C_3)^{\omega(n)}}{n},
\end{equation}
where we have used the fact that $(1, \dots, 1)$ is in an okay grid by itself for the final bound. The sum over $(n_1, \dots, n_k)$ with $n_1 \cdots  n_k$ having at least $(\log \log H)^2$ prime factors is then at most 
\begin{align*}
\exp\left(-(\log \log H)^2\right) \sum_{\substack{n \le H \\ \mu^2(n) = 1}} \frac{A (e \cdot k C_2 C_3)^{\omega(n)}}{n} &\le A \exp\left(-(\log \log H)^2\right) \prod_{p \le H} \left(1 + ekC_2C_3 p^{-1} \right) \\ &\le A\exp\left(-(\log \log H)^2\right) \cdot (10\log H)^{ekC_2C_3},
\end{align*}
where the last inequality follows for $H \ge 25$ by Mertens' theorem.

Note that a squarefree integer $n = n_1 \cdots n_k$ lies in a grid of height $H$ if $n$ has no two prime divisors $p_1, p_2$ satisfying
\begin{equation}
\label{eq:close_primes}
\alpha_0(H) \le p_1 < p_2 \le \alpha(H) p_1
\end{equation}
and if
\[
\alpha(H)^{\omega(n)} n \le H.
\]
A squarefree $n$ not satisfying the latter condition with at most $(\log \log H)^2$ prime factors must lie in the interval
\[
H \alpha(H)^{- (\log \log H)^2} \le n \le H,
\]
and \eqref{eq:onen_gridding} gives that the contribution of these $n$ to \eqref{eq:gridding_prop} is at most
\[
\frac{A \cdot (k C_2C_3)^{(\log \log H)^2}}{H \alpha(H)^{- (\log \log H)^2}} \left(1 + H \left(1 - \alpha(H)^{-(\log \log H)^2} \right)\right).
\]
For an appropriate choice of the absolute constant $C > 0$, the bound \eqref{eq:H_grid_bnd} implies this is at most $A/100$. Meanwhile, the contribution from $n$ with $2$ prime divisors satisfying \eqref{eq:close_primes} to the sum may be bounded by
\[
\sum_{(p_1, p_2)} \frac{(k C_2)^2}{p_1p_2} \sum_{n \le H} \frac{A \mu^2(n) (k C_2 C_3)^{\omega(n)}}{n},
\]
where the sum is over $(p_1, p_2)$ satisfying \eqref{eq:close_primes}. By the effective prime number theorem, we have
\[
\sum_{(p_1, p_2)} \frac{(k C_2)^2}{p_1p_2} \ll (k C_2)^2 (\alpha(H) - 1) \log \log H
\]
for $H \ge 25$, where the implicit constant is absolute. Meanwhile, the latter sum is bounded by
\[
A(10 \log H)^{kC_2C_3},
\]
again by Mertens' theorem. 

In a similar vein, if we take $Y$ to be the set of pairs $(I, j)$, where $I$ is an interval of the form \eqref{eq:grid_axis} for some nonnegative integer with supremum no larger than $H$, where $j$ is a positive integer no larger than $k$, and where
\[
\sum_{p \in I} g_j(p) \le \exp^{(3)} \left(\tfrac{2}{7} \log^{(3)} H\right)^{-1},
\]
then we find that
\[
\sum_{(I, j)} \sum_{p \in I}  g_j(p) \le \exp^{(3)} \left(\tfrac{2}{7} \log^{(3)} H\right)^{-1} \cdot k \exp^{(3)} \left(\tfrac{1}{4} \log^{(3)} H\right) \cdot \log H, 
\]
which is very small. Calling this quantity $\delta$, we find from \eqref{eq:onen_gridding} that the contribution to \eqref{eq:gridding_prop} from grids not satisfying the second part of Definition \ref{defn:goodgrid} is at most $A \delta (10 \log H)^{kC_2C_3}$.

All together, if $H \geq \exp^{(2)} \left(Ck C_2 C_3\right)$ for a sufficiently large absolute constant $C$, we find that the subsum of \eqref{eq:gridding_prop} over the $(n_1, \dots, n_k)$ such that $n_i =1$ for all rare $i$ but such that $(n_1, \dots, n_k)$ does not lie in an okay grid of height $H$ is at most $\tfrac{1}{2}A$. This is acceptable in view of \eqref{eq:trivial_lower}.

We now claim that the subsum over all okay grids that are not good is bounded by
\[
\tfrac{1}{2} A \sum_{n_1 \cdots n_k \le H} \mu^2(n_1 \cdots n_k) g_1(n_1) \cdots g_k(n_k).
\]
Note that the claim implies the proposition as the good grids can be directly handled via our assumption \eqref{eq:good_assu}.

First consider the okay grids that are not good because $S_{\text{sm}}$ is too large. Taking $\omega_1(n)$ to be the number of prime divisors of $n$ smaller than $\alpha_0(H)$, the subsum of \eqref{eq:gridding_prop} of $(n_1, \dots, n_k)$ coming from such a grid is at most
\[
A \sum_{\substack{n_1 \cdots n_k \le H \\ \omega_1(n_1 \cdots n_k) \ge m}} C_3^{\omega_1(n_1 \cdots n_k)} g_1(n_1) \cdots g_k(n_k)
\]
by our assumption on okay grids \eqref{eq:okay_assu}, where $m := \left(\log^{(2)} H\right)^{1/3 + 1/100}$. This is no larger than 
\begin{align*}
&Ae^{-m} \sum_{n_1 \cdots n_k \le H} (eC_3)^{\omega_1(n_1 \cdots n_k)} \mu^2(n_1 \cdots n_k) g_1(n_1) \cdots g_k(n_k) \\
&\le Ae^{-m} \left(\prod_{p \le \alpha_0(H)} \left(1 + ekC_2C_3p^{-1}\right) \right) \sum_{n_1 \cdots n_k \le H} \mu^2(n_1 \cdots n_k) g_1(n_1) \cdots g_k(n_k).
\end{align*}
The product in this expression is at most
\[
\left(10 \log \alpha_0(H)\right)^{ekC_2C_3}
\]
for $H$ greater than some absolute constant, and this is less than $\tfrac{1}{4}e^m$ for $\log^{(3)} H \gg 1 + \log(k C_2 C_3)$, where the implicit constant is absolute. This implies that the subsum with $S_{\text{sm}}$ too large is negligible. A similar argument handles $S_{\text{med}}$.

This leaves the final criterion. By condition 4 and our assumption \eqref{eq:okay_assu}, we wish to show that 
\[
C_3^{(\log^{(2)} H)^{\frac{1}{3} + \frac{1}{100}}} \sum_{\substack{n_1 \cdots n_k \le H \\ \mu^2(n_1 \cdots n_k) = 1 \\ \text{5 is not satisfied}}} g_1(n_1) \cdots g_k(n_k) \le \frac{1}{4} \sum_{\substack{n_1 \cdots n_k \le H \\ \mu^2(n_1 \cdots n_k) = 1}} g_1(n_1) \cdots g_k(n_k).
\]
Fix some $j$ such that $\sum_{p \le H} g_j(p) \ge c_1 \log \log H$. Take $g'_j$ to be the multiplicative function supported on squarefrees given by
\[
g'_j(p) = 
\begin{cases} 
\tfrac{1}{3} g_j(p) \quad &\text{ if } \sum_{q \in I(p)} g_j(q) \ge \sum_{q \in I(p)} \tfrac{1}{3}c_1 q^{-1} \\ 
g_j(p) &\text{ otherwise}
\end{cases}  
\]
on primes, where $I(p)$ is the interval of the form \eqref{eq:grid_axis} containing $p$. Then we claim that for a good choice of the absolute constant $C > 0$ and for all $H \geq \exp^{(2)} \left(C(C_2/c_1)^2\right)$
\begin{equation}
\label{eq:gjprime}
\sum_{p \le H} g'_j(p) \le \frac{4}{5} \sum_{p \le H} g_j(p).
\end{equation}
In order to prove this, let $\mathcal{J}$ be the collection of intervals of the form \eqref{eq:grid_axis} with supremum no larger than $H$. Define $\mathcal{J}_{j, \text{sm}}$ the subcollection for which
$$
\sum_{q \in I(p)} g_j(q) < \sum_{q \in I(p)} \tfrac{1}{3}c_1 q^{-1}
$$
and define $\mathcal{J}_{j, \text{la}}$ to be its complement in $\mathcal{J}$. We will now prove that
\begin{equation}
\label{eq:gjprime2}
\sum_{I \in \mathcal{J}} \sum_{p \in I} g'_j(p) \leq \frac{3}{4} \sum_{I \in \mathcal{J}} \sum_{p \in I} g_j(p).
\end{equation}
Note that \eqref{eq:gjprime2} easily implies \eqref{eq:gjprime} thanks to the bound $g_j(p) \leq C_2/p$ and our assumption $H \geq \exp^{(2)} \left(C(C_2/c_1)^2\right)$. We now observe that for large enough $H$
\begin{align*}
\sum_{I \in \mathcal{J}_{j, \text{sm}}} \sum_{p \in I} g'_j(p) \leq \frac{c_1}{3} \log \log H \leq \frac{1}{3} \sum_{p \leq H} g_j(p) \leq \frac{5}{12} \sum_{I \in \mathcal{J}_{j, \text{sm}}} \sum_{p \in I} g_j(p) + \frac{5}{12} \sum_{I \in \mathcal{J}_{j, \text{la}}} \sum_{p \in I} g_j(p)
\end{align*}
and that
$$
\sum_{I \in \mathcal{J}_{j, \text{la}}} \sum_{p \in I} g'_j(p) = \frac{1}{3} \sum_{I \in \mathcal{J}_{j, \text{la}}} \sum_{p \in I} g_j(p).
$$
These easily combine to give equation \eqref{eq:gjprime2}.

For convenience, suppose $j = 1$ fails condition 5. The sum of $g_1(n_1) \cdots g_k(n_k)$ over tuples coming from okay grids not satisfying condition 5 with $j = 1$ is bounded by
\[
C_3^{(\log^{(2)} H)^{\frac{1}{3} + \frac{1}{100}}} 3^{(\log^{(2)} H)^{\frac{5}{6}}} \sum_{n_1 \cdots n_k \le H} \mu^2(n_1 \cdots n_k) g_1'(n_1) \cdots g_k(n_k),
\]
The inner sum is at most
\[
\prod_{p \le H} (1 + g'_1(p) + \dots + g_k(p)) \le \prod_{p \le H} \left(1 + \frac{g'_1(p) - g_1(p)}{1 + kC_2p^{-1}}\right) \prod_{p \le H} (1 + g_1(p) + \dots + g_k(p)).
\]
By equation \eqref{eq:gjprime} and the bound $\sum_{p \le H} g_1(p) \geq c_1 \log \log H$, we can show that the first product is at most $e^{C k C_2} (\log H)^{-c_1/5}$. We use Lemma \ref{lNT} with $\delta = 1$ and $x, z$ both equal to our $H$ and $F(n) = \mu^2(n) \sum_{n_1 \cdots n_k = n} g_1(n_1) \cdots g_k(n_k)$ to deduce that
$$
\prod_{p \le H} (1 + g_1(p) + \dots + g_k(p)) \leq (1 + (CkC_2)^{kC_2} \exp(2ekC_2 - 1)) \sum_{\substack{n_1 \cdots n_k \le H \\ \mu^2(n_1 \cdots n_k) = 1}} g_1(n_1) \cdots g_k(n_k).
$$
Using the last assumption on the size of $H$, we can then compute that
$$
\frac{C_3^{(\log^{(2)} H)^{\frac{1}{3} + \frac{1}{100}}} 3^{(\log^{(2)} H)^{\frac{5}{6}}} (1 + (CkC_2)^{kC_2} \exp(2ekC_2 - 1))}{e^{-C k C_2} (\log H)^{c_1/5}} \leq \frac{1}{4k}.
$$
Summing over $1 \leq j \leq k$ finishes the proof. 
\end{proof}
\section{Character sums}
\label{sChar}
Throughout this section we will rely on the theory developed in \cite[Section 3]{Smi22a}. In particular, we shall make extensive use of starting tuples, classes, spins, symbols and ramification sections.

Fix a function 
\[
\lambda: \{\text{Rational primes}\} \to \{\text{Primes of } \overline{\QQ}\}
\]
such that $\lambda(p) \cap \QQ = (p)$ for every rational prime $p$. We will use the notation $G_p$ for the decomposition group in $G_{\QQ}$ associated to $\lambda(p)$, with $\text{Frob}\, p$ notation for the associated Frobenius element, which is well-defined up to an element in the inertia group $I_p$.

Fix a finite $G_{\Q}$-module $M$ as in the statement of Theorem \ref{thm:main}. Take $K$ to be the minimal Galois extension of $\QQ$ so that $M$ is a $\Gal(K/\QQ)$-module and so that $K$ contains the $e_0^{th}$ roots of unity, where $e_0$ is the exponent of $M$. Take $\Vplac_0$ to be a set of rational places so that $(K/\QQ, \Vplac_0, e_0)$ is an unpacked starting tuple. 

If $p$ is a rational prime outside $\Vplac_0$, we have an exact sequence
\[
0 \to H^1_{\text{ur}}(G_p, M) \xrightarrow{\,\text{inf}\,} H^1(G_p, M) \xrightarrow{\mathfrak{R}_{\lambda(p)}} M(-1)^{G_p} \to 0,
\]
where $\mathfrak{R}_{\lambda(p)}$ is the ramification-measuring homomorphism \cite[Definition 3.4]{Smi22a}. From the starting tuple $(K/\QQ, \Vplac_0, e_0)$, we can define a ramification section \cite[Definition 3.10]{Smi22a}
\[
\mathfrak{B}_{\lambda(p)}: M(-1)^{G_p} \to H^1(G_{\QQ}, M),
\]
which is a section for the ramification-measuring homomorphism $\mathfrak{R}_{\lambda(p)}$. The image of the ramification section consists of cocycle classes ramified only at primes in $\Vplac_0 \cup \{p\}$. 

Going forward, we will feel free to omit the $\lambda$ from the notation $\mathfrak{B}_{\lambda(p)}$ and $\mathfrak{R}_{\lambda(p)}$.

We have another ramification-measuring homomorphism
\[
\mathfrak{R}_p^{\vee}: H^1\left(G_p, M^{\vee}\right) \to M^{\vee}(-1)^{G_p} = (M^*)^{G_p}.
\]
Given a subgroup $\mathscr{L}$ of $H^1(G_p, M)$, we take
\begin{align*}
&A_{\mathscr{L}} = \mathfrak{R}_p(\mathscr{L}) \subseteq M(-1)^{G_p}\quad\text{and}\\
&R_{\mathscr{L}} = \mathfrak{R}_p^{\vee}\big(\mathscr{L}^{\perp}\big) \subseteq (M^*)^{G_p},
\end{align*}
and we define a bilinear pairing $\Omega_{\mathscr{L}} : A_{\mathscr{L}} \times R_{\mathscr{L}} \to \QQ/\Z$ by the formula
\[
\Omega_{\mathscr{L}}(a, r) = r\big((w - \mfB_p(a))(\text{Frob}\, p)\big),
\]
where $w \in \mathscr{L}$ is chosen so $\mathfrak{R}_p(w) = a$. 

Given primes $p_1, p_2$ outside $\Vplac_0$ and subgroups $\mathscr{L}_i \subseteq H^1(G_{p_i}, M)$ for $i = 1, 2$ with $\mathscr{L}_i \neq H^1_{\text{ur}}(G_{p_i}, M)$, we say that $(p_1, \mathscr{L}_1)$ and $(p_2, \mathscr{L}_2)$ are \emph{strongly equivalent} if $\lambda(p_1)$ and $\lambda(p_2)$ are in the same class with respect to the starting tuple, have the same spin, and
\[
\left( A_{\mathscr{L}_1}, \,R_{\mathscr{L}_1}, \,\Omega_{\mathscr{L}_1}\right)  = \left( A_{\msL_2},\, R_{\msL_2}, \,\Omega_{\msL_2}\right).\]
We enumerate these equivalence classes $2, \dots, k$. We reserve the first equivalence class for tuples $(p, \mathscr{L})$ where either $p$ is in $\Vplac_0$ or $\mathscr{L}$ is the empty set. For $i \ge 2$, write $\left(C^i, A^i, R^i, \Omega^i\right)$ for the tuple corresponding to the $i^{th}$ equivalence class, where $C^i$ is the subset of rational primes $p$ with the prescribed class and spin.

We note that, for a fixed prime $p \not \in \Vplac_0$, the subgroup $\msL$ may be recovered from the data $\left(A_{\msL}, R_{\msL}, \Omega_{\msL}\right)$. In particular, we may define a multiplicative function $g_i$ which is given on the primes by
\begin{equation}
\label{eKeyPipeline}
g_i(p) = 
\begin{cases} 
\tilde{g}_p(\msL)^\nu \mu_p(\msL) (1 - \mathbf{1}_{p > 2C_1^2} \cdot m(p))^{-1} &\text{ if there is  } \msL \text{ with } (p, \msL) \text{ in class } i \\ 
0 &\text{ otherwise,}
\end{cases}
\end{equation}
where $\tilde{g}_p$ and $\mu_p$ are as in Definition \ref{defn:effective-constant} and $\nu$ is as in Theorem \ref{thm:main}.

Given positive integers $n_1, \dots, n_k$ whose product is squarefree, and such that $n_i$ is a product of primes in $C^i$, we then define a Selmer group
\[
\Sel \,M(n_1, \dots, n_k) = \ker\left(H^1(G_{\QQ}, M) \to \prod_{v} H^1(G_v, M)/\msL_v\right),
\]
where $\msL_v = H^1(G_v, M)$ if $v$ is in $\Vplac_0$ or divides $n_1$, where $\msL_v$ is the subgroup associated to the tuple $(A^i, R^i, \Omega^i)$ if $v$ divides $n_i$ for some $i \ge 2$, and where $\msL_v$ otherwise equals the unramified local conditions.

Define 
\begin{align*}
&\text{up}(g) = \max\left(2, \sup_{p, i} \frac{g_i(p)}{p}\right), \\
&\text{low}(g) = \sup_{(p, \msL)} \tilde{g}_p(\msL)^{-\kappa}, \\
&\text{sup}(g) = \text{up}(g) \cdot \text{low}(g).
\end{align*}
Then $\text{sup}(g) \geq 2$ is a real number by Definition \ref{defn:effective-constant}. Finally, we recall that we have defined a constant $c > 0$ in Definition \ref{defn:effective-constant}.

\begin{theorem}
\label{thm:charsum}
There exists $C > 0$ depending only on the starting tuple $(K/\Q, \Vplac_0, e_0)$, $c$ and $k$ such that the following holds. Choose a grid $(r, (X_i)_{i \le r}, (k_i)_{i \le r})$ of height $H$. Take $X = \prod_{i \le r} X_i$.

Suppose $X$ is an okay grid. Then, for $\log^{(3)} H \geq C \log^{(3)} (C |M|)$ 
\begin{align*}
\sum_{x \in X} g(\bar{n}(x)) \cdot |\Sel \,M(\bar{n}(x))| \le \textup{sup}(g)^C |M|^{2 |\Ssm|} \exp\left(C \left(\log |M|\right)^2\right) \mcTbnd(M) \sum_{x \in X} g(\bar{n}(x)).
\end{align*}
 If $X$ is a $\frac{c}{2k \cdot \textup{low}(g)}$-good grid, we have for $H \geq \exp^{(2)}\big((\log C|M|)^C\big)$
the bound
\begin{align*}
\sum_{x \in X} g(\bar{n}(x)) \cdot |\Sel \,M(\bar{n}(x))| \le(C \cdot \textup{sup}(g))^{2 \log_2 |M|} \exp\left(C \left(\log |M|\right)^2\right) \mcTbnd(M) \sum_{x \in X} g(\bar{n}(x)).
\end{align*}
\end{theorem}

Note that the right hand side of Theorem \ref{thm:charsum} makes sense, as the Tamagawa bound $\mcTbnd(M)$ is determined by $k_1, \dots, k_r$ (i.e. is independent of $x \in X$ provided that $g(\bar{n}(x)) \neq 0$). 

\subsection{Some local computations}
We shall make repeated use of the following key facts. Let $\msL \subseteq H^1(G_p, M)$ be a subgroup. If we have a surjection $\pi: M \rightarrow U$, then we recall that $U$ is endowed with the pushforward local conditions $\pi(\msL)$, see \cite[Section 4]{MSFinite}.

\begin{lemma}
\label{lSubquotient}
Let $U \subseteq V \subseteq M$ be submodules. Then the two ways of endowing local conditions on the subquotient $V/U$ coincide (namely first pullback and then pushforward or first pushforward and then pullback).
\end{lemma}

\begin{proof}
See \cite[p.~20]{MSFinite}.
\end{proof}

Our next result is essentially by definition.

\begin{lemma}
\label{lFormula}
Let $v \not \in \Vplac_0$ and let $\msL \subseteq H^1(G_v, M)$. Then we have
$$
\mcT_v(M) = \frac{|A_\msL|}{|R_\msL|}.
$$
\end{lemma}

\begin{proof}
Combining the identities 
\begin{gather*}
|H^0(G_v, M)| = |H^1_{\text{ur}}(G_v, M)|, \quad |R_\msL| \cdot |H^1_{\text{ur}}(G_v, M) \cap \msL| = |H^1_{\text{ur}}(G_v, M)|, \\
|\msL| = |A_\msL| \cdot |H^1_{\text{ur}}(G_v, M) \cap \msL|
\end{gather*}
with the definition $\mcT_v(M) = |\msL|/|H^0(G_v, M)|$ gives the lemma.
\end{proof}

\begin{lemma}
\label{lTamaSES}
Let $v$ be a place of $\Q$ and let
$$
0 \rightarrow T \xrightarrow{\iota} M \xrightarrow{\pi} U \rightarrow 0
$$
be a short exact sequence of $G_v$-modules. Then
$$
\mathcal{T}_v(M) = \mathcal{T}_v(T) \cdot \mcT_v(U).
$$
\end{lemma}

\begin{proof}
By the long exact sequence we get
$$
0 \rightarrow H^0(G_v, T) \rightarrow H^0(G_v, M) \rightarrow H^0(G_v, U) \rightarrow \iota^{-1}(\msL) \rightarrow \msL \rightarrow \pi(\msL) \rightarrow 0,
$$
which immediately implies the result.
\end{proof}

Take $T$ to be a submodule of $M$. Recall that we endow $T$ with the local conditions $\msL'$ given by the preimage of $\msL$ under the map $H^1(G_p, T) \to H^1(G_p, M)$. We will compute this preimage explicitly in our next lemma. We have a perfect pairing
$$
\frac{M}{(\Frob \, p - 1)M} \times (M^\ast)^{G_p} \rightarrow \Q/\Z.
$$
We write $T^\perp \subseteq (M^\ast)^{G_p}$ for the orthogonal complement of $T$ under this pairing; here we view $T$ as a subset of $\frac{M}{(\Frob \, p - 1)M}$ via $\iota$ and the canonical quotient map. Via the identification
\begin{equation}
\label{eUrIdentification}
\frac{M}{(\Frob \, p - 1)M} \cong H^1_{\text{ur}}(G_p, M)
\end{equation}
we see that $T^\perp$ is exactly the orthogonal complement of $\iota(H^1_{\text{ur}}(G_p, T))$.

\begin{lemma}
\label{lPullBack}
Let $p \not \in \Vplac_0$ and let $\iota: T \rightarrow M$ be an inclusion. Then the local conditions $\msL' := \iota^{-1}(\msL)$ on $T$ are associated to the groups
\begin{align}
A_{\msL'} &= \{a \in T(-1)^{G_p} \,:\, \iota(a) \in A_{\msL}, \, \Omega_{\msL}(\iota(a), r) = 0 \textup{ for all } r \in T^{\perp} \cap R_{\msL}\} \qquad\textup{and} \label{eASub} \\
R_{\msL'} &= R_{\msL}/(T^{\perp} \cap R_{\msL}). \label{eRSub}
\end{align}
Here equation equation \eqref{eRSub} means that the natural map $\iota^\ast$ from $(M^\ast)^{G_p}$ to $(T^\ast)^{G_p}$, once restricted to $R_{\msL}$, is surjective with image $R_{\msL'}$ and has kernel $T^\perp \cap R_{\msL}$.
\end{lemma}

\begin{proof}
We will first prove \eqref{eRSub}. Under the pairing given by local Tate duality, $H^1_{\text{ur}}(G_p, M)$ and $H^1_{\text{ur}}(G_p, M^\vee)$ annihilate each other for $p \not \in \Vplac_0$. Hence $R_{\msL}$ is exactly the orthogonal complement of $H^1_{\text{ur}}(G_p, M) \cap \msL$ under the natural pairing
$$
H^1_{\text{ur}}(G_p, M) \times (M^\ast)^{G_p} \to \QQ/\Z
$$
given by evaluation at Frobenius. Consider the diagram
$$
\begin{tikzcd}[column sep=0.1in]
H^1_{\text{ur}}(G_p, T) \ar[d, "\iota"] & \times & (T^\ast)^{G_p} & \to & \Q/\Z \ar[d, "="] \\
H^1_{\text{ur}}(G_p, M) & \times & (M^\ast)^{G_p} \ar[u, "\iota^*"] & \to & \Q/\Z
\end{tikzcd}
$$
which gives $\langle \iota(t), m \rangle = \langle t, \iota^\ast(m) \rangle$. The diagram shows that the kernel of $\iota^\ast: (M^\ast)^{G_p} \rightarrow (T^\ast)^{G_p}$ is $T^{\perp}$. Thus, we need to show that the orthogonal complement $R_{\msL'}$ of $H^1_{\text{ur}}(G_p, T) \cap \iota^{-1}(\msL)$ under the pairing
$$
H^1_{\text{ur}}(G_p, T) \times (T^\ast)^{G_p} \to \QQ/\Z
$$
is exactly $\iota^\ast(R_{\msL})$. Using the diagram again, we find that $\iota^\ast(R_{\msL})^\perp = \iota^{-1}(R_{\msL}^\perp)$. But we also know that 
$$
\iota^{-1}(R_{\msL}^\perp) = \iota^{-1}(H^1_{\text{ur}}(G_p, M) \cap \msL) = H^1_{\text{ur}}(G_p, T) \cap \iota^{-1}(\msL) =  (R_{\msL'})^\perp,
$$
where the middle equality relies on the fact that $I_p$ acts trivially on $M$ and where the other two equalities are by definition. This shows that $\iota^\ast(R_{\msL})^\perp = (R_{\msL'})^\perp$ and hence $\iota^\ast(R_{\msL}) = R_{\msL'}$.

It remains to prove the first part of the lemma. We will first show $\subseteq$ in equation \eqref{eASub}. To this end, take some $a \in A_{\msL'}$. By definition, this means that there is some cocycle $\phi \in \iota^{-1}(\msL)$ such that $\mathfrak{R}_p(\phi) = a$. Since the ramification-measuring homomorphism commutes with module maps, it is then clear that $a \in T(-1)$ and $\iota(a) \in A_{\msL}$, so it remains to show that
$$
\Omega_{\msL}(\iota(a), r) = 0 \quad \text{ for all } r \in T^{\perp} \cap R_{\msL}.
$$
But we have by definition
$$
\Omega_{\mathscr{L}}(\iota(a), r) = r\big((\iota(\phi) - \mfB_p(\iota(a)))(\text{Frob}\, p)\big).
$$
By naturality of ramification sections \cite[Definition 3.10 (2)]{Smi22a}, we have $\mfB_p(\iota(a)) = \iota(\mfB_p(a))$ and hence $\iota(\phi) - \mfB_p(\iota(a)) \in \iota(H^1_{\text{ur}}(G_p, T))$. This gives the vanishing of $\Omega_{\mathscr{L}}(\iota(a), r) = 0$ for all $r \in T^{\perp} \cap R_{\msL}$.

Finally, we have to show $\supseteq$ in equation \eqref{eASub}. We start by considering the diagram
\begin{equation}
\label{eRamDiagram}
\begin{tikzcd}[column sep=0.3in]
H^1(G_p, T) \ar[d, "\mathfrak{R}_{p, T}", two heads] \ar[r, "\iota"] & H^1(G_p, M) \ar[d, "\mathfrak{R}_{p, M}", two heads] \\
T(-1)^{G_p} \arrow[r, hook, "\iota"] & M(-1)^{G_p} 
\end{tikzcd}
\end{equation}
and we take $a' \in M(-1)^{G_p}$ that is in the image of $\iota$ (say $\iota(a) = a'$), lies in $\mathfrak{R}_{p, M}(\msL)$, and moreover satisfies $\Omega_{\msL}(a', r) = 0$ for all $r \in T^{\perp} \cap R_{\msL}$. We must then find a cocycle $\phi \in \iota^{-1}(\msL)$ such that $\mathfrak{R}_{p, T}(\phi) = a$. 

First, take some cocycle $\psi \in \msL$ such that $\mathfrak{R}_{p, M}(\psi) = a'$ and take some cocycle $z \in H^1(G_p, T)$ with $\mathfrak{R}_{p, T}(z) = a$. We then compute
\begin{equation}
\label{eOmegaCompute}
0 = \Omega_{\mathscr{L}}(\mathfrak{R}_{p, M}(\psi), r) = r\big((\psi - \mfB_p(\iota(\mathfrak{R}_{p, T}(z))))(\text{Frob}\, p)\big) \quad \text{ for all } r \in T^{\perp} \cap R_{\msL}.
\end{equation}
By the identification \eqref{eUrIdentification} and the text immediately afterwards, $T^\perp$ is exactly the orthogonal complement of $\iota(H^1_{\text{ur}}(G_p, T))$. Hence the orthogonal complement of $T^{\perp} \cap R_{\msL}$ under the pairing $(M^\ast)^{G_p} \times H^1_{\text{ur}}(G_p, M) \rightarrow \Q/\Z$ is exactly $\msL \cap H^1_{\text{ur}}(G_p, M) + \iota(H^1_{\text{ur}}(G_p, T))$. Thus, equation \eqref{eOmegaCompute} yields
$$
\psi - \mfB_p(\iota(\mathfrak{R}_{p, T}(z))) = m + \iota(t),
$$
for some $m \in \msL \cap H^1_{\text{ur}}(G_p, M)$ and $t \in H^1_{\text{ur}}(G_p, T)$. Rearranging then shows that
$$
\psi - m = \iota(t) + \mfB_p(\iota(\mathfrak{R}_{p, T}(z))) = \iota(t + \mfB_p(\mathfrak{R}_{p, T}(z))).
$$
We now take $\phi := t + \mfB_p(\mathfrak{R}_{p, T}(z))$. Then $\iota(\phi)$ maps to $\psi - m \in \msL$, and hence $\phi \in \iota^{-1}(\msL)$ and moreover $\mathfrak{R}_{p, M}(\iota(\phi)) = a'$. But this forces $\mathfrak{R}_{p, T}(\phi) = a$ by injectivity of the bottom row of \eqref{eRamDiagram}.
\end{proof}

Meanwhile, we may endow $M/T$ with the local conditions $\msL^{''}$ coming from the image of $\msL$ under the map $H^1(G_p, M) \to H^1(G_p, M/T)$. The corresponding groups take the form
\begin{align*}
A_{\msL^{''}} &=  A_{\msL}/(T(-1) \cap A_{\msL})\qquad\text{and}
\\
R_{\msL^{''}} & = \big\{r \in T^{\perp} \cap R_{\msL}\,:\,\, \Omega_{\msL}(a, r) = 0 \text{ for all } a \in T(-1) \cap A_{\msL}\big\}.
\end{align*}
Indeed, these operations are dual to each other under the operator $M \mapsto M^{\vee}$, so the above formulae follow from Lemma \ref{lPullBack}. More precisely, taking $\msL^{\vee}$ to be the corresponding local conditions, we have
\begin{align*}
A_{\msL^{\vee}} &= R_{\msL} \quad\text{in} \quad M^{\vee} (-1)^{G_p} = (M^*)^{G_p}\qquad\text{and }\\
R_{\msL^{\vee}} &= A_{\msL} \quad\text{in} \quad \left((M^{\vee})^*\right)^{G_p} = M(-1)^{G_p}.
\end{align*}
Our next lemma justifies the terminology weak and strong equivalence.

\begin{lemma}
\label{lStrongWeak}
Let $(p_1, \msL_1)$ and $(p_2, \msL_2)$ be strongly equivalent. Then $(p_1, \msL_1)$ and $(p_2, \msL_2)$ are weakly equivalent.
\end{lemma}

\begin{proof}
We recall the definitions of $\mathrm{Tine}_{\Q} \, \lambda(p)$ and $\mathrm{Frob}_{\Q} \, \lambda(p)$ in \cite[Section 2]{Smi22a}. By abuse of notation, we will also write $\mathrm{Tine}_{\Q} \, \lambda(p)$ and $\mathrm{Frob}_{\Q} \, \lambda(p)$ for their images inside $G_p/mI_p$. Define
$$
N := \mathrm{Map}(M(-1), M).
$$
We view $N$ as a Galois module by sending $\sigma \in G_{\Q}$ and $f \in \mathrm{Map}(M(-1), M) = N$ to the map $a \mapsto \sigma(f(a))$. This makes $N$ isomorphic to just $|M|$ copies of $M$, indexed by $M(-1)$. Since $p_1$ and $p_2$ have the same spin by strong equivalence, we get from \cite[Proposition 3.20]{Smi22a}, applied with $N$ and with $m \in N(-1)$ the map sending $\zeta$ to $a \mapsto a(\zeta)$, the equality
$$
\big(\mfB_{p_1}(a)(\mathrm{Frob}_{\Q} \, \lambda(p_1))\big)_{a \in M(-1)} = \big(\mfB_{p_2}(a)(\mathrm{Frob}_{\Q} \, \lambda(p_2))\big)_{a \in M(-1)} \quad \text{ in } \, \, \frac{N}{\text{im } m}.
$$
This means precisely that we can choose $\mathrm{Frob}_{\Q} \, \lambda(p_1)$, by changing $\mathrm{Frob}_{\Q} \, \lambda(p_1)$ by a power of $\mathrm{Tine}_{\Q} \, \lambda(p_1)$ if necessary, in such a way that
\begin{equation}
\label{eLiftTrick}
\big(\mfB_{p_1}(a)(\mathrm{Frob}_{\Q} \, \lambda(p_1))\big)_{a \in M(-1)} = \big(\mfB_{p_2}(a)(\mathrm{Frob}_{\Q} \, \lambda(p_2))\big)_{a \in M(-1)},
\end{equation}
and we shall henceforth work with this particular choice of $\mathrm{Frob}_{\Q} \, \lambda(p_1)$.

Since $p_1, p_2 \not \in \Vplac_0$ and since $\lambda(p_1)$ and $\lambda(p_2)$ share the same class, the natural map $\iota: G_{p_1}/mI_{p_1} \rightarrow G_{p_2}/mI_{p_2}$ sending
$$
\mathrm{Tine}_{\Q} \, \lambda(p_1) \mapsto \mathrm{Tine}_{\Q} \, \lambda(p_2), \quad \quad \mathrm{Frob}_{\Q} \, \lambda(p_1) \mapsto \mathrm{Frob}_{\Q} \, \lambda(p_2)
$$
extends uniquely to a continuous group isomorphism. Moreover, since $p_1, p_2 \not \in \Vplac_0$, since $\lambda(p_1)$ and $\lambda(p_2)$ share the same class, and since $\mathrm{Frob}_{\Q} \, \lambda(p_1)$ maps to $\mathrm{Frob}_{\Q} \, \lambda(p_2)$ under $\iota$, we also see that $\tau m = \iota(\tau) m$ for all $\tau \in G_{p_1}/mI_{p_1}$ and all $m \in M$. We can therefore take $\sigma = \mathrm{id}$ in the definition of weak equivalence.

Denote by $\inf_i: H^1(G_{p_i}/mI_{p_i}, M) \xrightarrow{\sim} H^1(G_{p_i}, M)$ the inflation map. We also write $\msL_i' := \inf_i^{-1}(\msL_i)$. It remains to prove that $\iota$ identifies $\msL_2'$ with $\msL_1'$. 

Since $\iota$ is an isomorphism, it suffices to show that $\phi \in \msL_2'$ implies $\iota(\phi) \in \msL_1'$. To this end, take $\phi \in \msL_2'$. To start, we unwind the definitions of $\mathfrak{R}_{\lambda(p_i)}$ and $\iota$ to conclude that the diagram
\begin{equation*}
\begin{tikzcd}[column sep=0.5in]
H^1(G_{p_2}/mI_{p_2}, M) \ar[d, "\iota"] \ar[r, "\inf_2"] & H^1(G_{p_2}, M) \ar[r, "\mathfrak{R}_{\lambda(p_2)}"] & M(-1)^{G_{p_2}} \ar[d, "="] \\
H^1(G_{p_1}/mI_{p_1}, M) \ar[r, "\inf_1"] & H^1(G_{p_1}, M) \arrow[r, "\mathfrak{R}_{\lambda(p_1)}"] & M(-1)^{G_{p_1}} 
\end{tikzcd}
\end{equation*}
commutes. Set
$$
a := \mathfrak{R}_{\lambda(p_2)}(\mathrm{inf}_2(\phi)) = \mathfrak{R}_{\lambda(p_1)}(\mathrm{inf}_1(\iota(\phi))).
$$
Since $\phi \in \msL_2'$, we must certainly have that $a \in A_{\msL_2}$, thus $a \in A_{\msL_2} = A_{\msL_1}$ by strong equivalence. Take $w_1 \in \msL_1'$ mapping to $a$. Using that $R_{\msL_1} = R_{\msL_2}$ and $\Omega_{\msL_1} = \Omega_{\msL_2}$ by strong equivalence, we get
\begin{align*}
r\big((\mathrm{inf}_2(\phi) - \mfB_{p_2}(a))(\mathrm{Frob}_{\Q} \, \lambda(p_2))\big) &= \Omega_{\msL_2}(a, r) = \Omega_{\msL_1}(a, r) \\
&= r\big((\mathrm{inf}_1(w_1) - \mfB_{p_1}(a))(\mathrm{Frob}_{\Q} \, \lambda(p_1))\big).
\end{align*}
Since $\iota$ sends $\mathrm{Frob}_{\Q} \, \lambda(p_1)$ to $\mathrm{Frob}_{\Q} \, \lambda(p_2)$, we get from equation \eqref{eLiftTrick}
$$
r\big((\mathrm{inf}_2(\phi) - \mfB_{p_2}(a))(\mathrm{Frob}_{\Q} \, \lambda(p_2))\big) = r\big((\mathrm{inf}_1(\iota(\phi)) - \mfB_{p_1}(a))(\mathrm{Frob}_{\Q} \, \lambda(p_1))\big)
$$
for all $r \in R_{\msL_1} = R_{\msL_2}$. Combining the last two equations, we conclude that
$$
r\big((\mathrm{inf}_1(\iota(\phi)) - \mathrm{inf}_1(w_1))(\mathrm{Frob}_{\Q} \, \lambda(p_1))\big) = 0
$$
for all $r \in R_{\msL_1} = R_{\msL_2}$, which implies $\iota(\phi) \in \msL_1'$.
\end{proof}

In some of our arguments, it will be important that the same equivalence class $k_i$ appears for many $1 \leq i \leq r$. Since Definition \ref{defn:effective-constant} involves weak equivalence instead of strong equivalence, we are only able to show an abundance of indices in the same ``weak equivalence class'', as we formulate more precisely in our next result. 

\begin{lemma}
\label{lPigeon}
Let $(r, (X_i)_{i \le r}, (k_i)_{i \le r})$ be a $\frac{c}{2k \cdot \textup{low}(g)}$-good grid of height $H \geq \exp^{(2)}(2/c)$. Let $1 \leq i \leq r$. Then there are at least $(\log^{(2)} H)^{5/6}$ choices of $j$ such that $k_i$ and $k_j$ are weakly equivalent, and moreover
$$
\sum_{p \in X_j} g_{k_j}(p) \geq \sum_{p \in X_j} \frac{c}{6k \cdot \textup{low}(g) \cdot p}.
$$
\end{lemma}

\begin{proof}
Define $\mathcal{S}$ to be the set of indices $j$ such that $k_i$ and $k_j$ are weakly equivalent. This is well-defined by Lemma \ref{lStrongWeak}. By Definition \ref{defn:goodgrid} (3), $k_i$ is not rare. Hence the hypothesis of Definition \ref{defn:effective-constant} (3) (``no moderately rare local conditions'') is satisfied. In particular, using that $\tilde{g}_p(\msL)^\kappa$ is lower bounded by $\text{low}(g)^{-1}$, we have 
$$
\sum_{j \in \mathcal{S}} \sum_{\substack{(p, \msL_p) \text{ strongly equivalent to } j \\ p \leq H}} g_j(p) \geq \frac{-1 + c \log \log H}{\text{low}(g)}.
$$
By the pigeonhole principle, this implies that the hypothesis of Definition \ref{defn:goodgrid} (5) is met for at least one strong equivalence class $k_j$ with $j \in \mathcal{S}$, and now the conclusion of Definition \ref{defn:goodgrid} (5) gives the lemma.
\end{proof}

\subsection{Proof of Theorem \ref{thm:charsum}: Initial reductions}
We start by defining, for each submodule $T$ of $M$, a modified set of local conditions on $M$.

\begin{mydef}
\label{dNewLocal}
Given a submodule $T \subseteq M$ and $x \in X$, define local conditions $(\msL'_{xv})_v$ for $M$ as follows:
\begin{itemize}
\item Given $v = \pi_s(x)$ for any $s \in \Smed \cup \Slg$, we take
\[
\msL'_{xv} = H^1_{\text{ur}}(G_v, T) + \ker\left(\msL_{xv} \,\xrightarrow{\quad}\, H^1(I_v, M/T)\right).
\]
\item At all other places $v$, we take $\msL'_{xv} = \msL_{xv}$.
\end{itemize}
\end{mydef}

Take $T^{\circ}$ to be the maximal submodule of $M$ such that the image of $\msL'_{xv}$ (or, equivalently, $\msL_{xv}$) in $H^1(G_v, M/T)$ contains the image of $H^1_{\text{ur}}(G_v, T^{\circ}/T)$ for all $v$ of the form $\pi_s(x)$ with $s \in \Smed \cup \Slg$. Since, for each fixed $s$, all elements of $\{(v, \msL_{xv}) : x \in X, v = \pi_s(x), g_{k_s}(v) \neq 0\}$ are strongly equivalent, we know that this definition does not depend on the choice of $x$ as long as $g(\bar{n}(x)) \neq 0$. Moreover, this is indeed well-defined by the following two facts:
\begin{itemize}
\item The image of $\msL'_{xv}$ in $H^1(G_v, M/T)$ contains the image of $H^1_{\text{ur}}(G_v, T/T)$.
\item If the image of $\msL'_{xv}$ in $H^1(G_v, M/T)$ contains the image of both $H^1_{\text{ur}}(G_v, U_1/T)$ and $H^1_{\text{ur}}(G_v, U_2/T)$, then it also contains the image of $H^1_{\text{ur}}(G_v, (U_1 + U_2)/T)$.
\end{itemize}

\begin{lemma}
\label{lMTUnr}
Let $v$ be of the form $\pi_s(x)$ with $s \in \Smed \cup \Slg$. The pushforward local conditions on $M/T$ (induced from $\msL'_x$) are contained inside $H^1_{\textup{ur}}(G_v, M/T)$. The local conditions on $T^\vee$, via pullback and local Tate duality, are also contained inside $H^1_{\textup{ur}}(G_v, T^\vee)$. Moreover, the subquotient local conditions on $T^\circ/T$ are exactly equal to $H^1_{\textup{ur}}(G_v, T^\circ/T)$.
\end{lemma}

\begin{proof}
By definition of $(\msL'_{xv})_v$, we see that the pushforward local conditions on $M/T$ are contained inside $H^1_{\text{ur}}(G_v, M/T)$ and that the pullback local conditions on $T$ contain $H^1_{\text{ur}}(G_v, T)$, hence the resulting local conditions on $T^\vee$ are contained inside $H^1_{\textup{ur}}(G_v, T^\vee)$.

Since $I_v$ acts trivially on $M$, we have an injection $\Hom(I_v, T^\circ/T) \rightarrow \Hom(I_v, M/T)$. Hence the pullback of the unramified local conditions on $M/T$ are contained inside the unramified local conditions on $T^\circ/T$. Therefore Lemma \ref{lSubquotient} shows that the induced local conditions on $T^\circ/T$ are contained inside $H^1_{\text{ur}}(G_v, T^\circ/T)$ as well. Now the lemma is an immediate consequence of the definition of $T^\circ$.
\end{proof}

Similarly, take $T_{\circ}$ to be the minimal submodule of $T$ such that the preimage of $\msL'_{xv}$ (or, equivalently, $\msL_{xv}$) in $H^1(G_v, T)$ has trivial image in $H^1(I_v, T/T_{\circ})$ for all $v$ of the form $\pi_s(x)$ with $s \in \Smed \cup \Slg$. 

\begin{lemma}
\label{lChebSmall}
There exists $C > 0$ depending only on the starting tuple $(K/\Q, \Vplac_0, e_0)$, $c$ and $k$ such that the following holds.

Take $\pi: M \to M/T^{\circ}$ and $\iota: T_{\circ} \to M$ to be the standard projection and inclusion, respectively. Let $X$ be a $\frac{c}{2k \cdot \textup{low}(g)}$-good grid of height $H \geq \exp^{(2)}\big(\log(C|M|)^{C}\big)$. Define $X_{\textup{bad}}$ to be the subset of $x \in X$ satisfying
\[
\left|\Sel\, (M/T^{\circ}, \pi(\msL'_x))\right| \ge (C \cdot \textup{sup}(g))^{\log_2 |M|} \quad \textup{or} \quad \left|\Sel\, (T_{\circ}^{\vee}, \iota^{-1}(\msL'_x)^{\perp})\right| \ge (C \cdot \textup{sup}(g))^{\log_2 |M|}.
\]
Then we have
$$
\sum_{x \in X_{\textup{bad}}} g(\bar{n}(x)) \leq |M|^{(C + |\Ssm|)^2} \exp\left(|\Ssm \cup \Smed| - (\log^{(2)} H)^{5/6}\right) \sum_{x \in X} g(\bar{n}(x)).
$$
\end{lemma}

\begin{proof}
Take $\Vplac_{\text{sm}}$ to be the union of $\Vplac_0$ with the set of places of the form $\pi_s(x)$ with $s \in \Ssm$, and take $W$ to be a subgroup of 
\begin{equation}
\label{eWHypothesis}
W \subseteq \ker\left(H^1(G_{\QQ}, M/T^{\circ}) \to \prod_{v \not \in \Vplac_{\text{sm}}} H^1\left(I_v, M/T^{\circ}\right)\right).
\end{equation}
For each $s \in \Smed \cup \Slg$, we denote by $R_s''$ the subset of $(M/T^\circ)^\ast \subseteq M^\ast$ obtained from the pushforward local conditions $\pi(\msL'_{x \pi_s(x)})$ on $M/T^\circ$ (by strong equivalence, this is the same set independent of the choice of $x \in X$ and $v = \pi_s(x)$ as long as $g_{k_s}(v) \neq 0$). We claim that there exists some sequence $s_1, \dots, s_n \in \Smed \cup \Slg$ and corresponding elements $\tau_1, \dots, \tau_n \in \Gal(K/\Q)$ such that
\begin{equation}
\label{eMTSum}
\sum_{i = 1}^n \tau_iR_{s_i}'' = (M/T^\circ)^\ast.
\end{equation}
In order to prove the claim, we define $V_0$ to be the subgroup of $(M/T^\circ)^\ast$ generated by $\tau r_s$ with $\tau \in \Gal(K/\Q)$ and $r_s \in R_s''$ for some $s \in \Smed \cup \Slg$. Note that then $V_0$ is in fact a $\Gal(K/\Q)$-submodule of $(M/T^\circ)^\ast$. Since $(M/T^\circ)^\ast$ is finite, the desired existence follows once we have shown that $V_0 = (M/T^\circ)^\ast \subseteq M^\ast$, which we do now.

We claim that the local conditions on $V_0^\perp/T$ equal the full $H^1_{\text{ur}}(G_v, V_0^\perp/T)$. Note that $V_0^\perp$ certainly contains $T^\circ$. Hence, once we have established the claim, we deduce $V_0^\perp = T^\circ$ by maximality of $T^\circ$, so it remains to prove the claim. Consider the exact sequence
$$
0 \rightarrow T^\circ/T \rightarrow V_0^\perp/T \rightarrow V_0^\perp/T^\circ \rightarrow 0.
$$
Now all of these are subquotients of $M/T$, hence by Lemma \ref{lSubquotient} and Lemma \ref{lMTUnr}, their local conditions are all contained inside respectively $H^1_{\text{ur}}(G_v, T^\circ/T)$, $H^1_{\text{ur}}(G_v, V_0^\perp/T)$, and $H^1_{\text{ur}}(G_v, V_0^\perp/T^\circ)$. But for $T^\circ/T$ the local conditions are the full $H^1_{\text{ur}}(G_v, T^\circ/T)$ by definition of $T^\circ$, and so are the local conditions on $V_0^\perp/T^\circ$ by using the definition of $V_0$ and equation \eqref{eRSub}. From these facts, it follows that the local conditions on $V_0^\perp/T$ also equal the full $H^1_{\text{ur}}(G_v, V_0^\perp/T)$; indeed the size of the local conditions is now at least 
$$
|H^0(G_v, T^\circ/T)| \cdot |H^0(G_v, V_0^\perp/T^\circ)| \geq |H^0(G_v, V_0^\perp/T)|,
$$ 
while we also know that the local conditions are contained inside $H^1_{\text{ur}}(G_v, V_0^\perp/T)$ which has size $|H^0(G_v, V_0^\perp/T)|$.

By removing indices if necessary, we may assume that $n \leq \log_2 |M|$. Also take generators $r_1, \dots, r_n$ for respectively $R_{s_1}'', \dots, R_{s_n}''$. Rephrasing equation \eqref{eMTSum}, we conclude that the homomorphism $\varphi$ of abelian groups
\[
M/T^{\circ} \xhookrightarrow{\oplus_{\tau_i r_i}} \bigoplus_{1 \le i \le n} \Q/\Z
\]
is injective. Take $W_0$ to be the image of $W$ in $H^1(G_K, M/T^{\circ}) = \Hom(G_K, M/T^\circ)$. Therefore there exists $1 \leq i \leq n$ such that the image of $\tau_i r_i(W_0)$ inside $\Hom(G_K, \Q/\Z)$ is of size at least $|W_0|^{1/\log_2 |M|}$. 

Take $s \in \Slg$ such that $k_i$ and $k_s$ are weakly equivalent. This implies the existence of some $\sigma \in G_\Q$ such that $R_s = \sigma(R_i)$. Take $L$ to be the minimal subfield of $K$ so that $\lambda(p_s) \cap L$ is inert in $K/L$ for any $p_s$ taken from $\{p \in X_s : g_{k_s}(p) \neq 0\}$. If $B$ is the image of $W$ in $H^1(G_L, M/T^{\circ})$, we see that every element in $W$ satisfies the local condition at a given $p_s$ (i.e.~lies in $\pi(\mathscr{L}_{x p_s}')$) only if every element $b$ of $B$ satisfies $\sigma r_i(b(\text{Frob}\, \lambda(p_s))) = 0$. If $H \geq \exp^{(2)}(\log(C|M|)^C)$ for a constant $C > 0$ depending only on the starting tuple, then the local conditions are satisfied only if $\lambda(p_s) \cap L$ splits completely in a field $E/L$ satisfying 
\begin{equation}
\label{eEbounds}
[E : L] = |\sigma r_i(B)| \leq |M|^{C (\log^{(2)} H)^{\left(\frac{1}{3} + \frac{1}{100}\right)}}, \quad \quad |\Delta_E| \leq \exp^{(3)}\left(\left(\tfrac{1}{3} + \tfrac{2}{100}\right) \log^{(3)} H\right).
\end{equation}
Applying the effective Chebotarev density theorem \cite{LO}, as codified by Thorner--Zaman in \cite[Theorem 1.1]{TZ}, and using our bounds on the degree and discriminant of $E$ from equation \eqref{eEbounds}, there exists $C > 0$ such that 
\begin{equation}
\label{eCheb1}
\sum_{\substack{\alpha_0(H) \alpha(H)^m \leq p < \alpha_0(H) \alpha(H)^{m + 1} \\ \lambda(p) \cap L \text{ splits in } E}} \frac{1}{p} \leq \frac{C}{|\sigma r_i(B)|} \sum_{\alpha_0(H) \alpha(H)^m \leq p < \alpha_0(H) \alpha(H)^{m + 1}} \frac{1}{p}.
\end{equation}
In the above inequality, we have incorporated any potential contribution from a Siegel zero into the main term, which is possible as we are only searching for an upper bound instead of an asymptotic. Take $X'_s$ to be the set of $p_s \in X_s$ such that $\lambda(p_s) \cap L$ splits completely in the field $E/L$. It follows from \eqref{eCheb1} that
\begin{equation}
\label{eCheb2}
\sum_{p \in X'_s} g_{k_s}(p) \leq \hspace{-0.3cm} \sum_{\substack{\alpha_0(H) \alpha(H)^m \leq p < \alpha_0(H) \alpha(H)^{m + 1} \\ \lambda(p) \cap L \text{ splits in } E}} \hspace{-0.3cm} \frac{\text{up}(g)}{p} \leq \frac{C \cdot \text{up}(g)}{|\sigma r_i(B)|} \hspace{-0.2cm}  \sum_{\alpha_0(H) \alpha(H)^m \leq p < \alpha_0(H) \alpha(H)^{m + 1}} \frac{1}{p},
\end{equation}
where $m$ is the unique integer such that $X_s$ consists of all the primes in the half-open interval $[\alpha_0(H) \alpha(H)^m, \alpha_0(H) \alpha(H)^{m + 1})$. In order to make use of equation \eqref{eCheb2}, we will now give a lower bound for $|\sigma r_i(B)|$. Since $\sigma r_i(B)$ restricts to $\sigma r_i(W_0)$, it suffices to do this for $\sigma r_i(W_0)$. 

Now observe that we have for every $\tau \in G_\Q$ a commutative diagram
$$
\begin{tikzcd}
\Hom(G_K, M/T^\circ)^{\Gal(K/\Q)} \arrow[d] \arrow[dr] & \\
\Hom\left(G_K, M/(T^{\circ} + (r_i)^{\perp})\right) \arrow[r, "\cong"] & \Hom\left(G_K, M/(T^{\circ} + \tau((r_i)^{\perp}))\right),
\end{tikzcd}
$$
where the bottom map sends a homomorphism $\phi \in \Hom\left(G_K, M/(T^{\circ} + (r_i)^{\perp})\right)$ to the homomorphism $g \mapsto \tau \phi(\tau^{-1} g \tau)$. By inflation-restriction, we have $W_0 \subseteq \Hom(G_K, M/T^\circ)^{\Gal(K/\Q)}$. Hence the size of the kernel of the map
$$
W_0 \rightarrow \Hom\left(G_K, M/(T^{\circ} + \tau((r_i)^{\perp}))\right)
$$
does not depend on $\tau$. Since this is also the kernel of $\tau r_i: W_0 \rightarrow \Hom(G_K, \Q/\Z)$, we conclude that the size of the image of $\tau r_i(W_0)$ does not depend on $\tau$. Recalling that $\tau_i r_i(W_0)$ is of size at least $|W_0|^{1/\log_2 |M|}$, we deduce the same bound for $|\sigma r_i(W_0)|$ and hence $|\sigma r_i(B)|$. By the inflation-restriction exact sequence and the bound $|H^1(\Gal(K/\Q), M)| \leq |M|^{[K : \QQ]}$, we have 
\[
|W_0| \ge |W| \cdot |M|^{-[K : \QQ]}.
\]
Thus, if we have $|W| \geq |M|^{[K : \QQ]} \cdot \left(\frac{6e \cdot C \cdot k \cdot \text{sup}(g)}{c}\right)^{\log_2 |M|}$, we obtain
\begin{equation}
\label{eCheb3}
\frac{C \cdot \text{up}(g)}{|\sigma r_i(B)|} \sum_{\alpha_0(H) \alpha(H)^m \leq p < \alpha_0(H) \alpha(H)^{m + 1}} \frac{1}{p} \leq \frac{c}{6e \cdot k \cdot \text{low}(g)} \sum_{\alpha_0(H) \alpha(H)^m \leq p < \alpha_0(H) \alpha(H)^{m + 1}} \frac{1}{p}.
\end{equation}
Combining equations \eqref{eCheb2} and \eqref{eCheb3}, we deduce that for every $s \in \Slg$ such that $k_i$ and $k_s$ are weakly equivalent, we have 
$$
\sum_{p \in X'_s} g_{k_s}(p) \leq \frac{c}{6e \cdot k \cdot \text{low}(g)} \sum_{\alpha_0(H) \alpha(H)^m \leq p < \alpha_0(H) \alpha(H)^{m + 1}} \frac{1}{p}.
$$
Applying Lemma \ref{lPigeon} to the index $i$ and using the above bound for every $s \in \Slg$ such that $k_i$ and $k_s$ are weakly equivalent, we find that for $|W| \geq |M|^{[K : \QQ]} \cdot (6e \cdot C \cdot k \cdot \text{sup}(g)/c)^{\log_2 |M|}$
\[
\sum_{x \in X} g(\bar{n}(x)) \cdot \mathbf{1}_{W \subseteq \Sel(M/T^{\circ}, \pi(\msL'_x))} \le \exp\left(|\Ssm \cup \Smed| - (\log^{(2)} H)^{5/6}\right) \sum_{x \in X} g(\bar{n}(x)).
\] 
We now apply this theory to $W = \Sel\, (M/T^{\circ}, \pi(\msL'_x))$, which is contained inside \eqref{eWHypothesis} by Lemma \ref{lMTUnr}. Summing over all subspaces $W$ ends the proof in case $\left|\Sel\, (M/T^{\circ}, \pi(\msL'_x))\right| \ge (C' \cdot \textup{sup}(g))^{\log_2 |M|}$. The case $\left|\Sel\, (T_{\circ}^{\vee}, \iota^{-1}(\msL'_x)^{\perp})\right| \ge (C' \cdot \textup{sup}(g))^{\log_2 |M|}$ is similar.
\end{proof}

\begin{lemma}
\label{lTamaDichotomy}
There exists $C > 0$ depending only on the starting tuple $(K/\Q, \Vplac_0, e_0)$, $c$ and $k$ such that the following holds. Suppose that $X$ is a $\frac{c}{2k \cdot \textup{low}(g)}$-good grid of height $H \geq \exp^{(2)}\big((\log C|M|)^3\big)$. Suppose that
\begin{equation}
    \label{eq:Tcircnobetter}
\max\left(\mcT(T_{\circ}), \mcT(T^{\circ})\right) \le \exp\left(\frac{1}{2} (\log^{(2)} H)^{5/6}\right) \mcT(T),
\end{equation}
where the local conditions here correspond to $\msL_x$ for any $x \in X$ with $g(\bar{n}(x)) \neq 0$.

Then, for any such $x \in X$ and any place $v \not \in \Vplac_0$, the subquotient local conditions of
\[
H^1(G_v, T^{\circ}/T) \quad \textup{and} \quad H^1(G_v, T/T_{\circ}) 
\]
corresponding to $\msL_{xv}$ are precisely the unramified local conditions.
\end{lemma}

\begin{proof}
For all places $v$ and any $x \in X$ with $g(\bar{n}(x)) \neq 0$, if we endow all modules with the subquotient local conditions from $\msL_{xv}$, we have by Lemma \ref{lTamaSES}
\[
\mathcal{T}_v(T^{\circ}) = \mathcal{T}_v(T) \cdot \mcT_v(T^{\circ}/T).
\]
By our definition of $T^{\circ}$, the local conditions for $T^{\circ}/T$ corresponding to $\msL_{xv}$ contain the full $H^1_{\text{ur}}(G_v, T^{\circ}/T)$ for all $v$ outside $\Vplac_0$. So $\mcT_v(T^{\circ}/T)$ is at least $1$ at all such $v$.

If these local conditions properly contain $H^1_{\text{ur}}(G_v, T^{\circ}/T)$ for some $v$ outside $\Vplac_0$, then Lemma \ref{lPigeon} implies that this happens for at least $(\log^{(2)} H)^{5/6} - |\Ssm|$ choices of $v$ of the form $\pi_s(x)$ with $s \in \Smed \cup \Slg$. From this, we deduce that
\[
\mcT(T^{\circ}/T) \ge |M|^{-C' (|\Vplac_0| + |\Ssm|)} \cdot 2^{(\log^{(2)} H)^{5/6} - |\Ssm|}, 
\]
where $C' > 0$ depends only on the starting tuple. For a suitable choice of $C > 0$, this contradicts \eqref{eq:Tcircnobetter} for $H \geq \exp^{(2)}((\log C|M|)^3)$. This gives the lemma for $T^{\circ}/T$. A similar argument then works for $T/T_{\circ}$.
\end{proof}

\begin{proposition}
\label{pCheb}
There exists $C > 0$ depending only on the starting tuple $(K/\Q, \Vplac_0, e_0)$, $c$ and $k$ such that for all $\frac{c}{2k \cdot \textup{low}(g)}$-good grids $X$ of height $H \geq \exp^{(2)}\big((\log C|M|)^C\big)$
\begin{multline*}
\sum_{x \in X} g(\bar{n}(x)) \left|\Sel(M, (\msL_{xv}')_v)\right| \le (C \cdot \textup{sup}(g))^{2 \log_2 |M|} \cdot \max_{U \subseteq M} \mcT(U) \cdot \\
\left(\prod_{s \in \Smed \cup \Slg} |R_s/(R_s \cap T^{\perp})|\right) \cdot \sum_{x \in X} g(\bar{n}(x)).
\end{multline*}
\end{proposition}

\begin{proof}
Throughout our proof, all of our implicit local conditions will be the subquotient local conditions obtained from $(\msL_{xv}')_v$. We will proceed by showing a general upper bound for $\Sel(M)$, and will then derive the proposition from this bound. Taking Selmer groups gives an exact sequence
$$
\Sel(T) \rightarrow \Sel(M) \rightarrow \Sel(M/T).
$$
We use the Selmer group functor again and obtain the exact sequence
$$
\Sel(T^\circ/T) \rightarrow \Sel(M/T) \rightarrow \Sel(M/T^\circ).
$$
So far we have shown that
\begin{equation}
\label{eBound1}
|\Sel(M)| \leq |\Sel(T)| \cdot |\Sel(T^\circ/T)| \cdot |\Sel(M/T^\circ)|.
\end{equation}
We now apply the Greenberg--Wiles' formula \eqref{eq:Wiles} to incorporate $\Sel(T^\vee)$. Writing $\iota: T \rightarrow M$ for the natural inclusion, we see that $\iota^{-1}(\msL_{xv}') = \iota^{-1}(\msL_{xv}) + H^1_{\text{ur}}(G_v, T)$ for places $v$ of the shape $\pi_s(x)$ with $s \in \Smed \cup \Slg$. This readily implies
$$
A_{\iota^{-1}(\msL_{xv}')} = A_{\iota^{-1}(\msL_{xv})}, \quad \quad R_{\iota^{-1}(\msL_{xv}')} = 0.
$$
Thus Lemma \ref{lFormula} and equation \eqref{eRSub} of Lemma \ref{lPullBack} yield
$$
|\iota^{-1}(\msL_{xv}')| = |\iota^{-1}(\msL_{xv})| \cdot |R_s/(R_s \cap T^{\perp})|.
$$
Hence the Greenberg--Wiles' formula gives the bound 
\begin{equation}
\label{eBound2}
|\Sel(T)| \leq |M|^C \cdot |\Sel(T^\vee)| \cdot \mcT(T) \cdot \left(\prod_{s \in \Smed \cup \Slg} |R_s/(R_s \cap T^{\perp})|\right),
\end{equation}
where $\mcT(T)$ is taken with respect to the local conditions $\iota^{-1}(\msL_{xv})$. Now the inclusion $T_\circ \xhookrightarrow{} T$ dualizes to
$$
0 \rightarrow (T/T_\circ)^\vee \rightarrow T^\vee \rightarrow T_\circ^\vee \rightarrow 0.
$$
Therefore we get the upper bound
\begin{equation}
\label{eBound3}
|\Sel(T^\vee)| \leq |\Sel((T/T_\circ)^\vee)| \cdot |\Sel(T_\circ^\vee)|.
\end{equation}
Stitching the bounds \eqref{eBound1}, \eqref{eBound2} and \eqref{eBound3} together, we always have
\begin{multline}
\label{eStitch}
|\Sel(M)| \leq |M|^C \cdot \mcT(T) \cdot \prod_{s \in \Smed \cup \Slg} |R_s/(R_s \cap T^{\perp})| \cdot \\
|\Sel(T^\circ/T)| \cdot |\Sel(M/T^\circ)| \cdot |\Sel((T/T_\circ)^\vee)| \cdot |\Sel(T_\circ^\vee)|.
\end{multline}
At places $v = \pi_s(x)$ for $s \in \Smed \cup \Slg$, it follows from Lemma \ref{lMTUnr} and Lemma \ref{lSubquotient} that the local conditions on the four Selmer groups appearing on the right hand side of \eqref{eStitch} are always contained inside the unramified local conditions. Thus, equation \eqref{eStitch} implies
\begin{equation}
\label{eStich2}
|\Sel(M)| \leq |M|^C \cdot |M|^{C |\Ssm|} \cdot \mcT(T) \cdot \prod_{s \in \Smed \cup \Slg} |R_s/(R_s \cap T^{\perp})|
\end{equation}
for some $C > 0$.

Now to finish the proof, we distinguish two cases. Call $x \in X$ discardable if $x \in X_{\textup{bad}}$ (see Lemma \ref{lChebSmall} for the definition of $X_{\textup{bad}}$) or if $x$ violates the inequality \eqref{eq:Tcircnobetter} from Lemma \ref{lTamaDichotomy}. In this case, we employ the bound from equation \eqref{eStich2}, which, using Definition \ref{defn:goodgrid} (4), combines with the savings from respectively Lemma \ref{lChebSmall} and Lemma \ref{lTamaDichotomy} to give an acceptable saving for $H \geq \exp^{(2)}((\log C|M|)^C)$.

Now suppose that $x$ is not discardable. Then by Lemma \ref{lChebSmall}, $|\Sel(M/T^\circ)|$ and $|\Sel(T_\circ^\vee)|$ are bounded by $(C \cdot \textup{sup}(g))^{\log_2 |M|}$, and moreover $|\Sel((T/T_\circ)^\vee)|$ and $|\Sel(T^\circ/T)|$ are bounded by $|M|^C$ thanks to Lemma \ref{lTamaDichotomy}, as they carry the unramified local conditions. Hence equation \eqref{eStitch} gives the proposition.
\end{proof}

\subsection{Bilinear methods}
By considering the support of $g_{k_s}(p)$ in equation \eqref{eKeyPipeline}, we may pass without loss of generality to the subset $Y_s$ of $X_s$ consisting of those primes $p \in X_s$ that come with a choice of $\msL_p$ such that $(p, \msL_p)$ is in the $k_s^{th}$ strong equivalence class. By definition of strong equivalence, we may sensibly speak about $A_s$, $R_s$ and $\Omega_s$ in the subset $Y_s$. Take
\[
\mathscr{M} = \ker\left(H^1(G_\Q, M) \rightarrow \prod_{v \not \in \Vplac_0} H^1(I_v, M)\right) \oplus \bigoplus_{s \in \Ssm \cup \Smed \cup \Slg} A_s.
\]
We have a map $\Psi_x: \mathscr{M} \to H^1(G_\Q, M)$ given by
\[
(\phi_0, (a_s)_s) \mapsto \phi_0 + \sum_s \mathfrak{B}_{\pi_s(x)}(a_s).
\]

\begin{proposition}
\label{pBilinear}
There exists $C > 0$ depending only on the starting tuple $(K/\Q, \Vplac_0, e_0)$ such that the following holds. Let $X$ be an okay grid of height $H$.

Let $T$ be a submodule of $M$. For $s \in \Smed \cup \Slg$, partition $Y_s$ into equivalence classes according to the symbols at small primes, and take $X'_s$ to be one such class. Take $X' = \prod_{s \in \Smed \cup \Slg} X'_s$; this naturally corresponds to a subgrid of $X$. We assume that
\begin{equation}
\label{egNoVanish}
\sum_{x \in X'} g(\bar{n}(x)) \ge \exp^{(3)} \left(\frac{1}{4} \log^{(3)} H\right)^{-1} \sum_{x \in X} g(\bar{n}(x)).
\end{equation} 
Take $m := (\phi_0, (a_s)_s) \in \mathscr{M}$. We assume that the $a_s$ with $s \in \Smed \cup \Slg$ generate $T(-1)$. Take $Z$ to be the subset of $z \in X'$ such that $\Psi_z(m) \in \Sel\, M(\bar{n}(z))$. Then, if $\log^{(3)} H \geq C \log^{(3)}(C|M|)$,
\[
\sum_{z \in Z} g(\bar{n}(z)) \le \textup{sup}(g)^C \cdot \exp\left( C (\log |M|)^2\right) \cdot \prod_{s \in \Smed \cup \Slg} \frac{|R_s \cap T^{\perp}|}{|R_s|} \cdot \sum_{x \in X'} g(\bar{n}(x)).
\]
\end{proposition}

Recall that, given a submodule $T$ of $M$, we defined new local conditions $\msL'_x$ on $M$ as in Definition \ref{dNewLocal}. Since the $a_s$ are all in $T(-1)$ by assumption, it follows that $\Sel \, M(\bar{n}(x)) \subseteq \Sel(M, (\msL_{xv}')_v)$ for all $x \in X'$. Moreover, since the symbols at small primes are fixed inside $X'$, we have for each $p \in \Vplac_0$ and for each $p$ of the shape $\pi_s(x)$ for some $s \in \Ssm$ that
$$
\res_p \Psi_x(m) \in \msL_{x p}
$$
holds for either all $x \in X'$ or none of the $x \in X'$. Since the left hand side of Proposition \ref{pBilinear} is zero in the latter case, we may and will assume that $\res_p \Psi_x(m) \in \msL_{x p}$ for all $x \in X'$. Under this assumption, we always have 
\begin{equation}
\label{eAlwaysThere}
\Psi_x(m) \in \Sel(M, (\msL_{xv}')_v) \quad \textup{ for all } x \in X'.
\end{equation}
We now write down our local conditions explicitly.

\begin{lemma}
\label{lLocalConditions}
Let $m := (\phi_0, (a_s)_s) \in \mathscr{M}$, let $s \in \Smed \cup \Slg$ and let $p = \pi_s(x)$. Then $\res_p(\Psi_x(m)) \in \msL_{xp}$ if and only if
\begin{equation}
\label{eLocalCondition}
r\left(\left(\Psi_x(m) - \mathfrak{B}_p(a_s)\right)(\Frob\, p) \right) = \Omega_s\left(a_s, r\right)\quad\textup{for all } r \in R_s.
\end{equation}
\end{lemma}

\begin{proof}
We start by observing that a necessary condition for $\res_p(\Psi_x(m)) \in \msL_{xp}$ is that $\mathfrak{R}_p(\Psi_x(m)) = a_s \in A_s$, which holds by definition of $\mathscr{M}$. Assuming that this necessary condition is satisfied, we have $\res_p(\Psi_x(m)) \in \msL_{xp}$ if and only if there exists $w \in \msL_{xp}$ such that $\mathfrak{R}_p(w) = a_s$ and
\begin{align}
\label{eNrCondition}
\res_p(\Psi_x(m)) - w \in \msL_{xp} \cap \ker(\mathfrak{R}_p).
\end{align}
We remark that for $w \in \msL_{xp}$ satisfying $\mathfrak{R}_p(w) = a_s$, the truth of \eqref{eNrCondition} is independent of the choice of $w$. Equation \eqref{eNrCondition} may be rephrased as
$$
r\left((\res_p(\Psi_x(m)) - w)(\Frob \, p)\right) = 0
$$
for all $r \in R_s$. Observing that
$$
r\left((\res_p(\Psi_x(m)) - w)(\Frob \, p)\right) = r\left(\left(\Psi_x(m) - \mathfrak{B}_p(a_s)\right)(\Frob\, p)\right) - \Omega_s\left(a_s, r\right)
$$
ends the proof of the lemma.
\end{proof}

For $r \in R_s \cap T^\perp$, the validity of equation \eqref{eLocalCondition} is independent of $z \in X'$. Indeed, this follows from our assumption that $a_s \in T(-1)$ and functoriality of $\mfB_p$, see \cite[Definition 3.10]{Smi22a}. Henceforth we shall assume that equation \eqref{eLocalCondition} holds for all $r \in R_s \cap T^\perp$. Set
$$
\mathscr{R} := \bigoplus_{s \in \Smed \cup \Slg} \frac{R_s}{R_s \cap T^\perp}.
$$
We now deduce from our assumption regarding $R_s \cap T^\perp$ and from Lemma \ref{lLocalConditions} that we may define, for every $x = (p_s)_s \in X'$, a pairing $\langle\,\,\,,\,\,\rangle_x: \mathscr{M} \times \mathscr{R} \to \Q/\Z$ by
\[
\langle (\phi_0, (a_s)_s),\, (r_s)_s \rangle_x = \sum_s r_s\left(\left( \Psi_x(\phi_0, (a_s)_s) - \mathfrak{B}_{p_s}(a_s)\right)\left(\Frob \,p_s\right)\right) - \Omega_s(a_s, r_s). 
\]
Recall that we have fixed $m := (\phi_0, (a_s)_s) \in \mathscr{M}$. Then our goal is to give a reasonable upper bound for
\begin{align}
\label{eFullCharSum}
\sum_{z \in Z} g(\bar{n}(z)) &= \sum_{x \in X'} \frac{g(\bar{n}(x))}{\# \mathscr{R}} \sum_{r \in \mathscr{R}} e\left(\langle m, r \rangle_x \right) \leq \frac{1}{\# \mathscr{R}} \sum_{r \in \mathscr{R}} \left|\sum_{x \in X'} g(\bar{n}(x))  e\left(\langle m, r \rangle_x \right)\right| \nonumber \\
&= \prod_{s \in \Smed \cup \Slg} \frac{|R_s \cap T^{\perp}|}{|R_s|} \sum_{r \in \mathscr{R}} \left|\sum_{x \in X'} g(\bar{n}(x))  e\left(\langle m, r \rangle_x \right)\right|.
\end{align}
Our strategy will be to further fix $r$. Depending on $(m, r)$, the pairing $\langle m, r \rangle_x$ may behave in rather different ways as a function of $x$: in some situations it will be constant, for example when $m = r = 0$, while in ``typical'' situations we expect cancellation. For this reason, we will now focus our efforts on understanding the pairing $\langle m, r \rangle_x$ better as a function of $x$.

Recall that all elements in $X_s'$ share the same class, and we denote by $\sigma_s \in \Gal(K/\Q)$ the unique element imposed by that class. Given $s, t \in \Smed \cup \Slg$, we define $B(s, t)$ to be a set of representatives for the collection of double cosets $\langle \sigma_s \rangle \backslash \Gal(K/\Q) / \langle \sigma_t \rangle$. We will also frequently use the notation $\symb{\ovp_s}{\ovp_t} \in (\mu_{e_0})_{\langle \sigma_s \rangle \cap \langle \sigma_t \rangle}$ to mean the alternative symbol defined in \cite[Definition 3.21]{Smi22a}. Here $\ovp_s$ and $\ovp_t$ are the primes above $p_s$ and $p_t$ corresponding to our fixed embeddings of $G_q$ inside $G_\Q$. We also fix a total ordering $<$ on $\Smed \cup \Slg$.

\begin{lemma}
\label{lTwoLines}
Let $m = (\phi_0, (a_s)_s) \in \mathscr{M}$ and let $r = (r_s)_s \in \mathscr{R}$. Then there is some $c \in \Q/\Z$ such that for all $x = (p_s)_s \in X'$
\begin{align}
\label{eCharSum1}
\langle m, r \rangle_x = c + \sum_{\substack{s, t \in \Smed \cup \Slg \\ s \neq t}} \sum_{\tau \in B(t, s)} r_s\left(\tau^{-1}\left(a_t\left(\symb{\tau \ovp_s}{\ovp_t}\right)\right)\right).
\end{align}
Moreover, there is also some $c' \in \Q/\Z$ such that for all $x \in X'$
\begin{align}
\label{eCharSum2}
\langle m, r \rangle_x = c' + \sum_{\substack{s, t \in \Smed \cup \Slg \\ s > t}} \sum_{\tau \in B(s, t)} \left(r_s \cdot \tau a_t + \tau r_t \cdot a_s\right) \cdot \symb{\tau \ovp_t}{\ovp_s}.
\end{align}
The first and second dot product denote the evaluation pairing
\[
\frac{R_s}{R_s \cap T^\perp} \times T(-1) \to \Hom(\mu_{e_0}, \Q/\Z)
\]
induced from the evaluation pairing $M^\ast \times M(-1) \to \Hom(\mu_{e_0}, \Q/\Z)$, while the third dot product denotes the evaluation pairing
$$
\Hom(\mu_{e_0}, \Q/\Z) \times \mu_{e_0} \to \Q/\Z.
$$
\end{lemma}

\begin{proof}
In order to prove equation \eqref{eCharSum1}, we start with the identity 
\begin{multline*}
\langle m, r \rangle_x = \sum_s r_s\left(\phi_0(\Frob \, p_s) + \sum_{t \in \Ssm} \mathfrak{B}_{p_t}(a_t)(\Frob \, p_s)\right) - \sum_s \Omega_s(a_s, r_s) \\
+ \sum_{\substack{s, t \in \Smed \cup \Slg \\ s \ne t}} r_s\left(\mathfrak{B}_{p_t}(a_t)(\Frob \, p_s)\right).
\end{multline*}
Then we define 
$$
c := \sum_s r_s\left(\phi_0(\Frob \, p_s) + \sum_{t \in \Ssm} \mathfrak{B}_{p_t}(a_t)(\Frob \, p_s)\right) - \sum_s \Omega_s(a_s, r_s).
$$
Since all $p_s \in X_s'$ share the same strong equivalence class and have the same behavior with respect to $\Vplac_0$, we see that $c$ depends only on $m$ and $r$. Then we conclude that
$$
\langle m, r \rangle_x = c + \sum_{s \ne t} r_s\left(\mathfrak{B}_{p_t}(a_t)(\Frob \, p_s)\right) = c + \sum_{s \ne t} \sum_{\tau \in B(t, s)} r_s\left(\tau^{-1}\left(a_t\left(\symb{\tau \ovp_s}{\ovp_t}\right)\right)\right)
$$
by \cite[Proposition 3.22]{Smi22a}, which proves \eqref{eCharSum1}.

It remains to prove \eqref{eCharSum2}, and we will do so by using reciprocity and \eqref{eCharSum1}. We start by recalling the relevant reciprocity law for us \cite[eq.~(5.1)]{Smi22b}, namely
$$
\symb{\tau \ovp}{\ovq} = \zeta \cdot \tau\left(\symb{\tau^{-1} \ovq}{\ovp}\right),
$$
where $\zeta$ depends only on $\tau$, the class of $p$ and the class of $q$. We split the sum
\begin{multline}
\label{eSplit}
\sum_{s \ne t} \sum_{\tau \in B(t, s)} r_s\left(\tau^{-1}\left(a_t\left(\symb{\tau \ovp_s}{\ovp_t}\right)\right)\right) = \sum_{s > t} \sum_{\tau \in B(t, s)} r_s\left(\tau^{-1}\left(a_t\left(\symb{\tau \ovp_s}{\ovp_t}\right)\right)\right) + \\
\sum_{s < t} \sum_{\tau \in B(t, s)} r_s\left(\tau^{-1}\left(a_t\left(\symb{\tau \ovp_s}{\ovp_t}\right)\right)\right).
\end{multline}
For the terms $s > t$, we note that inversion defines a bijection between $B(t, s)$ and $B(s, t)$. Hence we obtain for those terms
\begin{align}
\label{eST}
\sum_{s > t} \sum_{\tau \in B(t, s)} r_s\left(\tau^{-1}\left(a_t\left(\symb{\tau \ovp_s}{\ovp_t}\right)\right)\right) 
&= \sum_{s > t} \sum_{\tau \in B(s, t)} r_s\left(\tau\left(a_t\left(\symb{\tau^{-1} \ovp_s}{\ovp_t}\right)\right)\right) \nonumber \\
&= \sum_{s > t} \sum_{\tau \in B(s, t)} \left(c_{s, t, \tau} + r_s\left(\tau\left(a_t\left(\tau^{-1}\left(\symb{\tau \ovp_t}{\ovp_s}\right)\right)\right)\right)\right) \nonumber \\
&= c_1 + \sum_{s > t} \sum_{\tau \in B(s, t)} (r_s \cdot \tau a_t) \cdot \left(\symb{\tau \ovp_t}{\ovp_s}\right)
\end{align}
for some constant $c_{s, t, \tau}$ depending only on $s, t, \tau, m, r$ and for some constant $c_1$ depending only on $m, r$, where we used reciprocity in the middle equation and where we used the definition $(\tau a_t)(x) = \tau(a_t(\tau^{-1}(x)))$ in the last equation.

For the terms $s < t$, we switch the roles of the letters $s$ and $t$ to obtain
\begin{align}
\label{eTS}
\sum_{s < t} \sum_{\tau \in B(t, s)} r_s\left(\tau^{-1}\left(a_t\left(\symb{\tau \ovp_s}{\ovp_t}\right)\right)\right) &= \sum_{s > t} \sum_{\tau \in B(s, t)} r_t\left(\tau^{-1}\left(a_s\left(\symb{\tau \ovp_t}{\ovp_s}\right)\right)\right) \nonumber \\
&= \sum_{s > t} \sum_{\tau \in B(s, t)} (\tau r_t \cdot a_s) \cdot \left(\symb{\tau \ovp_t}{\ovp_s}\right)
\end{align}
because $\tau r_t(x) = r_t(\tau^{-1}(x))$ by definition. Now we set $c' := c + c_1$. Then combining equations \eqref{eCharSum1}, \eqref{eSplit}, \eqref{eST}, \eqref{eTS} ends the proof of the lemma.
\end{proof}

We will now state the relevant cancellation result in our setting; this is a variant of the classical large sieve. In order to state this result, we will use the following definitions.

\begin{mydef}
Take $\sigma, \tau \in \Gal(K/\Q)$. For $\rho \in \Gal(K/\Q)$, we define the composite field
$$
L(\rho) := K^{\langle \sigma \rangle} \rho(K^{\langle \tau \rangle}).
$$
We also define $m_\rho$ to be the maximal divisor of $e_0$ such that $\mu_{m_\rho}$ is a subgroup of $L(\rho)^\times$. This is easily seen to depend only on the class of $\rho$ in the double coset space $\langle \sigma \rangle \backslash \Gal(K/\Q) / \langle \tau \rangle$. The group of symbols $G(\sigma, \tau)$ attached to $\sigma, \tau$ is by definition
$$
G(\sigma, \tau) = \left\{ f \in \mathrm{Map}(B(\sigma, \tau), \mu_{e_0}) : f(\rho) \in \mu_{m_\rho} \text{ for all } \rho \right\}.
$$
\end{mydef}

\begin{mydef}
Recall the definition of (general symbols) $\symb{\ovp}{\ovq}_{\text{gen}}$ in \cite[Definition 3.13]{Smi22a}. Via \cite[Proposition 3.17]{Smi22a}, these can naturally be viewed as elements of $G(\sigma, \tau)$ (with $\sigma = \Frob \, \ovp$ and $\tau = \Frob \, \ovq$). 
\end{mydef}

If $\ovp$ and $\ovq$ lie over different primes of $\Q$, then general symbols are related to symbols via the formula
\begin{align}
\label{eConversion}
\symb{\ovp}{\ovq} = \symb{\ovp}{\ovq}_{\text{gen}}(1)^{m_{\text{id}}/e_0},
\end{align}
where $m_{\text{id}}/e_0$ denotes the inverse of the isomorphism $(\mu_{e_0})_{\langle \sigma \rangle \cap \langle \tau \rangle} \cong \mu_{m_{\text{id}}}$ given by raising to the $e_0/m_{\text{id}}$ power. 

\begin{proposition}
\label{pLargeSieve}
Let $(K/\Q, \Vplac_0, e_0)$ be an unpacked starting tuple. Then there exists $C > 0$ such that the following holds. 

Let $[\ovp_0]$ and $[\ovq_0]$ be classes. Let $X_1$ be a finite subset of $[\ovp_0]$ and let $X_2$ be a finite subset of $[\ovq_0]$. We assume that no two primes in $X_1 \cup X_2$ lie over the same prime of $K(\Vplac_0)$. Define
$$
N_i := \max_{\ovp \in X_i} \ [\Z : \ovp \cap \Q].
$$
Then we have for all nonzero $\rho \in G(\sigma_s, \sigma_t)^\ast$ and all coefficients $d_{\ovq} \in \C$ of absolute value bounded by $1$
$$
\sum_{\ovp \in X_1} \left|\sum_{\ovq \in X_2} d_{\ovq} e\left(\rho(\symb{\ovp}{\ovq}_{\textup{gen}})\right)\right| \leq \frac{CN_1N_2}{\min(N_1, N_2)^{\frac{1}{3[K : \Q] + 3}}}.
$$
\end{proposition}

\begin{proof}
Without the coefficients $d_{\ovq}$, this is nothing more than a simplified version of \cite[Theorem 5.2]{Smi22a}. Inspecting the proof of \cite[Theorem 5.2]{Smi22a}, one sees that the more general result with coefficients $d_{\ovq}$ is actually proven \cite[p.~39]{Smi22a}.
\end{proof}

We are now ready to prove Proposition \ref{pBilinear}.

\begin{proof}[Proof of Proposition \ref{pBilinear}]
We start by picking a subset $S_{\text{gen}} \subseteq \Smed \cup \Slg$ such that the set $\{a_s : s \in S_{\text{gen}}\}$ generates $T(-1)$ and such that $|S_{\text{gen}}| \leq \log_2 |M|$. Let $r = (r_s)_s \in \mathscr{R}$. We now distinguish two cases.

\paragraph{Case 1.} Assume that for all $s \not \in S_{\text{gen}}$ and for all $t \in S_{\text{gen}}$ and for all $\tau \in \Gal(K/\Q)$ we have 
$$
r_s \cdot \tau a_t + \tau r_t \cdot a_s = 0.
$$
If we fix $r_t$ for $t \in S_{\text{gen}}$, then, using that $\{a_s : s \in S_{\text{gen}}\}$ generates $T(-1)$, we see that the above equation uniquely determines $r_s$ for all $s \in (\Smed \cup \Slg) - S_{\text{gen}}$. Thus there are at most $|M|^{|S_{\text{gen}}|}$ possibilities for $(r_s)_s$ in this case. Bounding the sum in equation \eqref{eFullCharSum} trivially for all such $(r_s)_s$, we stay within the bound of Proposition \ref{pBilinear}.

\paragraph{Case 2.} Assume that there exists $s \not \in S_{\text{gen}}$, $t \in S_{\text{gen}}$ and $\tau \in \Gal(K/\Q)$ such that
$$
r_s \cdot \tau a_t + \tau r_t \cdot a_s \neq 0.
$$
We apply Lemma \ref{lTwoLines}, and we insert the formula \eqref{eCharSum2} into \eqref{eFullCharSum}. Here we take care to include $\tau$ in our choice of representative set $B(s, t)$ for the double cosets $\langle \sigma_s \rangle \backslash \Gal(K/\Q) / \langle \sigma_t \rangle$.

We use the identity \eqref{eConversion} and the identity 
$$
\symb{\tau \ovp_t}{\ovp_s}_{\text{gen}}(1) = \tau\left(\tau^{-1}\left(\symb{\tau \ovp_t}{\ovp_s}_{\text{gen}}(1)\right)\right) = \tau\left(\symb{\ovp_t}{\ovp_s}_{\text{gen}}(\tau^{-1})\right)
$$
from \cite[Proposition 3.16 (3)]{Smi22a} to convert between symbols and general symbols. Now the desired upper bound follows from Proposition \ref{pLargeSieve}, which proves equidistribution of the general symbol $\symb{\ovp_t}{\ovp_s}_{\text{gen}}(\tau^{-1})$ as $\tau$ runs through $B(s, t)$ (hence $\tau^{-1}$ runs through $B(t, s)$), where we use the assumption \eqref{egNoVanish} and Definition \ref{defn:goodgrid} (2) to write our final bound in terms of $\sum_{x \in X'} g(\bar{n}(x))$. 
\end{proof}

\subsection{Proof of Theorem \ref{thm:charsum}}
We now combine Proposition \ref{pCheb} and Proposition \ref{pBilinear} to prove Theorem \ref{thm:charsum}.

\begin{proof}[Proof of Theorem \ref{thm:charsum}]
Throughout the proof, $C$ denotes a constant depending only on the starting tuple, $c$ and $k$.

Firstly, let $X$ be an okay grid of whose height $H$ satisfies $\log^{(3)} H \geq C \log^{(3)}(C|M|)$. For $s \in \Smed \cup \Slg$, recall that $X'_s$ is the subset of $Y_s$ cut out by symbols at the small primes and that $X' = \prod_{s \in \Smed \cup \Slg} X'_s$. Moreover, if the subgrid $X'$ does not satisfy \eqref{egNoVanish}, then we use Definition \ref{defn:goodgrid} (1) and the trivial bound
$$
|\Sel \,M(\bar{n}(x))| \leq |M|^{C + r} \leq |M|^{C + (\log \log H)^2},
$$
which stays within the claimed bound even after summing over the at most $|M|^{C r |\Ssm|}$ possible subgrids $X'$. We henceforth assume that $X'$ satisfies \eqref{egNoVanish}.

In this case, we have
$$
\sum_{x \in X'} g(\bar{n}(x)) \cdot |\Sel \,M(\bar{n}(x))| \leq \sum_{T \subseteq M} \sum_{\substack{m = (\phi_0, (a_s)_s) \in \mathscr{M} \\ \langle a_s : s \in \Smed \cup \Slg \rangle = T(-1)}} \sum_{x \in X'} g(\bar{n}(x)) \cdot \mathbf{1}_{\Psi_x(m) \in \Sel\, M(\bar{n}(x))}.
$$
We apply Proposition \ref{pBilinear} to each inner sum. This yields the bound
$$
\sum_{T \subseteq M} \sum_{\substack{m = (\phi_0, (a_s)_s) \in \mathscr{M} \\ \langle a_s : s \in \Smed \cup \Slg \rangle = T(-1)}} \textup{sup}(g)^C \cdot \exp\left( C (\log |M|)^2\right) \cdot \prod_{s \in \Smed \cup \Slg} \frac{|R_s \cap T^{\perp}|}{|R_s|} \cdot \sum_{x \in X'} g(\bar{n}(x)).
$$
Recall that we assumed throughout that equation \eqref{eLocalCondition} holds for all $r \in R_s \cap T^\perp$, i.e.
\begin{equation}
\label{eLocalCondition2}
r\left(\left(\Psi_x(\phi_0, (a_s)_s) - \mathfrak{B}_p(a_s)\right)(\Frob\, p) \right) = \Omega_s\left(a_s, r\right).
\end{equation}
Note that if $\langle a_s : s \in \Smed \cup \Slg \rangle = T(-1) = \langle b_s : s \in \Smed \cup \Slg \rangle$ and $a_s = b_s$ for $s \in \Ssm$, we have that
$$
r\left(\left(\Psi_x(\phi_0, (a_s)_s) - \mathfrak{B}_p(a_s)\right)(\Frob\, p) \right) = r\left(\left(\Psi_x(\phi_0, (b_s)_s) - \mathfrak{B}_p(b_s)\right)(\Frob\, p) \right)
$$
for all $x \in X'$. Hence the number of $(a_s)_{s \in \Smed \cup \Slg}$ satisfying equation \eqref{eLocalCondition2} is at most the number of $(a_s)_{s \in \Smed \cup \Slg}$ such that $\Omega_s\left(a_s, r\right) = 0$ for all $r \in R_s \cap T^\perp$. Summing over all such $a_s$ with $s \in \Smed \cup \Slg$, we recognize precisely the Tamagawa ratio $\mcT_v(T)$ at the medium and large primes by Lemma \ref{lFormula} and Lemma \ref{lPullBack}. Incurring a loss of $|M|^{2C + 2|\Ssm|}$ at the small primes from the possible $a_s$ for $s \in \Ssm$ and from the Tamagawa ratio, we then sum over all subgrids $X'$ satisfying \eqref{egNoVanish} to obtain the theorem in the case of okay grids.

Secondly, let $X$ be a $\frac{c}{2k \cdot \textup{low}(g)}$-good grid of height $H \geq \exp^{(2)}\big((\log C|M|)^C\big)$. We start as before and incorporate the condition \eqref{eAlwaysThere} to arrive at the bound
\begin{multline*}
\sum_{T \subseteq M} \sum_{\substack{m = (\phi_0, (a_s)_s) \in \mathscr{M} \\ \langle a_s : s \in \Smed \cup \Slg \rangle = T(-1)}} \textup{sup}(g)^C \cdot \exp\left( C (\log |M|)^2\right) \cdot \\
\prod_{s \in \Smed \cup \Slg} \frac{|R_s \cap T^{\perp}|}{|R_s|} \cdot \sum_{x \in X} g(\bar{n}(x)) \cdot \mathbf{1}_{\Psi_x(m) \in \Sel(M, (\msL_{xv}')_v)}.
\end{multline*}
after summing over the various $X'$ (note that we may simply omit those $X'$ failing \eqref{eAlwaysThere}, as the sum is empty in that case as argued immediately after Proposition \ref{pBilinear}).

Pulling the sum over $m$ to the inside, we recognize exactly $|\Sel(M, (\msL_{xv}')_v)|$, and hence an application of Proposition \ref{pCheb} ends the proof.
\end{proof}

\subsection{Proof of Theorem \ref{thm:main}}
\label{ssec:proof_main}
We have now established the main results of Section \ref{sSieve}, Section \ref{sGrid}, and Section \ref{sChar}, which are respectively Theorem \ref{tSieve}, Proposition \ref{prop:gridding} and Theorem \ref{thm:charsum}. Our final task is to combine them in order to prove our main theorem.

\begin{proof}[Proof of Theorem \ref{thm:main}]
Let $X$ be a set, let $M$ be a $G_\Q$-module and let $H > 1$ be a real number. Let $h: X \rightarrow \R_{\geq 0}$ be a height function. Let $\{M_x : x \in X\}$ be a constant-module family with effectively equidistributed local conditions. Let $\tilde{g}: \prod' P_p \rightarrow \R_{>0}$ also be as in Definition \ref{defn:effective-constant}. Our task is then to show that
$$
\sum_{x \in X_H} \left(\frac{\# \Sel\, M_x}{\mcTbnd(M_x)}\right)^{\kappa} \tilde{g}(\msL_x)^{\nu} \,\le\, \exp \exp (C\kappa) \sum_{x \in X_H} \tilde{g}(\msL_x)^{\nu},
$$
where we recall that $X_H := \{x \in X : h(x) \leq H\}$, and $\nu$ is at most $\kappa$.

We first note that it suffices to prove the result under the assumption that $\kappa \in \Z_{\geq 0}$. To see this, suppose $\kappa \in \R_{\geq 0} - \Z$, and take $\kappa_0 := \lceil \kappa \rceil$. Applying H\"{o}lder's inequality with $p = \kappa_0/\kappa$, $q = \kappa_0/(\kappa_0 - \kappa)$ gives
$$
\sum_{x \in X_H} \left(\frac{\# \Sel\, M_x}{ \mcTbnd(M_x)}\right)^{\kappa} \tilde{g}(\msL_x)^{\nu} \le \left(\sum_{x \in X_H} \left(\frac{\# \Sel\, M_x}{ \mcTbnd(M_x)}\right)^{\kappa_0}\tilde{g}(\msL_x)^{\nu}\right)^{\frac{\kappa}{\kappa_0}} \left(\sum_{x \in X_H} \tilde{g}(\msL_x)^{\nu}\right)^{\frac{\kappa_0 - \kappa}{\kappa_0}}
$$
and hence, if the result holds for the integer $\kappa_0$, we find that
\[
\sum_{x \in X_H} \left(\frac{\# \Sel\, M_x}{ \mcTbnd(M_x)}\right)^{\kappa} \tilde{g}(\msL_x)^{\nu} \le \exp \exp(C \kappa_0)^{\kappa/\kappa_0} \sum_{x \in X_H} \tilde{g}(\msL_x)^{\nu}.
\]
A small computation shows that $\exp\exp(C\kappa_0)^{\kappa/\kappa_0} \le \exp\exp(e^C \kappa)$, so this is acceptable.

Henceforth we assume that $\kappa \in \Z_{\geq 0}$. We start by an application of Theorem \ref{tSieve} with the function $\lambda$ mapping $x \in X_H$ to $(\prod_{\msL_{xp} \neq H^1_{\text{ur}}(G_p, M)} p, (\msL_{xp})_p)$ and the function $S$ being
$$
S(n, (\msL_p)_p) := \left(\frac{\# \Sel(M, (\msL_p)_p)}{\mcTbnd(M, (\msL_p)_p)}\right)^{\kappa} \tilde{g}((\msL_p)_p)^{\nu}.
$$
The assumption \eqref{eLevel} follows from Definition \ref{defn:effective-constant} (1), while the assumption \eqref{eSieveAssu} readily follows from Definition \ref{defn:effective-constant} (2) and our assumptions on $\tilde{g}$.

Hence Theorem \ref{tSieve} produces the upper bound
\begin{multline}
\label{ePostSieve}
\sum_{x \in X_H} \left(\frac{\# \Sel\, M_x}{\mcTbnd(M_x)}\right)^{\kappa} \tilde{g}(\msL_x)^{\nu} \le \\
\# X_H e^{e^{C \kappa}} \prod_{C^\kappa < p \leq H^{c/3}} (1 - m(p)) \sum_{\substack{(d, \mathcal{D}) \\ d \leq H^{c/3}}} S(d, \mathcal{D}) \prod_{p \mid d} \mu_p(\msL_p) \prod_{\substack{p \mid d \\ p > C^\kappa}} (1 - m(p))^{-1},
\end{multline}
where $C > 0$ is independent of $H$, $\kappa$ and $\nu$. Indeed if we let $C_1 \geq 2$ denote a valid constant, independent of $H$, $\kappa$ and $\nu$, to which we can apply Theorem \ref{tSieve} in the special case $\kappa = 1$, then we can apply Theorem \ref{tSieve} in the general case $S^\kappa$ with a valid choice of constant being $C_1^\kappa$. The double exponential loss in $\kappa$ in equation \eqref{ePostSieve} then comes from estimating the infinite sum in Theorem \ref{tSieve}.

We next aim for an application of Proposition \ref{prop:gridding} to the inner sum in equation \eqref{ePostSieve}. To this end, we observe that taking $g_i$ as in equation \eqref{eKeyPipeline}, the inner sum of equation \eqref{ePostSieve} is
\begin{multline*}
\sum_{\substack{(d, \mathcal{D}) \\ d \leq H^{c/3}}} S(d, \mathcal{D}) \prod_{p \mid d} \mu_p(\msL_p) \prod_{\substack{p \mid d \\ p > C^\kappa}} (1 - m(p))^{-1} = \\
\sum_{n_1 \cdots n_k \le H^{c/3}} g_1(n_1) \cdots g_k(n_k) \left(\frac{\# \Sel\, M(n_1, \dots, n_k)}{\mcTbnd(M(n_1, \dots, n_k))}\right)^{\kappa}.
\end{multline*}
Set
$$
S'(n_1, \dots, n_k) := \left(\frac{\# \Sel\, M(n_1, \dots, n_k)}{\mcTbnd(M(n_1, \dots, n_k))}\right)^{\kappa}.
$$
We always have the trivial bound
$$
\sum_{n_1 \cdots n_k = n} g_1(n_1) \cdots g_k(n_k) S'(n_1, \dots, n_k) \le \frac{C^{\omega(n) (\kappa + 1)}}{n},
$$
where $C > 1$ does not depend on $H$, $\kappa$ or $\nu$. Hence
$$
\sum_{n_1 \cdots n_k = n} g_1(n_1) \cdots g_k(n_k) S'(n_1, \dots, n_k) \le \prod_{p \le H^{c/3}} \left(1 + \frac{C^{\kappa + 1}}{p} \right) \le (10\log H)^{C^{\kappa + 1}},
$$
where the last inequality follows for $H \ge 25$ by Mertens' theorem. Summing the left hand side of this estimate over $n_1 \dots n_k \le H^{c/3}$ is already enough to prove our main result Theorem \ref{thm:main} for $H \leq \exp^{(3)}(C' \kappa)$ for every fixed $C' > 0$. Thus we will now freely assume that $H > \exp^{(3)}(C' \kappa)$. 

We now apply Proposition \ref{prop:gridding} with $g_1, \dots, g_k$ and $S'$. In order to verify the hypotheses of Proposition \ref{prop:gridding}, we must now exhibit $A > 0$ for which equation \eqref{eq:okay_assu} and \eqref{eq:good_assu} hold. We now apply Theorem \ref{thm:charsum} with the module $M$ in that theorem equal to our $M^{\kappa}$ and note that $\mcTbnd(M) := \mcTbnd(M(\bar{n}(x)))$ is constant on those $x \in X$ with $g(\bar{n}(x)) \neq 0$. It then follows from Theorem \ref{thm:charsum} that the hypotheses \eqref{eq:okay_assu} and \eqref{eq:good_assu} are satisfied with $A = \exp^{(2)}(C\kappa)$ for some sufficiently large $C > 1$ independent of $H$, $\kappa$ and $\nu$. Thus Proposition \ref{prop:gridding} yields the following bound for the inner sum of equation \eqref{ePostSieve}
$$
\exp^{(2)}(C\kappa) \sum_{\substack{n_1 \cdots n_k \le H^{c/3} \\ \mu^2(n_1 \cdots n_k) = 1}} g_1(n_1) \cdots g_k(n_k)
$$
for some $C > 1$ independent of $H$, $\kappa$ and $\nu$, and thus a total upper bound
$$
\# X_H e^{e^{C \kappa}} \prod_{C^\kappa < p \leq H^{c/3}} (1 - m(p))\sum_{\substack{n_1 \cdots n_k \le H^{c/3} \\ \mu^2(n_1 \cdots n_k) = 1}} g_1(n_1) \cdots g_k(n_k).
$$
We now apply Theorem \ref{tSieve2} with the same $\lambda$ and with $S(d, \mathcal{D}) := \tilde{g}(\mathcal{D})^\nu$. Our assumptions on $\tilde{g}$ enforce the validity of \eqref{eLowerS}. We note that Theorem \ref{tSieve2} involves $v$ instead of $c$, but upon lowering $c$ once and for all at the start of the proof, we could have arrived at the upper bound
$$
\# X_H e^{e^{C \kappa}} \prod_{C^\kappa < p \leq H^v} (1 - m(p)) \sum_{\substack{n_1 \cdots n_k \le H^v \\ \mu^2(n_1 \cdots n_k) = 1}} g_1(n_1) \cdots g_k(n_k).
$$
Then an application of Theorem \ref{tSieve2} ends the proof.
\end{proof}
\section{Applications to abelian varieties}
\label{sGeometry}
Our next goal is to generalize Conjecture \ref{conj:ell} to general families of abelian varieties.

\begin{mydef}
\label{def:effective}
Take $\mcX$ to be an integral separated scheme of finite type over $\Z$ such that $X := \mcX_{\QQ}$ is smooth and geometrically irreducible. Fix a height function $h: \mcX(\Z) \to \R_{\ge 0} \cup \{\infty\}$, and take $\mcX_H$ to be the set of points in $\mcX(\Z)$ of height at most $H$ for any $H \ge 1$. We assume that there is some quasi-finite morphism $X \to \mathbb{P}^d_{\QQ}$ such that, taking $h_0$ to be the associated height function, $h_0(x) \le h(x)$ for all $x$ in $\mcX(\Z)$. If $h = h_0$, we call $h$ a \emph{geometric height}.

We say that $(\mcX, h)$ has \emph{effective equidistribution} if there are positive constants $C, c$  such that, for all $H > C$, any positive squarefree integer $a < H^c$ not divisible by a prime $p < C$, and any class $\overline{x}$ in $\mcX(\Z/a^2\Z)$, we have 
\begin{equation}
    \label{eq:eff_equid_X}
\left| \frac{\#\{ x \in \mcX_H\,:\,\, x \equiv \overline{x}\, \text{ mod }\, a^2\}}{\# \mcX_H}\, -\, \prod_{p\mid a} \# \mcX(\Z/p^2\Z)^{-1}\right| \le H^{-c}.
\end{equation}
\end{mydef}

\begin{mydef}[Families of abelian varieties]
\label{defn:family_abelian}
With $\mcX$ as above, take $\mcA \to \mcX$ to be a scheme of finite type. We will assume that there is a positive integer $B$ such that $\mcA_{\Z[1/B]}$ is a group scheme above $\mcX_{\Z[1/B]}$. Taking $\eta$ to be the generic point of $\mcX$, we further assume that $\mcA_{\eta}$ is an abelian variety over the residue field $\kappa(\eta)$ of $\mcX$ at $\eta$.

Fix an algebraic closure $\overline{\kappa(\eta)}$ for $\kappa(\eta)$. Given a finite torsion subgroup $M$ of $\mcA_{\eta}\big(\overline{\kappa(\eta)}\big)$, we call $M$ \emph{rational} if it is stable under $\Gal\big(\overline{\kappa(\eta)}/\kappa(\eta)\big)$. We call it \emph{constant} if there is a field of the form $L\kappa(\eta)$ with $L$ a number field such that every point in $M$ is fixed by the Galois action over $L \kappa(\eta)$.

There is some nonempty open subscheme $U$ of $\mcX$ such that all geometric fibers of $\mcA_U \to U$ are abelian varieties. Given a finite rational torsion subgroup $M$ as above, we may assume $U$ is such that there is a subgroup scheme $\mcM$ of $\mcA_{U}$ such that, for $x \in U(\ovQQ)$, $M_x := (\mcM_x)_{\QQ}$ is identified with $M$ as an abelian group.  We endow the $G_{\QQ}$-module $M_x$ with the local conditions coming from the inclusion $M_x \hookrightarrow \mcA_x(\ovQQ)$.

If $M$ is constant, it has the structure of a $G_{\Q}$-module, and $M$ is isomorphic to $M_x$ as a $G_{\Q}$-module for all $x \in U(\Q)$.
\end{mydef}

\begin{example}
\label{ex:ell_counts}
The equation \eqref{eq:generic_ell_fib} gives an example of a scheme $\mcA$ over 
\[\mcX = \mathbb{A}^n_{\Z} = \text{Spec} \, \Z[u_1, \dots, u_n];\] 
the height $h(b_1, \dots, b_n) = \max(|b_1|^{\gamma_1}, \dots, |b_n|^{\gamma_n})$ corresponds to the geometric height corresponding to the quasi-finite map $\mathbb{A}^n_{\Z} \to \mathbb{P}^n_{\Z}$ given by
\[(b_1, \dots, b_n) \mapsto \left[b_1^{\gamma_1}, \dots, b_n^{\gamma_n}, 1\right].\]
It is straightforward to see that $(\mcX, h)$ has effective equidistribution in this family, and $\mcA/\mcX$ gives an example of a family of abelian varieties satisfying the conditions of Definition \ref{defn:family_abelian}.

Effective equidistribution is known for occasional examples besides affine space. We refer the interested reader to \cite{El-Baz Loughran Sofos} for more details.
\end{example}

We now give a version of Conjecture \ref{conj:ell} that also applies for choices of $\mcA/\mcX$ outside the scope of Example \ref{ex:ell_counts}. We will show that this conjecture implies Conjecture \ref{conj:ell} with Proposition \ref{prop:Tate_alg}.

\begin{conjecture}
\label{conj:AV}
Choose $\mcX$ as above, and choose a geometric height $h$ on $\mcX(\Z)$ such that $(\mcX, h)$ has effective equidistribution. Choose a family $\mcA \to \mcX$ as above, so that $\mcA_{\Z[1/B]}$ is a group scheme over $\mcX_{\Z[1/B]}$ for some $B \ge 1$, and so that the generic fiber of $\mcA \to \mcX$ is an abelian variety. Choose a rational subgroup 
\[M \subseteq \mcA_{\eta}\big(\overline{\kappa(\eta)}\big)\]
 and use it to define a subgroup scheme
\[\mcM \to \mcA_{U} \to U\]
as above, where $U$ is some nonempty open subscheme of $\mcX$. We take $\mcX_H'$ to denote the points in $\mcX_H \subseteq \mcX(\Z)$ whose associated rational point lies in $U(\Q)$.

Then, for $\kappa \ge 0$, there is some $C_{\kappa} > 0$ so that, for $H$ sufficiently large, we have
\[
\sum_{x \in \mcX_H'} \left(\# \Sel\, M_x\right)^{\kappa} \le C_{\kappa} \sum_{x \in \mcX_H'} \mcTbnd(M_x)^{\kappa}\]
and
\[
\limsup_{H \to \infty} \frac{1}{\# \mcX_H'} \sum_{x \in \mcX_H'} \left(\frac{\# \Sel \, M_x}{\mcTbnd(M_x)}\right)^{\kappa} \le C_{\kappa}.\]
\end{conjecture}

\begin{theorem}
\label{thm:AV}
If $M$ is constant, Conjecture \ref{conj:AV} holds with 
\[C_{\kappa} = \exp \exp(C\kappa)\]
for some $C > 0$ not depending on $\kappa \ge 0$. The conclusion remains true even without the assumption that $h$ is geometric.
\end{theorem}
We prove this in  Section \ref{ssec:stability}.

\subsection{Stability and likelihood of local conditions}
\label{ssec:stability}
We next state a pair of propositions giving us more information about the distribution of the local conditions in Conjecture \ref{conj:AV}. These propositions will be proved in Sections \ref{ssec:AG_input} and \ref{ssec:DenefPas} using a combination of algebraic geometry and model theory. Before we prove these propositions, we will show that they imply the conditions of Theorem \ref{thm:main} in the case when $M$ is constant, allowing us to apply this theorem to prove Theorem \ref{thm:AV}.

We take $\mcX$,  $U$, $\mathcal{A}$, and $M$ as above. 

\begin{mydef}
Given a prime $p$ not dividing $B$ and a point $\overline{x}$ in $\mcX(\Z/p^2\Z)$, we call $\overline{x}$ \emph{very bad} if, for some $x$ in $\mcX(\Z_p)$ projecting to $\overline{x}$, the corresponding point in $X(\QQ_p)$ does not lie in $U(\QQ_p)$. If $\overline{x}$ is not very bad, then every $x$ projecting to $\overline{x}$ corresponds to an abelian variety $\mcA_x/\QQ_p$ and a torsion subgroup scheme $M_x/\QQ_p$ isomorphic as an abelian group to $M$. 

Given $\overline{x}$ that is not very bad, and given $x$ in $\mcX(\Z_p)$ over $\overline{x}$, we define
\[ 
\msL_{xp} = \ker\left(H^1(G_p, M_x) \to H^1(G_p, \mathcal{A}_x\left(\overline{\QQ_p}\right)\right) \quad\text{and}\quad\mcT_{xp} = \frac{\# \msL_{xp}}{\# H^0(G_p, M_x)}.
\]
This definition of the local conditions is synonymous with the one used in Conjecture \ref{conj:AV}.

We say that $\overline{x}$ is \emph{unusable} if it is very bad, or if there are distinct $x, y$ in $\mcX(\Z_p)$ over $\overline{x}$ such that 
\begin{alignat*}{2}
& \mcT_{xp} \ne \mcT_{yp} \quad&&\text{if } M \text{ is nonconstant or}\\
& \msL_{xp} \ne \msL_{yp} \quad&&\text{if } M \text{ is constant.}
\end{alignat*}
\end{mydef}

\begin{proposition}
\label{prop:geom_usable}
There is $C> 0$ determined from $\mcX$, $\mathcal{A}$, and $\mathcal{M}$ such that, for all primes $p > C$,
\[\frac{\# \left \{\overline{x} \in \mcX(\Z/p^2\Z) \,:\,\, \overline{x} \textup{ is unusable}\right\} }{\# \mcX\left(\Z/p^2\Z\right)} \le Cp^{-2}.\]
\end{proposition}

If $\overline{x}$ is usable (i.e.~not unusable), then $\mcT_{xp}$ does not depend on the choice of $x$ above $\overline{x}$. We denote this by $\mcT_{\overline{x} p}$. Similarly, if $M$ is constant and $\overline{x}$ is usable, we may define
\[\msL_{\overline{x} p} = \msL_{x p}\]
independently of the choice of $x$ above $\overline{x}$. 

\begin{proposition}
\label{prop:geom_nonconstant}
Choose a positive rational number $t$ other than $1$. There is a finite Galois extension $L/\QQ$, a constant $C > 0$, and a class function $f: \Gal(L/\QQ)/\sim \to \QQ_{\ge 0}$ so that, for all primes $p > C$,
\[\bigg|f(\textup{Frob }p) p^{-1} - \frac{\# \left\{\overline{x} \in \mcX(\Z/p^2\Z) \,:\,\, \overline{x} \textup{ is usable and } \mcT_{\overline{x} p} = t\right\}}{\# \mcX(\Z/p^2\Z)}  \bigg| \le C p^{-3/2}.\]
If $M$ is constant, take $\Vplac_0$ to be the minimal set of places containing $\infty$, the primes dividing $|M|$, and the primes $p$ where the action of $I_p$ on  $M$ is nontrivial. Take $\mathcal{L}$ to be a  weak equivalence class of local conditions defined with respect to $(M, \Vplac_0)$, in the sense of Definition \ref{defn:equiv_local}. We assume that this weak equivalence class contains $(p, \msL_p)$ with $\msL_p$ not equal to $H^1_{\textup{ur}}(G_p, M)$ for some $p$.

Then we may instead define the class function $f$ and real number $C$ so that, for any sufficiently large prime $p$, we have
\[\Bigg|f(\textup{Frob }p) p^{-1} - \sum_{\substack{\msL_p \textup{ such that} \\ (p, \msL_p) \in \mathcal{L}}} \frac{\# \left \{\overline{x} \in \mcX(\Z/p^2\Z) \,:\,\, \overline{x} \textup{ is usable and } \msL_{\overline{x} p} = \msL_p\right\}}{\# \mcX(\Z/p^2\Z)}  \Bigg| \le C p^{-3/2}.\]
\end{proposition}

\begin{proof}[Proof of Theorem \ref{thm:AV}]
Given $x$ in $\mcX_H'$, we adjust the decorated module $M_x = (M, (\msL_{xp})_p)$ to a quasi-decorated module $M'_x$ by replacing $\msL_{xp}$ with the empty set whenever $x$ maps to an unusable element in $\mcX(\Z/p^2\Z)$. 

Recall that we chose a quasi-finite map $X \to \mathbb{P}^d_{\Q}$  whose associated geometric height $h_0$ is a lower bound for $h$. By spreading out, this extends to a quasi-finite morphism $\mcX_{\Z[A^{-1}]} \to \mathbb{P}^d_{\Z[A^{-1}]}$ for some nonzero integer $A$. The image of the complement of $U_{\Z[A^{-1}]}$  under this map is contained in a proper closed subscheme of $\mathbb{P}^d_{\Z[A^{-1}]}$, and the ideal corresponding to this subscheme contains some nonconstant homogeneous polynomial $f$ in $\Z[u_0, \dots, u_d]$. Given $x \in \mcX(\Z)$ and a prime $p$ not dividing $A$, the local conditions group $\msL_{xp}$ is other than the unramified local conditions only when $f$ is divisible by $p$ at the image of $x$. So the product of all such primes is bounded by $C h_0(x)^{\text{deg}\, f}$ for some $C > 1$. Replacing $h$ with $(2h)^C$ if necessary for some $C > 1$, we may then assume that $h(x)$ is always an upper bound for the product of $p$ where $\msL_{xp}$ is not the unramified local conditions.

We will apply Theorem \ref{thm:main} to the quasi-decorated modules $M'_x = (M, (\msL_{xp}')_p)$ as above, with $\Vplac_0$ taken to be a finite set of primes containing $\infty$, all primes dividing $|M|$, all primes $p$ such that $I_p$ acts nontrivially on $M$, and all primes less than the $C$ specified in Definition \ref{def:effective}.

Given a prime $p$ outside $\Vplac_0$ and a subgroup $\msL_p$ of $H^1(G_p, M)$, we take 
\[\mu_p(\msL_p) = \frac{\#\left\{\overline{x} \in \mcX(\Z/p^2\Z)\,:\,\, \overline{x} \text{ is usable and } \msL_{\overline{x} p} = \msL_p\right\}}{\# \mcX\left(\Z/p^2\Z\right)}.\]
We then take 
\[\mu_p(\emptyset) = 1 - \sum_{\msL_p} \mu_p(\msL_p).\]

Then \eqref{eq:eff_equid_X} gives that, for any sufficiently large $H$, for any squarefree integer $a < H^c$ not divisible by any prime in $\Vplac_0$, and for any choice of $\msL_p \in P_p$ for every $p$ dividing $a$, we have
\[\left| \frac{\# \left\{ x \in \mcX_{H}\,:\,\, \msL_{xp}' = \msL_p \text{ for all } p \mid a\right\}}{\mcX_H} - \prod_{p | a} \mu_p(\msL_p)\right|   \le H^{-c} \cdot \prod_{p | a} \#\mcX\left(\Z/p^2\Z\right).\]
By adjusting the constant $c > 0$ as needed, we see that the quasi-decorated modules $M_x$ satisfy the effective equidistribution condition of Definition \ref{defn:effective-constant}, with the condition
\[\sum_p \mu_p(\emptyset) < \infty\]
following from Proposition \ref{prop:geom_usable}.

Both conditions (2) and (3) of Definition \ref{defn:effective-constant} follow from Proposition \ref{prop:geom_nonconstant}. So Theorem \ref{thm:main} may be applied, with \eqref{eq:quasi_realized} allowing us to replace the quasi-decorated modules $M'_x$ with $M_x$. By selecting $\tilde{g}$ as in Remark \ref{rmk:two_mult_weights}, we may show that the two claimed inequalities of Conjecture \ref{conj:AV} hold with $C_{\kappa}$ of the form $\exp \exp(C \kappa)$.
\end{proof}

\subsection{Some algebraic geometry}
Our goal from here to the end of Section \ref{ssec:DenefPas} is to prove Propositions \ref{prop:geom_usable} and \ref{prop:geom_nonconstant}. We first observe that we may assume that $\mcX$ is affine. After all,  we may always choose a set of affine open subschemes $\mcX_1, \dots, \mcX_k$ covering $\mcX$. From the injections $\mcX_i(\Z/p^2\Z)\hookrightarrow \mcX(\Z/p^2\Z)$, we find that, for any subset $Y$ of $\mcX(\Z/p^2\Z)$, we have
\[\left| Y \right| = \sum_{T \subseteq \{1, \dots, k\} } (-1)^{|T|} \cdot\left| Y \cap \left(\cap_{i \in T} \mcX_i\right)(\Z/p^2\Z)\right|\]
by inclusion-exclusion. Then these propositions hold for $\mcX$ if they hold for each $\mcX_i$.

We now collect the lemmas from algebraic geometry that we will need. Take $d$ to be the dimension of $X_{\QQ}$.
\label{ssec:AG_input}
\begin{lemma}
\label{lem:smoothlift}
Given $\mcX$ as above, there is $C > 0$ so, for all primes $p > C$, all $k, m \ge 1$, and any $\overline{x} \in \mcX(\Z/p^k\Z)$, the preimage of $\overline{x}$ in $\mcX(\Z/p^{k+m}\Z)$ has cardinality $p^{md}$.
\end{lemma}

\begin{proof}
Since $\mcX_{\Q}$ was smooth, we know that $\mcX_{\Z_p}$ is smooth for all sufficiently large $p$ by spreading out \cite[Theorem 3.2.1]{PoonenRationalPoints}.
% This points to EGA, but c'mon.
The result then follows by Hensel's lemma.
% More information: We can break $\mcX$ into affines of the form R[t1, ..., tn]/(f1, ..., fm) so the t1, ... , tm Jacobian is not divisible by p, see Stacks 01V7 (fi all polynomials with coeffs in Zp). Given a mod p^n point x,  the definition of smooth tells us we have a mod p^{n+1} point y over x (as in, a setting for t1, ..., tn) by Hensel's lemma  (or honestly what I'm writing here). We find the number of mod p^{n+1} adjustments we can make without losing x corresponds to a kernel of the full Jacobian matrix, which has the minimal rank of d.
\end{proof}

By the Lang--Weil bound \cite{LangWeil}, we may conclude that $\mcX(\Z/p^2\Z)$ has cardinality within $Cp^{2d - 1/2}$ of $p^{2d}$ for some $C > 0$ depending only on $\mcX$.

Take $\mathcal{Z}$ to be the closed reduced subscheme of $\mcX$ corresponding to $\mcX \backslash U$. Then, if $x$ is a geometric point of $\mcX$ outside $\mathcal{Z}$, then $\mathcal{A}_x$ is an abelian variety and $M_x$ is isomorphic as an abelian group to $M$. This will be useful for the next two lemmas:

\begin{lemma}
Call $\overline{x}$ in $\mcX(\Z/p^2\Z)$ \emph{bad} if it is very bad or if, for some $x \in \mcX(\Z_p)$ mapping to $\overline{x}$, 
\[\ker\left(H^1(G_p, M_x) \to H^1\left(G_p, \mathcal{A}_x(\overline{\QQ_p})\right)\right) \ne H^1_{\textup{ur}}(G_p, M_x).\]
Then there is $C > 0$ not depending on $p$ such that the number of bad $\overline{x}$ is bounded by $Cp^{2d - 1}$.
\end{lemma}

\begin{proof}
If $\overline{x}$ is in the image of a bad point, then its image in $\mcX(\FFF_p)$ must lie in $\mathcal{Z}(\FFF_p)$. Since $\mathcal{Z}_{\FFF_p}$ has positive codimension in $\mcX_{\FFF_p}$ for all but finitely many primes $p$, the lemma then follows from the Lang--Weil bound \cite{LangWeil} and Lemma \ref{lem:smoothlift}.
\end{proof}

\begin{lemma}
\label{lem:verybad}
There is $C > 0$ such that, for all $p$, the number of very bad points in $\mcX(\Z/p^2\Z)$ is bounded by $Cp^{2d - 2}$.
\end{lemma}

\begin{proof}
If $x \in \mcX(\Z_p)$ maps to a very bad point, then it maps to a point in $\mathcal{Z}(\Z/p^2\Z)$. By spreading out, for large enough $p$, $\mathcal{Z}_{\FFF_p}$ is a reduced scheme of codimension at least one in $\mcX_{\FFF_p}$, and whose singular locus has codimension at least $2$ in $\mcX_{\FFF_p}$.

The number of points in $\mcX(\Z/p\Z)$ corresponding to a singular point of $\mathcal{Z}_{\FFF_p}$ can then be bounded by $Cp^{d - 2}$ for large enough $C$, so Lemma \ref{lem:smoothlift} gives that only $Cp^{2d - 2}$ very bad points can map to this singular locus. Meanwhile, the number of $\Z/p^2\Z$ points of $\mathcal{Z}(\Z/p^2\Z)$ mapping to the nonsingular locus of $\mathcal{Z}_{\FFF_p}$ is no more than $C p^{2(d -1)}$ for some $C$ not depending on $p$, as follows from the Lang--Weil bound \cite{LangWeil} and Lemma \ref{lem:smoothlift}.
\end{proof}

Our final lemma will make use of the permissible assumption that $\mcX$ is affine.

\begin{lemma}
\label{lem:consistent}
Take $\mcX = \textup{Spec } A$ to be a affine normal integral scheme over $\Z$ of finite type such that $\mcX_{\QQ}$ is a geometrically irreducible smooth variety. Take $d$ to be the dimension of $\mcX_{\QQ}$, and choose nonzero $g$ in $A$.

For a prime $p$, take $Z_p$ to be the set of $x$ in $\mcX(\Z/p^2\Z)$ such that there are distinct $x_1, x_2$ in $\mcX(\Z_p)$ mapping to $x$ for which the valuation $v_p$ satisfies
\[v_p\left(g(x_1)\right) \ge v_p\left(g(x_1) - g(x_2)\right).\]
There then is some $C > 0$ depending only on $A$ and $g$ such that
\[|Z_p| \le Cp^{2d - 2}.\]
\end{lemma}

\begin{proof}
Take $\mfp_1, \dots, \mfp_k$ to be the distinct height one primes of $A$ containing $g$. Observe that $\text{Spec}\,(A/\mfp_i \otimes \QQ)$ is of dimension at most $d-1$ for each $i \le k$. (In the case that $\mfp_i$ has nonzero intersection with $\Z$, we note that $A/\mfp_i \otimes \Q$ is the zero ring; we are using the convention that the associated empty scheme has dimension $-\infty$).

For $i \le k$, choose $h_i \in A$ generating $\mfp_i$ in the localization $A_{\mfp_i}$, so that there are $r_i, s_i \in A$ outside $\mfp_i$ and a positive integer $a_i$ such that
\[g s_i = h_i^{a_i}r_i.\]
Then the closed subscheme corresponding to $(\mfp_i, r_is_i)$ has codimension at least $2$. By spreading out, the corresponding subscheme of $\mcX_{\FFF_p}$ has codimension at least $2$ for sufficiently large $p$.

We note that any point $x$ in $Z_p$ must satisfy $v_p(g(x)) \ge 2$, and so must correspond to an $\FFF_p$ point in some $\text{Spec}\, A/\mfp_i$. By the Lang--Weil bound \cite{LangWeil}, the number of such points that also lie in $\text{Spec}\, A/(r_is_i)$ is bounded by $Cp^{d - 2}$ for some $C > 0$ not depending on $p$. The number of $\Z/p^2\Z$ points mapping to such a point is then bounded by $Cp^{2d - 2}$. So it suffices to count the points in $Z_{p}$ where $h_i$ is zero modulo $p$ and where $r_i$ and $s_i$ are nonzero modulo $p$. Call this subset $Z_{p, i}$.

Given $x_1, x_2$ in $\mcX(\Z_p)$ over the same point in $Z_{p, i}$, we have $\frac{r_i(x_1)}{s_i(x_1)}  \equiv \frac{r_i(x_2)}{s_i(x_2)} \text{ mod } p^2$ and $h_i(x_1) \equiv h_i(x_2) \text{ mod } p^2$, so
\begin{alignat*}{2}
g(x_1) &\equiv h_i(x_1)^{a_i} \frac{r_i(x_2)}{s_i(x_2)} &&\text{ mod }p^{a_i+1} \\
&\equiv h_i(x_1)^{a_i-1} h_i(x_2) \frac{r_i(x_2)}{s_i(x_2)} \quad&&\text{ mod } p^{a_i + 1}\\
\dots & \equiv g(x_2) &&\text{ mod } p^{a_i+1}.
\end{alignat*}
That is, $v_p(g(x_1) - g(x_2)) \ge a_i + 1$.
Then $v_p(g(x_1))$ is also at least $a_i+1$, forcing $v_p(h_i(x_1))$ to be at least $2$. So $Z_{p, i}$ corresponds to a set of $\Z/p^2\Z$ points in $\Spec\, B$, where 
\[B = A[(r_is_i)^{-1}]/h_iA[(r_is_i)^{-1}].\]
But $B \otimes \QQ$ is reduced. By spreading out, $B \otimes \FFF_p$ is reduced for all but finitely may $p$. The number of $\Z/p^2\Z$ points in $\text{Spec}\, B$ can then be bounded by $Cp^{2d - 2}$ by the argument of Lemma \ref{lem:verybad}.
\end{proof}

\subsection{The Denef--Pas language}
\label{ssec:DenefPas}
Throughout this subsection, we will assume that $\mcX$ is affine.

We recall the Denef--Pas language $\msL_{\text{PR}}$ defined on \cite[p.~140]{PasCrelle}. This language has three sorts: a valued field $K$, its residue field $\overline{K}$, and a codomain for the valuation map, which we always take to be $\Z \cup \{\infty\}$. A formula in this language is built out of:
\begin{itemize}
\item The constants $0, 1$ and the operations $+, \cdot, -, \div$ for both $K$ and $\overline{K}$.
\item A valuation map $\text{ord}: K \to \Z \cup \{\infty\}$.
\item An angular component map $\text{ac}: K \to \overline{K}$.
\item The constants $0$, $1$ and $\infty$ in $\Z \cup \{\infty\}$ in addition to the operator $+$ and the relationships $\le$ and $\equiv \text{ mod } n$ for arbitrary $n > 1$.
\item Variables and quantifiers of any of the three sorts, and boolean operations on formulae.
\end{itemize}

For any prime $p$, there is a standard structure for this language with $K = \QQ_p$. In this structure, $\overline{K}$ is $\FFF_p$, $\text{ord}: K \to \Z \cup \{\infty\}$ is the standard valuation map, and $\text{ac}$ is defined by
\[\text{ac}(x) = p^{-\text{ord}(x)} \cdot x \,\text{ mod } p\quad\text{ for all } x \ne 0.\]

Since $\mcX$ is affine, we may view $\mcX(\QQ_p)$ as the zero locus of some finite collection of fixed integer polynomials in $n$ variables, and $U$ as an open subscheme of $\mcX$. Then there is a formula $\Phi_0$ in the Denef--Pas language taking $n$ variables from the valued field sort such that, for any $\mathbf{x} \in \QQ_p^n$,
\[\QQ_p \models \Phi_0(\mathbf{x}) \iff \mathbf{x} \in U(\QQ_p).\]

\begin{lemma}
\label{lem:logic_nonconstant}
Choose a positive rational number $t$. There is a formula $\Phi$ in the Denef--Pas language taking $n$ variables from the valued field sort so that, for any sufficiently large prime $p$ and any $\mathbf{x}$ in $\QQ_p^n$, if $\QQ_p \models \Phi_0(\mathbf{x})$, then
\[
\QQ_p \models \Phi(\mathbf{x}) \iff \mcT_{\mathbf{x} p} = t.
\]
\end{lemma}

\begin{proof}
We start by considering a general valued field $K$ and residue field $\overline{K}$. Take $m = |\text{Aut}(M)| \cdot |M|$, where $\text{Aut}(M)$ is the set of automorphisms of $M$ as an abelian group.

Firstly, there exists a formula, that given $a_0, \dots, a_{m-1}$ in $K$ of nonnegative valuation, detects whether the ring
\[
\overline{L} = \overline{K}[t]/(t^m + \overline{a}_{m-1}t^{m-1} + \dots + \overline{a}_0)
\]
is a field. Here, $\overline{a}_i$ denotes $\text{ac}(a_i)$ if $a_i$ has zero valuation, and is $0$ otherwise. With these chosen, we take
\[
L_{\text{ur}} = K[t]/(t^m + a_{m-1}t^{m-1} + \dots + a_0).
\]
Note that if $K = \Q_p$, then $a_0, \dots, a_{m-1}$ as above always exist, and $L_{\text{ur}}/K$ is a cyclic Galois extension for any such choice of $a_0, \dots, a_{m-1}$.

Secondly, there exists a formula, that detects whether there exists $b \in K$ of valuation $1$. Then take
\[
L = L_{\text{ur}}[s]/(s^m - b).
\]
If $K = \Q_p$ and if $p > m$, then $L/L_{\text{ur}}$ is a cyclic extension of degree $m$ for all such $b$ and moreover $L$ is the unique Galois extension of $\QQ_p$ of degree $m^2$ whose ramification degree is $m$. 

We think of elements in $L$ as tuples of $m^2$ elements in $K$; arithmetic in $L$ can then be reduced to arithmetic in $K$. Taking $G$ to be the set of $K$-algebra homomorphisms from $L$ to $L$ (so $G$ equals $\Gal(L/K)$ in case $K = \Q_p$), we may record elements of $G$ as $m^2 \times m^2$ matrices with coefficients in $K$ obeying certain conditions in the language of rings.

For $\mathbf{x}$ in $K^n$, the $L$-points of $M_{\mathbf{x}}$ can be specified as tuples valued in $K$ satisfying some formula in the language, with the addition law on the $L$-points also given by a formula; so the set of cocycles in $Z^1(G, M_{\mathbf{x}}(L))$ can be specified in the language, as can the set of coboundaries.

Furthermore, we can characterize the $L$-points of $\mathcal{A}_{\mathbf{x}}$ in terms of a formula in the language, in addition to their addition law. From this, the set
\begin{equation}
\label{eq:loc_cond_model}
\ker\left(H^1(G, M_{\mathbf{x}}(L)) \to H^1(G, \mathcal{A}_{\mathbf{x}}(L))\right)
\end{equation}
may be given as the set satisfying some formula. The set of $\mathbf{x}$ satisfying $\Phi_0$ such that this kernel has a certain size and such that $H^0(G, M_{\mathbf{x}})$ also has a certain size then is given by such a formula.

So $G$ is a quotient of $G_p = \Gal(\overline{\QQ_p}/\QQ_p)$. If $N$ is the kernel of this map, we see that $M_{\mathbf{x}}(\overline{\QQ_p})^N = M_{\mathbf{x}}(L)$ and that $\mathcal{A}_{\mathbf{x}}(L) = \mathcal{A}_{\mathbf{x}}(\overline{\QQ}_p)^N$. If $N_0$ is the normal subgroup fixing $M_{\mathbf{x}}(\overline{\QQ_p})$, then we see that the maximal abelian quotient of $N$ of exponent dividing $|M_{\mathbf{x}}|$ is a quotient of $N/N_0$ for $p$ larger than $m$ by the definition of $m$ and $L$. This implies that the restriction map $H^1(G_p, M_{\mathbf{x}}) \to H^1(N, M_{\mathbf{x}})$ is zero. So the inflation corresponding to $G_p \to G$  defines an isomorphism
\[H^1\big(G, \,M_{\mathbf{x}}(L)\big) \isoarrow H^1(G_p, M_{\mathbf{x}}).\]
This inflation also gives an injection
\[H^1(G, \mathcal{A}_{\mathbf{x}}(L)) \hookrightarrow H^1(G_p, \mathcal{A}_{\mathbf{x}}(\overline{\QQ_p})),\]
so the local conditions for $M_{\mathbf{x}}$ at $p$ are identified with \eqref{eq:loc_cond_model}.

Our above work then gives a formula in the language characterizing the $\mathbf{x}$ satisfying $\Phi_0(\mathbf{x})$ such that $\mcT_{\mathbf{x} p}$ takes a given value for all sufficiently large $p$, as claimed. 
\end{proof}

We will also need a version of this lemma for weak equivalence classes:

\begin{lemma}
\label{lem:loc_equiv}
If $M$ is constant, and if $\mathcal{L}$ is a weak equivalence class of local conditions of $M$ in the sense of Definition \ref{defn:equiv_local}, there is a Denef--Pas formula $\Phi(\mathbf{x})$ taking $n$ variables from the valued field sort so that, for all sufficiently large $p$ and all $\mathbf{x}$ in $\QQ_p^n$, $\Phi(\mathbf{x})$ holds if and only if the image $\overline{x}$ of $\mathbf{x}$ is usable and $(p, \msL_{\overline{x} p})$ lies in $\mathcal{L}$.
\end{lemma}

\begin{proof}
We fix a finite extension $F$ of $\QQ$ such that $G_F$ acts trivially on $M$. Fix an integer $a$ coprime to $m$. For every prime equal to $a$ mod $m$ that is unramified in $F/\QQ$, we encode a local conditions subgroup as a subgroup of 
\[H^1(D, M(F))\quad\text{with}\quad D = I \rtimes \Z/m\Z,\,\, I = \Z/m\Z,\]
where the action on $I$ is given by 
\[\sigma \tau \sigma^{-1} = \tau^a\] for some generating $\sigma$ in $D/I$ and any $\tau$ in $I$. To make this encoding, we define $L$ and $L_{\text{ur}}$ above $\QQ_p$ as in Lemma \ref{lem:logic_nonconstant}, we choose an identification $\Gal(L/\QQ_p)  \isoarrow D$ taking $\Gal(L/L_{\text{ur}})$ onto $I$, and we choose an embedding $F \hookrightarrow L_{\text{ur}}$ in order to consider the points in $M(L) = M(L_{\text{ur}})$ as points in $M(F)$.

Given any subgroup $\msL$ of $H^1(D, M(F))$, there is a formula $\Phi(\mathbf{x})$ such that $\QQ_p \models \Phi(\mathbf{x})$ exactly when $\QQ_p \models \Phi_0(\mathbf{x})$ and the above procedure can be used to identify the local conditions 
\[\msL_{\mathbf{x} p} \subseteq H^1(\Gal(L/\QQ_p), M(L))\]
with $\msL$ for some choice of $\Gal(L/\QQ_p) \isoarrow D$ and $F \hookrightarrow L_{\text{ur}}$. After taking a union over mod $m$ classes as necessary, we may characterize weak equivalence classes with Denef--Pas formulae.
\end{proof}

In the case where we are focused on a single prime $p$, we can modify this lemma to exert a little more control on the local conditions. 

\begin{lemma}
\label{lem:logic_constant}
If $M$ is constant, there is a formula $\Phi(\mathbf{x}, \overline{z})$ in the Denef--Pas language, with $\mathbf{x}$ in $K^n$ and $\overline{z}$  in some power of $\overline{K}$ such that, for sufficiently large $p$ and $\mathbf{x}, \mathbf{y} \in \QQ_p^n$ satisfying $\Phi_0$, we have
\[
\QQ_p \models \forall \overline{z} \left( \Phi\left(\mathbf{x}, \overline{z}\right) \iff \Phi\left(\mathbf{y}, \overline{z}\right)\right) \quad\textup{if and only if}\quad \msL_{\mathbf{x} p} = \msL_{\mathbf{y} p}.
\]
\end{lemma}

\begin{proof}
In the proof of Lemma \ref{lem:loc_equiv}, we made a number of choices. Our goal is to show that these choices may be encoded by a tuple of elements in $\FFF_p$.

First, we chose the field $L$. This corresponds to a tuple $(a_0, \dots, a_{m-1}, b)$ of elements in $\Z_p$ obeying some conditions.

We also chose an embedding of $F$ into $L_{\text{ur}}$. Writing $F = \Q[\theta]$, this corresponds to a choice of $\theta_p$ in $L_{\text{ur}}$ satisfying the characteristic polynomial of $\theta$.

Finally, we fixed an isomorphism $\Gal(L/\Q_p) \isoarrow D$. An element $\sigma$ in the domain of this map can be specified by its action on the residue field of $L_{\text{ur}}$ and by the image of $\sigma(s)/s$ in this residue field. This defines an encoding $\tau_L: \Gal(L/\Q_p) \to \FFF_p^{m^2 + m}$, and the isomorphism with $D$ may be encoded as a map from the image of this map to $D$.

Suppose we choose a tuple $(a_0', \dots, a'_{m-1}, b', \theta'_p)$ satisfying the conditions above such that $\text{ac}(a_0') = \text{ac}(a_0')$, etc. Take $L'$ to be the extension of $\Q_p$ corresponding to this tuple. The residue field of $L$ is canonically identified with the residue field of $L'$. By Hensel's lemma, there is a unique isomorphism $\iota: L \to L'$ that acts as the identity on the residue field and which satisfies $v_p(\iota(s) - s) \ge 2$. Also by Hensel's lemma, we find that $\theta_p$ must equal $\theta'_p$ for all large enough $p$.

With this setup, we have commutative triangles
\[\begin{tikzcd} \Gal(L/\Q_p) \arrow{r}{\tau_L} \arrow{d}{\sim} & \FFF_p^{m^2 + m} \\ \Gal(L'/\Q_p) \arrow[ur, "\tau_{L'}"'] & \end{tikzcd}\quad\text{and}\qquad \begin{tikzcd} F \arrow{r}{\theta \mapsto \theta_p} \arrow[rd, "\theta \mapsto \theta'_p"'] & L \arrow{d}{\sim} \\ & L' \end{tikzcd},\]
and we are left with a diagram
\[\begin{tikzcd} H^1(\Gal(L/\Q_p), M(L)) \arrow{d}{\sim} \arrow{r} &H^1(D, M(F))\\ H^1(\Gal(L'/\Q_p), M(L')) \arrow{ur} &\end{tikzcd}.\]
The local conditions over $L$ and over $L'$ for a given $\mathbf{x}$ then have equal images in $H^1(D, M(F))$. In other words, for a fixed $\mathbf{x}$, the corresponding image of local conditions in $H^1(D, M(F))$ is determined by the angular components of  $a_0, \dots, a_{m-1}, b, f$. This lets us construct the formula $\Phi(\mathbf{x},  \overline{z})$.
\end{proof}

\begin{proof}[Proof of Proposition \ref{prop:geom_usable}]
For $p$ large enough, we may assume by \cite{PasCrelle} that the formulae $\Phi$ defined in Lemmas \ref{lem:logic_nonconstant} and \ref{lem:logic_constant} have no quantifiers over the valued field sort. We then can find a finite collection of polynomials $g_1, \dots, g_k$ in $\Z[x_1, \dots, x_n]$ such  that, in the formula $\Phi(\mathbf{x})$ or $\Phi\left(\mathbf{x}, \overline{z}\right)$, any appearance of $\mathbf{x}$ is contained in a subterm of the form $\text{ord}(g_i(\mathbf{x}))$ or $\text{ac}(g_i(\mathbf{x}))$ for some $i$.

We may handle the very bad $\overline{x}$ using Lemma \ref{lem:verybad}. If $p$ is large enough and $\overline{x} \in \mcX(\Z/p^2\Z)$ is unusable but not very bad, then, for some $g_i$ and some $\mathbf{x}$, $\mathbf{y}$ above $\overline{x}$, we have
\[\text{ord}(g_i(\mathbf{x})) \ne \text{ord}(g_i(\mathbf{y})) \quad\text{or}\quad \text{ac}(g_i(\mathbf{x})) \ne \text{ac}(g_i(\mathbf{y})).\]
In either case, we must have
\[\text{ord}(g_i(\mathbf{x})) \ge \text{ord}(g_i(\mathbf{x}) - g_i(\mathbf{y})),\]
and Lemma \ref{lem:consistent} gives the proposition.
\end{proof}

\begin{proof}[Proof of Proposition \ref{prop:geom_nonconstant}]
Take $\Phi$ to be the formula constructed in Lemma \ref{lem:logic_nonconstant} or \ref{lem:loc_equiv}. Take $Z_p$ to be the set of usable points in $\mcX(\Z/p^2\Z)$ in the image of some $\mathbf{x}$ satisfying $\Phi$. By \cite[Theorems 8.3.1 and 8.3.2]{Denef Loeser}, $|Z_p|$ is equal to the number of points in $\mcX(\FFF_p[t]/t^2)$ in the image of some $\mathbf{x} \in \mcX(\FFF_p[[t]])$ satisfying $\Phi(\mathbf{x})$ for all sufficiently large $p$.

By writing every element of $\FF_p[t]/t^2$ as $a + bt$ with $a, b \in \FF_p$, there then is a Denef--Pas formula $\Phi'(\overline{x})$ on the general valued field $K = \overline{K}[[t]]$, where $\overline{x}$ is valued in $\overline{K}^{2n}$, such that, for all sufficiently large $p$, $|Z_p|$ equals the number of points in $\FFF_p^{2n}$ satisfying $\Phi'$. Applying Pas's theorem \cite{PasCrelle}, we may assume that $\Phi'$ has no variables in $K$. Furthermore, since points satisfying $\Phi'$ correspond to bad points, we may assume that there is a closed subscheme $V$ of $\mathbb{A}_{\Z}^{2n}$ such that $V_{\QQ}$ has dimension at most $2d - 1$ and such that, for all sufficiently large $p$, $\Phi'(\overline{x})$ is not satisfied for any $\overline{x} \in \FFF_p^{2n}$ not in $V_{\FFF_p}$.

Then \cite{CDM} gives that there is nonnegative rational $\mu$ so
\[\left|\left|Z_p\right| - \mu p^{2d - 1} \right| \le Cp^{2d - 3/2}, \]
with $\mu$ taking one of finitely many values. Whether it takes a given value $c$ is given by a formulae $\Phi_c$ in the language of rings, with $\Phi_c$ holding over $\FFF_p$ if and only if $\mu = c$ in the above formula for large enough $p$ \cite[(2), p. 108]{CDM}.

By \cite[Theorem 2]{Kiefe}, we may suppose $\Phi_c$ is a boolean combination of terms of the form
\[\phi(a_0, \dots, a_m) := \left(\exists x: a_m x^m + \dots + a_0 = 0\right), \]
where $m$ is a positive integer and $a_0, \dots, a_m$ are fixed integers. We then find that there is a Galois extension $L/\Q$ and a conjugacy-invariant subset $\mathscr{C}$ of $\Gal(L/\Q)$ such that
\[\FFF_p \models \Phi_c \iff \text{Frob}\, p \in \mathscr{C}\]
for all large enough $p$. 
\end{proof}

\subsection{Controlling the Tamagawa ratio}
\label{ssec:Tamagawa_AV}
Suppose we are in the situation of Conjecture \ref{conj:AV}, except that we don't necessarily assume that $h$ is geometric. Our goal now is to understand the sum $\sum_{x \in \mcX_{H}'} \mcTbnd(M_x)^{\kappa}$.

Take 
\[G_0 =\Gal\left(\overline{\kappa(\eta)}/\kappa(\eta)\right),\]
where $\kappa(\eta)$ is the residue field at the generic point of $\mcX$. Then $M$ is a $G_0$-module.  For $x$ in $U(\QQ)$, we define the \emph{geometric Tamagawa bound} by
\[\mcT_{\text{g-bnd}}(M_x) = \max_{T \subseteq M} \mcT(M_x, T_x),\]
where the maximum is over all $G_0$-submodules $T$ of $M$. The difference between this and $\mcTbnd(M_x)$ is that we do not consider $G_{\QQ}$-submodules of $M_x$ that do not come from $G_0$-submodules of $M$.

Our first proposition shows that this restriction does not substantially change the Tamagawa bound in families.

\begin{proposition}
\label{prop:Serre_thin}
There is $c > 0$ such that, for any $\kappa \ge 0$, we have
\[\frac{1}{\# \mcX_H'}\sum_{x \in \mcX_H'} \left(\mcTbnd(M_x)^{\kappa} - \mcT_{\textup{g-bnd}}(M_x)^{\kappa}\right) \le \exp\left(- c \frac{\log H}{\log \log H}\right) \]
for all sufficiently large $H$.
\end{proposition}

\begin{proof}
Take $T$ to be a subgroup of $M$ that is not closed under $G_0$. Take $L$ to be the minimal extension of $\kappa(\eta)$ such that $T$ is sent to itself under $G_L$. Then $L/\kappa(\eta)$ is a finite extension. By spreading out, we may choose a dense open set $V$ of $\mcX$ and a finite map $\mathcal{Y} \to V$, where $\mathcal{Y}$ is an irreducible separated scheme whose generic point  has residue field identified with $L$. We may  assume $\mathcal{Y}$ is reduced.

There is a dense open subscheme $W$ of $V$ such that, if $x$ is an integer point of $\mcX$ whose corresponding $\QQ$ point lies in $W(\QQ)$, then $T_x$ is a $G_{\QQ}$-submodule of $M$ exactly when $x$ is in the image of $\mathcal{Y}(\QQ)$. For $p$ larger than some fixed $C_0$ not depending on $x$, this implies that the image of $x$ under the reduction map to $\mcX(\FFF_p)$ lies in the image of the map $\mathcal{Y}(\FFF_p) \to \mcX(\FFF_p)$.

Following Serre's arguments for thin sets of type II \cite{Serre}, we see that there is some $c > 0$ so that, for all sufficiently large $H_0$,
\[\# \big\{p \le H_0\,:\,\, p \text{ prime and } \# \text{image}(\mathcal{Y}(\FFF_p)) \le (1 - c) \cdot \#\mcX(\FFF_p)\big\} \ge \frac{ cH_0}{\log H_0}.\]
By applying effective equidistribution to the product of primes in a set of this form with $H_0$ a small multiple of  $\log H $, we find that
\[\frac{\#\{x \in \mcX_H'\,:\,\, x \in \text{image}(\mathcal{Y}(\FFF_p)) \text{ for all } p > C_0\}}{\#\mcX_H'} \le \exp\left( -c_1 \frac{\log H}{\log \log H}\right) \]
for some $c_1 > 0$ not depending on $H$ and for all sufficiently large $H$.

We similarly may use Serre's arguments for thin sets of type I \cite{Serre} to show that the proportion of $x \in \mcX'_H$ whose $\QQ$ point lies outside $W$ is also at most $\exp\left(-c_1 \frac{\log H}{\log \log H}\right)$. So
\begin{equation}
\label{eq:thinset}
\frac{\#\{x \in \mcX_H'\,:\,\, T_x \text{ is a }G_{\QQ}\text{-submodule}\}}{\#\mcX_H'} \le 2 \exp\left(-c_1 \frac{\log H}{\log \log H}\right)
\end{equation}
for $H$ sufficiently large. Applying Theorem \ref{tSieve} gives that
\[\frac{1}{\# \mcX'_H} \sum_{x \in \mcX'_H} \mcTbnd(M_x)^{2 \kappa} \le (\log H)^C\]
for large enough $H$, where $C > 0$ depends on $\kappa$ but not on $H$. The result follows from this inequality and \eqref{eq:thinset} by the Cauchy--Schwarz inequality.
\end{proof}

Now take $T$ to be a $G_0$-submodule of $M$. For each prime $p$ and every $x$ in $\mcX_H'$, $\mcT_p(T_x)$ is a fraction of the form $a/b$, where $a$ and $b$ are divisors of $|T|$. Given a rational number $t$ of the form $a/b$ with $a$ and $b$ divisors of $|T|$ besides $1$, Proposition \ref{prop:geom_nonconstant} gives that there is a finite Galois extension $L_t/\QQ$ and a class function $f_t: \Gal(L_t/\QQ){/\sim} \to \QQ_{\ge 0}$ such that, for all sufficiently large primes $p$, among the usable $\overline{x}$ in $\mcX(\Z/p^2\Z)$, the proportion of $\overline{x}$ such that 
\[\mcT_p(T_x) = t\text{ for any/every }  x \text{ over } \overline{x}\]
is $f_t(\text{Frob}\, p) p^{-1} + \mathcal{O}(p^{-3/2})$.

From this observation, we may prove the following:

\begin{proposition}
\label{prop:beta}
For $t$ as above, take $\gamma(t, T)$ to equal the average of $f_t$ over $\Gal(L_t/\QQ)$. Then there is $C > 0$ so, for $\kappa \ge 0$, we have
\[
\exp \exp (C\kappa)^{-1} (\log H)^{\beta(T, \kappa)}\, \le\, \frac{1}{\#\mcX_{H}'} \sum_{x \in \mcX_H'} \mcT(M_x, T_x)^{\kappa} \,\le\, \exp \exp (C\kappa) (\log H)^{\beta(T, \kappa)} 
\]
for all sufficiently large $H$, where
\[
\beta(T, \kappa) = \sum_t \gamma(t, T)(t^{\kappa} - 1).
\]
Here, the sum is over all rational numbers of the form $a/b$ with $a, b$ divisors of $|T|$.
\end{proposition}

\begin{proof}
By Theorems \ref{tSieve} and \ref{tSieve2}, this reduces to a sum of multiplicative functions. This sum may be handled as in \cite[Appendix A]{FI}.
\end{proof}

This result and Proposition \ref{prop:Serre_thin} together imply
$$
% \label{eq:g-bnd_asymp}
\frac{1}{\#\mcX_{H}'} \sum_{x \in \mcX_H'} \mcTbnd(M_x)^{\kappa} \asymp (\log H)^{\beta(\kappa)},
$$
where $\beta(\kappa)$ is the maximum of the $\beta(T, \kappa)$ over all $G_0$-submodules $T$ of $M$. If $M$ is constant, Theorem \ref{thm:AV} then gives
$$
%\label{eq:AV_b}
\frac{1}{\#\mcX_{H}'} \sum_{x \in \mcX_H'} (\# \Sel\, M_x)^{\kappa} \asymp (\log H)^{\beta(\kappa)}.
$$
Our main conjecture for abelian varieties is that this remains true even when $M$ is nonconstant.

We will give two general situations where it is possible to calculate $\beta(\kappa)$ explicitly.

\begin{theorem}
%\label{thm:isogenyless}
Choose a prime $\ell$ such that the $\ell$-torsion finite group scheme $\mcA_{\eta}[\ell]$ over the generic point $\eta$ of $\mcX$ has no proper nonzero rational subgroup. Then there is some $C > 0$ such that, for all $\kappa \ge 0$, we have
\[
\frac{1}{\# \mcX_H'} \sum_{x \in \mcX_H'} \mcTbnd(\mcA_x[\ell])^{\kappa} \le \exp \exp(C\kappa).
\]
In the case that $\mcA_{\eta}[\ell]$ is constant, we then also have
\[
\frac{1}{\# \mcX_H'} \sum_{x \in \mcX_H'} (\#\Sel_{\ell} \, \mcA_x)^{\kappa} \le \exp \exp(C\kappa).
\]
\end{theorem}

\begin{proof}
With $M = \mcA_{\eta}[\ell]$, the only $G_0$-submodules of $M$ are the trivial submodule, which has Tamagawa bound $1$, and the module itself. But we have $\mcT_p(M_x) = 1$ for all primes $p$ other than $\ell$, as follows from the exact sequence 
\[
0 \to \mcA_x[\ell](\QQ_p) \to \mcA_x(\QQ_p) \xrightarrow{\,\,\cdot \ell\,\,} \mcA_x(\QQ_p) \to \msL_{xp} \to 0.
\] 
Then the Tamagawa bound $\mcT(M_x, M_x)$ is also bounded, and the first result follows from Proposition \ref{prop:Serre_thin}. The second then follows from Theorem \ref{thm:AV}.
\end{proof}

We also give a more explicit form of $\beta(\kappa)$ in the case that $\mcA$ is generically an elliptic curve over affine space.

\begin{mydef}
\label{defn:Tate_alg}
Take $\mcX = \text{Spec}\, \Z[t_1, \dots, t_n]$, and choose $a, b \in \Z[t_1, \dots, t_n]$ such that $\Delta = 4a^3 + 27b^2$ is nonzero. Take $\mcA$ to be the scheme
\[y^2 = x^3 + ax + b.\]
This defines an elliptic curve at all geometric points of $\mcX$ where $\Delta$ is nonzero.

Taking $\eta$ to be the generic point of $\mcX$, $\mcA_{\eta}$ is an elliptic curve over the function field $\kappa(\eta)$. We write this curve as $E$.  Choose a prime $\ell$, and choose an isogeny $\lambda: \mcA_{\eta} \to \mcA'_{\eta}$ of degree $\ell$ defined over $\kappa(\eta)$. This may be given in the form
\[E': y^2 = x^3 + a'x + b',\]
where $a'$ and $b'$ are polynomials in $\Z[t_1, \dots, t_n]$ such that $\Delta' = 4{a'}^3 + 27{b'}^2$ is nonzero. This formula then also gives a scheme $\mcA'$ over $\mcX$ which is an elliptic fibration on the open set where $\Delta'$ is nonzero and whose generic fiber over $\mcX$ is $\mcA'_{\eta}$.

Write $\Delta$ in the form $d h_1^{a_1} \dots h_k^{a_k}$, where the $h_i$ are primitive irreducible polynomials that are nonconstant and pairwise coprime, where $d$ is a nonzero integer, and where the $a_i$ are positive integers. Take $Z_i$ to be the closed subscheme of $\mcX$ corresponding to $h_i$.

Given $i \le k$, we take $\widehat{\mcO}(Z_i)$ to be the completion of $\Z[x_1, \dots, x_n]$ at the prime ideal $(h_i)$, and take $\kappa(Z_i)$ to be the residue field of this domain. We may speak of the reduction type for the N\'{e}ron model for $E$ over $\widehat{\mcO}(Z_i)$, and we may calculate this reduction type by Tate's algorithm.

For each $i \le k$, we define rational numbers $\omega_{-1}(Z_i)$ and $\omega_1(Z_i)$ as follows:

\begin{enumerate}
\item Suppose $E$ has reduction of type I${}_{\nu}$ at $Z_i$ for some $\nu > 0$; that is, suppose $E$ has multiplicative reduction. Take I${}_{\nu'}$ to be the reduction type for $E'$ at $Z_i$. Take $\beta = 1/2$ if the reduction of $E$ is nonsplit and either $\ell \ne 2$ or both $\nu$ and $\nu'$ are even, and take $\beta = 1$ otherwise. We then define
\[\left(\omega_{-1}(Z_i), \,\omega_{1}(Z_i)\right) = \begin{cases} (\beta, 0) &\text{ if } \nu = \ell \nu' \\ (0, \beta) &\text{ if } \nu' = \ell\nu. \end{cases}\]

\item Suppose $\ell = 3$, and suppose $E$ has reduction type IV or IV${}^*$ at $Z_i$ and that $-3$ is not a square in $\kappa(Z_i)$. Take $c(E)$ to be the Tamagawa number of $E$ over $Z_i$; this is the number of rational components of multiplicity one in the special fiber of the N\'{e}ron model. Similarly define $c(E')$.

Then we define
\[\left(\omega_{-1}(Z_i),\, \omega_{1}(Z_i)\right) = \begin{cases} (1/2,\, 0) &\text{ if } c(E) = 3 \text{ and } c(E') = 1  \\ (0, \,1/2) &\text{ if } c(E) = 1 \text{ and } c(E') = 3   \\ (1/4, \,1/4) &\text{ otherwise.}\end{cases}\]
\item Suppose $\ell = 2$, and suppose $E$ has reduction type I${}^*_{\nu}$ at $Z_i$ for some $\nu \ge 0$. We suppose that there is a quadratic extension $L$ of $\kappa(Z_i)$ such that the number of $L$-rational components of multiplicity $1$ in the special fiber of the N\'{e}ron model for  $E$ differs from the same count for $E'$.

 We then define
\[\left(\omega_{-1}(Z_i),\, \omega_{1}(Z_i)\right)  = \begin{cases} (1/2,\, 0) &\text{ if } c(E) = 4\text{ and } c(E') = 2 \\ (0, \,1/2) &\text{ if } c(E) = 2 \text{ and } c(E') = 4 \\ (1/4, \,1/4) &\text{ otherwise.}\end{cases}\]
\item In all other cases, take $\omega_{-1}(Z_i) = \omega_1(Z_i) = 0$. 
\end{enumerate}
\end{mydef}

\begin{proposition}
\label{prop:Tate_alg}
Take $T$ to be the kernel of the degree $\ell$-isogeny $\lambda: E \to E'$ considered above over the field $\kappa(\eta)$. Then the exponent $\beta(T, \kappa)$ appearing in Proposition \ref{prop:beta} is given by
\[\beta(T, \kappa) = \sum_{i= 1}^k \omega_1(Z_i) \left(\ell^{\kappa} - 1\right) + \omega_{-1}(Z_i)\left( \ell^{-\kappa} - 1\right).\]
\end{proposition}

\begin{proof}
Fix a dense open set $W$ of $\mcX$ such that $\lambda_W: \mcA_W \to \mcA_W'$ is a degree $\ell$-isogeny of elliptic curves on every geometric fiber of $W$. For any $x$ in $W(\QQ)$, we may then consider $\ker \, \lambda_x$ as a Galois module decorated with the local conditions coming from its inclusion into $\mcA_x$. For a given prime $p$, $\msL_{xp}$ is identified with the cokernel of $\mcA_x(\QQ_p) \to \mcA'_x(\QQ_p)$, so \cite[Lemma 4.2]{DokDok} gives
\[
\mcT_p(\ker\, \lambda_x) = \frac{c_p(\mcA'_x)}{c_p(\mcA_x)}
\]
so long as $p \ne \ell$, where $c_p$ denotes the Tamagawa number over $\QQ_p$ for an elliptic curve.

Now choose a point $x \in \mcX(\Z)$ whose associated $\QQ$ point lies in $W$. If $\mcA_x$ has bad reduction at $p$, then the $\FFF_p$ point corresponding to $x$ must lie in $Z_i(\FFF_p)$ for some $i \le k$.

For a prime $p$, call a point $x$ in $\mcX(\Z_p)$ \emph{exceptional} if the corresponding $\FFF_p$ point lies in $Z_i$, but the reduction type (Kodaira symbol) of $\mcA_x$ over $\Z_p$ does not match that of $\mcA$ over $Z_i$. By Tate's algorithm, there is some polynomial $h$ in $\Z[t_1, \dots, t_n]$ not depending on $p$ such that, for any exceptional $x$, $v_p(h(x))$ is larger than the valuation of $h$ at $Z_i$. As in Lemma \ref{lem:verybad}, we find that the image of the exceptional points in $\mcX(\Z/p^2\Z)$ has size at most $Cp^{2n - 2}$, where $C >0$ does not depend on $p$. This also bounds the number of $\Z/p^2\Z$ points whose corresponding $\FFF_p$ points lie in the intersection of any two distinct $Z_i$.

Our proof relies on the Lang--Weil estimate, whose form we now recall. A \emph{Frobenian function of average one} will be a class function $f: \Gal(K/\QQ) \to \Q_{ \ge 0}$ defined on the Galois group of a finite Galois extension $K$ of $\QQ$ whose average is $1$. The Lang--Weil bound and Lemma \ref{lem:smoothlift} gives that
\[\left|\# Z_i(\Z/p^2\Z) - f(\Frob\, p) p^{2n-1}\right| = \mathcal{O}(p^{2n- 3/2})\]
for some Frobenian function of average $1$, where the implicit constant does not depend on $p$.

Now suppose $L$ is some quadratic extension of $\kappa(Z_i)$. Then we may choose some dense open subscheme $V_i$ of $Z_i$ and a degree $2$ map $W_i \to V_i$ from another separated scheme whose generic fiber corresponds to $L/\kappa(Z_i)$. Lang--Weil gives that $W_i(\FFF_p)$ has size $p^{n-1}$ times a Frobenian function of average one, up to manageable error. This implies that the image of $W_i(\FFF_p)$ in $V_i(\FFF_p)$ has size $\tfrac{1}{2}p^{n-1}$ times a Frobenian function of average one up to manageable error.

But if an unexceptional point $x \in \mcX(\Z_p)$ has its corresponding $\FFF_p$ point $\overline{x}$ lying in $V_i$, we find that $L_{\overline{x}}/\FFF_p$ is a nontrivial extension exactly when $\overline{x}$ lies outside the image of $W_i(\FFF_p)$. We will repeatedly use this fact to control Tamagawa ratios.

Suppose to start that $E$ has multiplicative reduction at $Z_i$. If $E$ has split multiplicative reduction at $Z_i$, and if $x$ is an unexceptional $\Z_p$ point that reduces mod $p$ to a point in $V_i$, then $\mcA_x$ has split multiplicative reduction at $p$. So the number of points in $Z_i(\Z/p^2\Z)$ that correspond to split multiplicative reduction is $f(\Frob\, p)p^{2n-1} + \mathcal{O}(p^{n-3/2})$, where $f$ is Frobenian of average $1$.

If $E$ has nonsplit multiplicative reduction at $Z_i$, then by Tate's algorithm there is a minimal quadratic extension $L/\kappa(Z_i)$ over which the reduction of $E$ becomes split. By the above argument for degree $2$ covers, we find that the number of points  $\mcX(\Z/p^2\Z)$ reducing mod $p$ to $Z_i$ corresponding to split multiplicative reduction is $\tfrac{1}{2}p^{2n - 1}$ times a Frobenian function of average $1$, up to manageable error. This is also the form of the count for points of nonsplit reduction.

From Proposition \ref{prop:beta} and \cite[Table 1]{DokDok}, this accounts for the contribution of $Z_i$ to $\beta(T, \kappa)$ if $E$ has multiplicative reduction. We now continue to some other types that may have nontrivial Tamagawa ratio.

Next suppose $E$ has reduction type IV or IV${}^*$ at $Z_i$, that $\ell = 3$, and that $\kappa(Z_i)$ does not contain $\mu_3$. Take $L$ to be the minimal extension of $\kappa(Z_i)$ so that the three multiplicity one components of the special fiber of the N\'{e}ron model for $E$ are rational over $L$, and define $L'$ similarly for $E'$. These fields have degree at most $2$. If they are equal, as happens exactly when $\mu_3$ is contained in $\kappa(Z_i)$, then the Tamagawa numbers $c_p(\mcA_x)$ and $c_p(\mcA'_x)$ are equal for $x$ reducing mod $p$ to $Z_i$, outside a set whose image in $\mcX(\Z/p^2\Z)$ is negligible.

If $L$ and $L'$ are not equal, then we find that $c_p(\mcA_x)$ and $c_p(\mcA'_x)$ are unequal for a set of $x$ whose image in $\mcX(\Z/p^2\Z)$ is proportional $p^{2n-1}$ times some Frobenian function. The cases $L = \kappa(Z_i)$ and $L' = \kappa(Z_i)$, which correspond to $c(E) = 3$ and $c(E') = 3$ respectively, are handled as for the nonsplit multiplicative reduction case.

This leaves the case that $c(E) = c(E') = 1$. For this case, we apply the degree $2$-cover argument above to the three distinct quadratic extensions $\kappa(Z_i)(\mu_3)$, $L$, and $L'$. Taken together, these give that the set of points with $c_p(\mcA_x) = 3$ and $c_p(\mcA'_x) = 1$ has image in $\mcX(\Z/p^2\Z)$ of size $\tfrac{1}{4}p^{2n-1}$ times a Frobenian function of average $1$, up to manageable error. The same form of count gives the number with $c_p(\mcA_x) = 1$ and $c_p(\mcA'_x)=3$. This accounts for the value of $\omega_1$ and $\omega_{-1}$ in the IV and IV${}^*$ cases.

Finally, we have the case I${}^*_{\nu}$ with $\ell = 2$, which is entirely analogous to the case IV/IV${}^*$. This exhausts the reduction types for $Z_i$ such that the Tamagawa ratio can be anything other than $1$ for $p$ not equal to $\ell$, and the result follows.
\end{proof}

\begin{proof}[Proof of Theorem \ref{tGeom1}]
Applying Theorem \ref{thm:AV} as in Example \ref{ex:ell_counts}, we have
\[\frac{1}{\#\mcA_{\le H}}\sum_{E \in \mcA_{\le H}} (\# \Sel_{\ell}\, E)^{\kappa} \asymp \frac{1}{\#\mcA_{\le H}}\sum_{E \in \mcA_{\le H}} \mcTbnd(E[\ell])^{\kappa}.\]
The result then follows from Propositions \ref{prop:Serre_thin}, \ref{prop:beta} and \ref{prop:Tate_alg}.
\end{proof}
\section{Examples}
\label{sExa}
We now give some applications of Theorem \ref{thm:main} and of the theory for geometric families of abelian varieties developed in the previous section. In particular, we show that Theorems \ref{thm:quadtwist} and \ref{tGeom2} follow from Theorem \ref{thm:AV} in Section \ref{sGeom} (see respectively Example \ref{ex:abelian_twist} and Example \ref{ex:full_3tor}). We then prove Theorem \ref{tClass} in Section \ref{sClass} and Theorem \ref{t3Selmer} in Section \ref{s3Selmer}.

\subsection{Examples from algebraic geometry}
\label{sGeom}

\begin{example}
\label{ex:par_2tor}
Consider the elliptic fibration
\[\mcA: y^2 = x(x^2 + u_1x + u_2)\]
over the open subscheme of $\mcX = \Spec\,\Z[u_1, u_2]$ given by $u_2(u_1^2 - 4u_2) \ne 0$. The points on $\mcX$ correspond to pairs of integers $(a_1, a_2)$, and we define a height by 
\[
h(a_1, a_2) = \max\big(a_1^2, |a_2|\big).
\]
The generic fiber of $\mcA$ over $\mcX$ is an elliptic curve $E$ over $\QQ(u_1, u_2)$. There is a unique degree $2$-isogeny $\lambda: E \to E_0$ of elliptic curves over $\QQ(u_1, u_2)$. The codomain of this isogeny is
\[E_0 = x(x^2 - 2u_1x + (u_1^2 - 4u_2)).\]
We summarize the reduction types of $E$ and $E_0$ in Table \ref{tab:par_2tor}.

\begin{table}[ht]
    \centering
    \begin{tabular}{c|c | c }
        Divisor & $E$ & $E_0$ \\\hline
     $(u_2)$ & I${}_{2}$ nonsplit & I${}_1$ \\
     $(u_1^2 - 4u_2) $  & I${}_1$ nonsplit  & I${}_2$ 
    \end{tabular}
    \caption{Reduction types for Example \ref{ex:par_2tor}}
    \label{tab:par_2tor}
\end{table}
Applying Proposition \ref{prop:Tate_alg} gives
\[\frac{1}{\# \mcX_H'} \sum_{x \in \mcX_H'} \mcTbnd(\mcA_x[2])^{\kappa} \asymp (\log H)^{\beta(\kappa)}\]
with 
\begin{equation}
    \label{eq:p_2_tor_beta}
    \beta(\kappa) = 2^{\kappa} + 2^{-\kappa} - 2
\end{equation}
for all $\kappa \ge 0$.

Taking $\lambda': E_0 \to E$ to be the dual isogeny to $\lambda$, we have an exact sequence
\[\Sel\, \ker \lambda_x \to \Sel_2 \mcA_x \to \Sel\, \ker \lambda'_x,\]
and we also have
\[\frac{1}{\# \mcX_{H}'} \sum_{x \in \mcX_H'} \mcTbnd(\ker \, \lambda_x \oplus \ker \,\lambda'_x)^{\kappa} \asymp (\log H)^{\beta(\kappa)},\]
with $\beta(\kappa)$ defined as before. The family of decorated Galois modules $\ker \lambda_x \oplus \ker \lambda'_x$ satisfy the conditions of Theorem \ref{thm:AV}, so we may conclude
\[\frac{1}{\# \mcX_{H}'} \sum_{x \in \mcX_H'} (\#\Sel_2\, \mcA_x)^{\kappa} \asymp (\log H)^{\beta(\kappa)},\]
with $\beta(\kappa)$ given by \eqref{eq:p_2_tor_beta}.

This gives an example where Theorem \ref{thm:AV} is applicable even when the $\ell$-torsion is not a constant Galois module. An obvious deficiency of this approach is that we cannot show that the moments of $\# \Sel_2 \,\mcA_x/\mcTbnd(\mcA_x[2])$ are finite.

The distribution of $\mcTbnd(\mcA_x[2])$ in this family was first given by Klagsbrun and Lemke Oliver in \cite{KLOTama}. The average size of $\# \Sel_3 \,\mcA_x$ in this family was found by Bhargava--Ho \cite{BH}.
\end{example}

\begin{example}
\label{ex:YuI}
Consider the elliptic fibration
\[\mcA: y^2 = x(x - u_1)(x - u_2)\]
over the open subscheme of $\mcX = \Spec\, \Z[u_1, u_2]$ given by $2u_1u_2(u_1 - u_2) \ne 0$. Given $(a_1, a_2)$ an integer point of $\mcX$, we take $h(a_1, a_2) = \max(|a_1|, |a_2|)$. 

\begin{table}[!ht]
    \centering
    \begin{tabular}{c|c | c |c | c}
        Divisor & $E$ & $E_1$ & $E_2$ & $E_3$ \\\hline
     $(u_1)$ & I${}_{2}$ nonsplit & I${}_1$ & I${}_4$ & I${}_1$ \\
     $(u_2) $  & I${}_2$ nonsplit  & I${}_4$ & I${}_1$ & I${}_1$\\
     $(u_1 - u_2)$ & I${}_2$ nonsplit & I${}_1$ & I${}_1$ & I${}_4$
    \end{tabular}
    \caption{Reduction types for Example \ref{ex:YuI}}
    \label{tab:YuI}
\end{table}

The generic fiber is an elliptic curve $E$ over $\QQ(u_1, u_2)$ with three distinct isogenies of degree $2$.  Take $\varphi_1, \varphi_2, \varphi_3$ to be the degree $2$-isogenies with nontrivial kernel point given by  $(u_1, 0)$, $(u_2, 0)$, and $(0, 0)$, respectively, and take $E_i$ to be the codomain of the isogeny $\varphi_i$. Applying Tate's algorithm gives the reduction types in Table \ref{tab:YuI}. So Proposition \ref{prop:Tate_alg} and Theorem \ref{thm:AV} give
\[\frac{1}{\# \mcX_{H}'} \sum_{x \in \mcX_H'} (\#\Sel_2\, \mcA_x)^{\kappa} \asymp \frac{1}{\# \mcX_{H}'} \sum_{x \in \mcX_H'} \mcTbnd(\mcA_x[2])^{\kappa} \asymp (\log H)^{\beta(\kappa)}\]
with
\[\beta(\kappa) = \max\left(0, \, 2^{-\kappa + 1} + 2^{\kappa - 1} - \frac{5}{2}\right).\]

So, as $E$ varies among elliptic curves with full rational $2$-torsion ordered by na\"{i}ve height, the average of $(\#\Sel_2 \, E)^{\kappa}$ is bounded for $\kappa \le 2$ and unbounded for $\kappa > 2$.

In the case $\kappa = 1$, this recovers the main result of \cite{YuI}.
\end{example}
\begin{example}
\label{ex:YuII}
Take $b$ to be a fixed nonzero integer, and consider the elliptic fibration
\[\mcA: y^2 = x(x-b)(x-u)\]
above the subset of $\mcX = \Spec \,\Z[u]$ given by $2bu(u-b) \ne 0$. We define the height of a point $a$ in $\mcX(\Z)$ to be its absolute value.
\begin{table}[!ht]
    \centering
    \begin{tabular}{c|c | c |c | c}
       Divisor & $E$ & $E_1$ & $E_2$ & $E_3$ \\\hline
     $(u)$ & I${}_{2}$,$\,\,$ split iff $-b$ is square & I${}_4$ & I${}_1$ & I${}_1$ \\
     $(u - b)$ & I${}_{2}$, $\,\,$ split iff $b$ is square & I${}_1$ & I${}_1$ & I${}_4$
    \end{tabular}
    \caption{Reduction types for Example \ref{ex:YuII}} 
    \label{tab:YuII}
\end{table}

Take $E$ to be the corresponding elliptic curve over $\Q(u)$. Take $\varphi_1$, $\varphi_2$, $\varphi_3$ to be the degree $2$-isogenies from $E$ with nontrivial kernel point given by $(b, 0)$, $(u, 0)$, and $(0, 0)$, respectively, and take $E_i$ to be the codomain of $\varphi_i$. The reduction types are given as in Table \ref{tab:YuII}. So Proposition \ref{prop:Tate_alg} and Theorem \ref{thm:AV} give
\[\frac{1}{\# \mcX_{H}'} \sum_{x \in \mcX_H'} (\#\Sel_2\, \mcA_x)^{\kappa} \asymp \frac{1}{\# \mcX_{H}'} \sum_{x \in \mcX_H'} \mcTbnd(\mcA_x[2])^{\kappa} \asymp (\log H)^{\beta(\kappa)}\]
with
\[\beta(\kappa) =\begin{cases}  \max\left(0, \, 2^{-\kappa} + 2^{\kappa} -  2 \right) &\text{ if } b \text{ or } -b \text{ is square} \\ \max\left(0, \, 2^{-\kappa} + 2^{\kappa - 1} -  3/2 \right) &\text{ otherwise.}\end{cases}\]

If neither $b$ nor $-b$ is square, $\beta(\kappa)$ is $0$ for $\kappa$ in the range $[0, 1]$, so $2$-Selmer moments are bounded precisely for $\kappa$ in this range. If either $b$ or $-b$ is a square, then $2$-Selmer moments are unbounded for $\kappa > 0$.

In the case that $\kappa = 1$, this recovers the main result of \cite{YuII}.
\end{example}

\begin{example}
Consider the elliptic fibration
\[\mcA: y^2 = x^3 + u\]
above the open subscheme of $\mcX = \Spec\, \Z[u]$ given by $6u \ne 0$. This is a sextic twist family.
We again define the height of a given $a$ in $\mcX(\Z)$ by its absolute value.

Taking $E$ to be the generic fiber of $\mcA$, we have that $E$ is an elliptic curve over $\QQ(u)$. There is a unique rational isogeny $\varphi$ of degree $3$, from $E$ to $E_0: y^2 = x^3 - 27u$.

The curve $E$ has reduction type II at $(u)$, so, despite the existence of the isogeny $\varphi: E \to E_0$, Proposition \ref{prop:Tate_alg} gives
\[\frac{1}{\# \mcX_H'} \sum_{x \in \mcX_H'} \mcTbnd(\mcA_x[3])^{\kappa} \asymp 1.\]
Conjecture \ref{conj:AV} then claims that $\Sel_3\, \mcA_x$ has bounded $\kappa$ moment for any $\kappa \ge 0$. Bounding $\# \Sel_3\, \mcA_x$ by two isogeny Selmer groups as in Example \ref{ex:par_2tor} and applying the Cauchy--Schwarz inequality to a result of Bhargava--Elkies--Shnidman \cite{BES} shows that this holds for all $\kappa \le 1/2$.
\end{example}

\begin{example}
\label{ex:cubic_twist}
Choose a nonzero integer $b$, and consider the elliptic fibration
\[\mcA: y^2 = x^3 + bu^2\]
above the open subscheme of $\mcX = \Spec\, \Z[u]$ given by $6bu \ne 0$. This is a cubic twist family.

\begin{table}[ht]
    \centering
    \begin{tabular}{c|c | c }
        Divisor & $E$ & $E_0$ \\\hline
     $(u)$ & IV, $c = 3$ iff $b$ is square  & IV, $c = 3$ iff $-3b$ is square
    \end{tabular}
    \caption{Reduction types for Example \ref{ex:cubic_twist}}
    \label{tab:cubic_twist}
\end{table}
Taking $E$ to be the generic fiber of this fibration, we have a unique degree $3$-isogeny $\varphi: E \to E_0$ with $E_0$ given by $y^2 = x^3 - 27bu^2$. We summarize the reduction type and Tamagawa numbers in Table \ref{tab:cubic_twist}.

We then have

\[\frac{1}{\# \mcX_{H}'} \sum_{x \in \mcX_H'} \mcTbnd(\mcA_x[3])^{\kappa} \asymp (\log H)^{\beta(\kappa)}\]
with
\[\beta(\kappa) = \begin{cases} 0 &\text{ if } b \text{ is square}\\ \tfrac{1}{2}(3^{\kappa} - 1) &\text{ if } -3b \text{ is square}\\ \tfrac{1}{4}(3^{\kappa} + 3^{-\kappa} - 2) &\text{ otherwise}.\end{cases} \]
The same trick as Example \ref{ex:par_2tor} gives the somewhat unsatisfying result
\[\frac{1}{\# \mcX_{H}'} \sum_{x \in \mcX_H'} (\#\Sel_3\, \mcA_x)^{\kappa} \asymp (\log H)^{\beta(\kappa)} \quad\text{if } \,b \,\text{ is not square}.\]
However, unlike Example \ref{ex:par_2tor}, there are known tools for proving more satisfying distributional results for the $3$-Selmer groups in this family. See \cite{KoymansSmith} for more details.
\end{example}

\begin{example}
Consider the elliptic fibration
\[\mcA: y^2 + u_1xy + u_2y = x^3\]
above the open subscheme of $\mcX = \Spec\, \Z[u_1, u_2]$ given by $3u_2(u_1^3 - 27u_2) \ne 0$. This is the family of elliptic curves with a marked $3$-torsion point, at $(0, 0)$. The height of a given $(a_1, a_2)$ in $\mcX(\Z)$ is given by $\max(|a_1|^3, |a_2|)$.

Taking $E$ to be the generic fiber of $\mcA$ above $\mcX$, $E$ has a unique rational degree $3$-isogeny $\varphi$, to 
\[E_0: y^2 +u_1xy + u_2y = x^3  - 5u_1u_2 x - (u_1^3u_2 + 7u_2^2);\]
see \cite{Hadano}. We summarize the reduction types in Table \ref{tab:3tor}.

\begin{table}[ht]
    \centering
    \begin{tabular}{c|c | c }
        Divisor & $E$ & $E_0$ \\\hline
     $(u_2)$ & I${}_3$ split   & I${}_1$ split \\ $(u_1^3 - 27u_2)$ & I${}_1$ nonsplit  & I${}_3$ nonsplit
    \end{tabular}
    \caption{Reduction types for Example \ref{ex:cubic_twist}}
    \label{tab:3tor}
\end{table}

Then Proposition \ref{prop:Tate_alg} gives
\[\frac{1}{\# \mcX_H'} \sum_{x \in \mcX_H'} \mcTbnd(\mcA_x[3])^{\kappa} \asymp (\log H)^{\beta(\kappa)}\]
with
\[\beta(\kappa) = \max\left(0,\,\tfrac{1}{2}3^{\kappa} + 3^{-\kappa} - \tfrac{3}{2}\right).\]
So $(\#\Sel_3\, \mcA_x)^{\kappa}$ has unbounded average if  $\kappa > \log_3 2$. Conjecture \ref{conj:AV} would give that these moments are bounded for $\kappa \le \log_3 2$.

In particular, $\Sel_3 \, \mcA_x$ has unbounded average size. This may be compared to the $2$-Selmer groups in this family, which have average size at most $3$ \cite{BH}.
\end{example}

\begin{example}
\label{ex:full_3tor}
Consider the elliptic fibration
\[\mcA: y^2 + 3u_1xy + (u_1^3 + u_2^3)y = x^3\]
above the open subscheme of $\mcX = \Spec\, \Z[u_1, u_2]$ given by $3u_2(u_1 + u_2)(u_1^2 - u_1u_2 + u_2^2) \ne 0$. This is the family of elliptic curves over $\QQ$ with $3$-torsion isomorphic to $\Z/3\Z \oplus \mu_3$. The height of a given $(a_1, a_2)$ in $\mcX(\Z)$ will be taken to be $\max(|a_1|, |a_2|)$. 

\begin{table}[!ht]
    \centering
    \begin{tabular}{c|c | c |c }
        Divisor & $E$ & $E_0$ & $E_1$  \\\hline
     $(u_2)$ & I${}_{3}$ nonsplit & I${}_9$ & I${}_1$  \\
     $(u_1 + u_2) $  & I${}_3$ split  & I${}_1$ & I${}_9$\\
     $(u_1^2 - u_1u_2 + u_2^2)$ & I${}_3$ split & I${}_1$ & I${}_1$
    \end{tabular}
    \caption{Reduction types for Example \ref{ex:full_3tor}}
    \label{tab:full_3tor}
\end{table}

Taking $E$ to be the generic fiber of $\mcA$, $E$ has two distinct rational degree $3$-isogenies, one with kernel isomorphic to $\Z/3\Z$ and one with kernel isomorphic to $\mu_3$. Take $E_0$ to be the codomain of the former, and $E_1$ of the latter. We summarize reduction types in Table \ref{tab:full_3tor}.

Then Theorem \ref{thm:AV} and Proposition \ref{prop:Tate_alg} give
\[\frac{1}{\# \mcX'_H} \sum_{x \in \mcX'_H} \left(\#\Sel_3\,\mcA_x\right)^{\kappa} \asymp (\log H)^{\beta(\kappa)}\]
with
\[\beta(\kappa) = \max\left(0, 3^{\kappa} + \tfrac{3}{2} 3^{-\kappa} - \tfrac{5}{2}\right).\]
So $(\# \Sel_3 \,\mcA_x)^{\kappa}$ has bounded average exactly when $\kappa \le \log_3 3/2 \approx .369$. This agrees with work of Phillips, who shows that the average size of the $3$-Selmer group is unbounded in this family \cite{Phillips}. 

At the same time, if $T$ is the kernel of either rational degree $3$-isogeny, we have $\beta(\kappa, T) < 0$ for $\kappa = 0.3$. By Markov's inequality and Proposition \ref{prop:beta}, we get that there is $C > 0$ so that
\[\lim_{H \to \infty} \frac{1}{\# \mcX'_H} \#\{x \in \mcX'_H\,:\,\, \mcTbnd(\mcA_x[3]) \ge C\} = 0.\]
So Theorem \ref{tGeom2} is a consequence of Theorem \ref{thm:AV}.
\end{example}

\begin{example}
Choose integers $a$ and $b$ so $4a^3 + 27b^2$ is nonzero. Then we have an elliptic fibration 
\[
\mcA: y^2 = x^3 + u^2 ax + u^3b
\]
over the open subscheme of $\mcX = \Spec\, \Z[u]$ given by $2(4a^3 + 27b^2)u \ne 0$. This is the quadratic twist family of the curve $y^2 = x^3 + ax + b$.

In this case, the generic fiber of $\mcA$ has reduction type I${}_0^*$ at $(u)$, so $\mcTbnd(\mcA_x[\ell])$ can only have unbounded moments in the case $\ell = 2$ by Proposition \ref{prop:Tate_alg}. In this case, Theorem \ref{thm:AV} gives
\[\frac{1}{\# \mcX_H'} \sum_{x \in \mcX_H'} (\# \Sel_2\, \mcA_x)^{\kappa} \asymp (\log H)^{\beta(\kappa)},\]
where $\beta(\kappa) = 0$ unless $E := \mcA_\QQ$ has a unique rational degree $2$-isogeny; write $E_0$ for the codomain of this isogeny. In this case, 
\[
\beta(\kappa) = 
\begin{cases} 
0 &\text{ if } E[2] \cong E_0[2] \text{ over } \QQ(u) \\ 
\tfrac{1}{2}(2^{\kappa}  - 1) &\text{ if } E_0(\QQ(u))[2] \cong (\Z/2\Z)^2 \\ 
\tfrac{1}{4}(2^{\kappa} + 2^{-\kappa} - 2) &\text{ otherwise.}
\end{cases}
\]
This result was essentially already known. The Tamagawa bounds in this family have been understood for many years \cite{Xiong, Klagsbrun}, with finer distributional results established in \cite{HB, Kane, Smi22b, Smi25}.

In the case that $\ell \ne 2$, we know that the expected size of $\mcTbnd(\mcA_x)^{\kappa}$ is bounded as $x$ varies for any $\kappa$. By Conjecture \ref{conj:AV}, we then predict
\[\frac{1}{\# \mcX'_H} \sum_{x \in \mcX'_H} \left(\Sel_{\ell}\,\mcA_x\right)^{\kappa} \asymp 1\]
for $\ell$ odd. The only case where this is known is when $\mcA_1$ has a rational $3$-isogeny and $\kappa \le 1/2$, due to work of Bhargava--Klagsbrun--Lemke Oliver--Shnidman \cite{BKLOS}.

More generally, if $Q$ is a nonzero polynomial in $\Q[u_1, \dots, u_n]$, we may consider the elliptic fibration
\begin{equation}
\label{eq:general_q_twist_ell}
\mcA: y^2 = x^3 + Q^2ax + Q^3b.
\end{equation}
For this fibration, we again have that the $\mcTbnd(\mcA_x[\ell])$ have bounded moments as $x$ varies through $\mcX(\Z)$ whenever $\ell$ is odd.

In particular, suppose we have an elliptic curve $E$ above $\Q(u_1, \dots, u_n)$ with a degree $\ell$-isogeny for $\ell \ne 2, 3, 5, 7, 13$. This corresponds to a birational map from
$\Aff^n_{\Q}$ to  the modular curve $X_0(\ell)$ \cite[Ch 8]{Katz Mazur}, which must then be constant since $X_0(\ell)$ has positive genus. Then $E$ can be written as the generic fiber of an elliptic fibration of the form \eqref{eq:general_q_twist_ell}. So Conjecture \ref{conj:AV} implies the following:

\begin{conjecture}
Take $\mcA$ to be an arbitrary elliptic fibration above a dense open subscheme of $\mcX = \Spec\, \Z[u_1, \dots, u_n]$. Choose a prime $\ell$ other than $2, 3, 5, 7, 13$. Then for any $\kappa \ge 0$,
\[\frac{1}{\# \mcX'_H} \sum_{x \in \mcX'_H} \left(\#\Sel_{\ell}\, \mcA_x\right)^{\kappa} \asymp 1.\]
\end{conjecture}

We note that unbounded Selmer moments have been found for $\ell$ equal to any of $2, 3, 5, 7, 13$ \cite{Chan Verzobio}, so the condition of the conjecture is necessary.
\end{example}

\begin{example}
\label{ex:abelian_twist}
Our next example will be constructed to prove Theorem \ref{thm:quadtwist}. Take $A$ to be an abelian variety over a number field $F$, and take $P$ to be a nonzero polynomial in $F[u_1, \dots, u_n]$. Then we may define an abelian variety $\widetilde{A}$ over $F(u_1, \dots, u_n)$ as the quadratic twist of $A \times_F F(u_1, \dots, u_n)$ with respect to the quadratic character associated to $F(u_1, \dots, u_n)(\sqrt{P})$.

We then may choose an abelian fibration $\mcA$ over an open subset of $\mcX = \Spec\, \Z[u_1, \dots, u_n]$ such that the generic fiber $\mcA_{\eta}$ of $\mcA$ is identified with the Weil restriction of $\widetilde{A}$ to $\Q(u_1, \dots, u_n)$. Then there is some dense open subscheme of $\mcX$ so, for any rational point $\mbb$ in the subscheme, $\mcA_{\mbb}$ is identified with the Weil restriction of $A^{P(\mbb)}/F$ to $\QQ$.

We note that, for any $k \ge 0$, there is a number field $L$ such that $\mcA_{\eta}[2^k]$ carries the sign action corresponding to $L(u_1, \dots, u_n)(\sqrt{P})/L(u_1, \dots, u_n)$.

By a result of Faltings \cite{Faltings}, there are finitely many nonisomorphic abelian varieties $B$ isogenous to $\mcA_{\eta}$ over $\QQ(u_1, \dots, u_n)$. Write $B_1, \dots, B_k$ for this collection.

By our above observation on the $2^k$-torsion, we find there is some number field $L$ such that the points in $(B_1 \oplus \dots \oplus B_k)[2]$ are rational over $L(u_1, \dots, u_n)$. Then Theorem \ref{thm:AV} applies to any elliptic fibration $\mcB_1 \oplus \dots \oplus \mcB_k$ above an open set $W$ of $\mcX$ with this generic fiber. This gives
\begin{equation}
\label{eq:quad_twist_prefavored}
\frac{1}{\# \mcX'_H}\sum_{x \in \mcX'_H} \max_{ i \le k} \left(\frac{\#\Sel_2 \, \mcB_{ix}}{\mcTbnd(\mcB_{ix}[2])}\right)^{\kappa} \le \exp \exp(C\kappa) 
\end{equation}
for some $C \ge 0$ not depending on $\kappa$, and for all sufficiently large $H$.

We now claim that, for any $x \in \mcX'_H$,
\begin{equation}
\label{eq:favored}
\min_{i \le k} \mcTbnd(\mcB_{ix}) \ll 1.
\end{equation}
After all, if $T$ is a submodule of $\mcB_{ix}[2]$, then $\mcB_{ix}/T$ is isomorphic to $\mcB_{jx}$ for some $j \le k$, and we get
\[\mcT(T_x) \asymp \prod_{p \text{ of bad reduction for } \mcA_{x} } \frac{\#\mcB_{jx}(\Q_p)[2]}{\#\mcB_{ix}(\Q_p)[2]}.\]
From this observation, if $i$ is chosen so the product $\prod_{p \text{ bad}} \# \mcB_{ix}(\Q_p)[2]$ is maximized, we find that $\mcTbnd(\mcB_{ix}) \ll 1$, giving \eqref{eq:favored}. This and \eqref{eq:quad_twist_prefavored} imply
\begin{equation}
\label{eq:best_Bi}
\frac{1}{\# \mcX'_H}\sum_{x \in \mcX'_H} \min_{ i \le k} \left(\#\Sel_2 \, \mcB_{ix}\right)^{\kappa} \le \exp \exp(C\kappa)
\end{equation}
for some $C > 0$ not depending on $\kappa$.

Since $\# \Sel_2\, \mcB_{ix}$ is always at least $2^{\text{rank}(A^{P(x)}/F)}$, \eqref{eq:best_Bi} implies \eqref{eq:quadtwist_moments}. With this proved, Markov's inequality allows us to bound the left hand side of \eqref{eq:quadtwist_ranks} by $e^{-\kappa r} \exp \exp(C\kappa)$ for any $\kappa \ge 0$; taking $\kappa = C^{-1}\log r$ then gives the second part of the theorem.
\end{example}

\subsection{Application to class groups}
\label{sClass}
We now prove Theorem \ref{tClass}.

\begin{proof}[Proof of Theorem \ref{tClass}]
Fix an integer $k \geq 2$ and fix a squarefree polynomial $P(u) \in \Z[u]$. If $P(u)$ is constant, then Theorem \ref{tClass} is trivial. Henceforth we assume that the degree of $P(u)$ is at least $1$. 

We will now estimate
$$
\sum_{\substack{1 \leq b \leq H \\ P(b) \neq 0}} \# \mathrm{Cl}(\Q(\sqrt{P(b)}))[2^k] \leq \sum_{\substack{1 \leq b \leq H \\ P(b) \neq 0}} \# 2\mathrm{Cl}(\Q(\sqrt{P(b)}))[4]^{k - 1} \cdot 2^{\omega(P(b))},
$$
where we used the bound $\# 2\mathrm{Cl}(\Q(\sqrt{P(b)}))[2^k] \leq \# 2\mathrm{Cl}(\Q(\sqrt{P(b)}))[4]^{k - 1}$ and Gauss genus theory $\# \mathrm{Cl}(\Q(\sqrt{P(b)}))[2] \leq 2^{\omega(P(b))}$. We take
$$
X := \{b \in \Z_{\geq 1} : P(b) \neq 0\}
$$
and $h: X \rightarrow \mathbb{R}_{\geq 0}$ to be $h(b) := 2|P(b)|$. To each $b \in X$, we associate the quasi-decorated module $M_b := (\mathbb{F}_2, (\msL_{bp})_p)$, where
$$
\msL_{bp} =
\begin{cases}
H^1(G_p, \mathbb{F}_2) &\text{if } p \in \{2, \infty\} \\
H^1_{\text{ur}}(G_p, \mathbb{F}_2) &\text{if } p \nmid P(b) \\
\langle \chi_{P(b)} \rangle &\text{if } p \mid P(b), p^2 \nmid P(b) \\
\emptyset &\text{if } p^2 \mid P(b).
\end{cases}
$$
Then we have 
$$
\# 2\mathrm{Cl}(\Q(\sqrt{P(b)}))[4] \leq \# \Sel \, M_b, \quad \quad 1 \ll \mcTbnd(M_b) \ll 1.
$$
We now apply Theorem \ref{thm:main} with $\kappa = k - 1$ and $\tilde{g}_p(\msL) = 2^{1/(k - 1)}$ if $\msL \neq H^1_{\text{ur}}(G_p, \mathbb{F}_2)$. The set of weak equivalence classes are $\{0, H^1_{\text{ur}}(G_p, \FF_2), H^1(G_p, \FF_2), \{\langle \chi_p \rangle, \langle \chi_{\epsilon p} \rangle\}\}$, where $\epsilon$ is a nonsquare unit in $\Q_p^\ast$. It is then straightforward to verify the hypotheses in Definition \ref{defn:effective-constant}. Write $\delta_P(p)$ for the density of $b$ such that $P(b) \equiv 0 \bmod p$. Then we have
\begin{align}
\sum_{\substack{1 \leq b \leq H \\ P(b) \neq 0}} \# 2\mathrm{Cl}(\Q(\sqrt{P(b)}))[4]^{k - 1} \cdot 2^{\omega(P(b))} 
&\ll_k \sum_{\substack{1 \leq b \leq H \\ P(b) \neq 0}} \left(\frac{\# \Sel \, M_b}{\mcTbnd(M_b)}\right)^{k - 1} \cdot \tilde{g}(\msL_b)^{k - 1} \nonumber \\
&\ll_{k, P} H \prod_{p \leq H} (1 - \delta_P(p)) (1 + 2\delta_P(p)) \nonumber \\
&\ll_{k, P} \prod_{p \leq H} (1 + \delta_P(p)). \label{ePreLandau}
\end{align}
If $p$ does not divide the product of the pairwise resultants of the irreducible factors $f_1, \dots, f_r$ of our polynomial $P = f_1 \cdot \ldots \cdot f_r$, then we have the identity
$$
\delta_P(p) = \sum_{i = 1}^r \delta_{f_i}(p).
$$
By the Landau prime ideal theorem applied to the number field $K_f := \Q[t]/f(t)$, we see that the average of $\delta_f(p)$ over the primes $p$ is equal to $1$. Hence equation \eqref{ePreLandau} is at most $\ll_{k, P} H (\log H)^r$, as desired.
\end{proof}

\subsection{\texorpdfstring{Application to $2\phi$-Selmer groups}{Application to 2phi-Selmer groups}}
\label{s3Selmer}
Finally, we prove Theorem \ref{t3Selmer}. Since the lower bound is trivial, it suffices to establish the upper bound in Theorem \ref{t3Selmer}. We note that an asymptotic for the average $\phi$-Selmer group follows from the works \cite{BKLOS, BSW} and an asymptotic for the average $2$-Selmer group follows from \cite{Smi22a, Smi22b} under mild conditions. Informally speaking, Theorem \ref{t3Selmer} shows that $\Sel_2$ and $\Sel_{\phi_t}$ are uncorrelated in a weak sense.

\begin{proof}[Proof of Theorem \ref{t3Selmer}]
We write $\chi_d: G_\Q \rightarrow \mu_2$ for the quadratic character attached to a class $d \in \Q^\ast/\Q^{\ast 2} \cong \{\text{squarefree integers}\}$ by Kummer theory. Since $A[\phi]$ is a Galois module of size equal to $3$, we must have $A[\phi] \cong \mathbb{F}_3(\chi_d)$ for some squarefree integer $d$. Define $\Vplac_0$ to be the union of the rational places dividing $6d\infty$ and the places of bad reduction of $A$. We take $M : = A[2]$, we take
$$
X := \{(t, \psi) : t \text{ squarefree}, \, \psi \in \Sel_{\phi_t} \, A^t\},
$$
and we take $h: X \rightarrow \R_{\geq 0}$ to be the height function given by $h(t, \psi) := |t| \prod_{p \in \Vplac_0 - \{\infty\}} p$. To each pair $x = (t, \psi) \in X$, we attach the decorated module
$$
M_x := A^t[2]
$$
endowed with the usual $2$-Selmer local conditions $(\msL_{tp})_p$ (which depend only on $t$!). We now apply Theorem \ref{thm:main} with $\tilde{g}$ identically equal to $1$. We claim that
\begin{equation}
\label{eTamaComp}
\mcTbnd(A^t[2]) = 1.
\end{equation}
Recall that $A^t[2] = A[2]$ is irreducible by assumption. Moreover, taking cohomology of
$$
0 \rightarrow A^t[2] \rightarrow A^t \xrightarrow{\cdot 2} A^t \rightarrow 0
$$
shows that $\# \delta(A^t(\Q_p)/2A^t(\Q_p)) = \#H^0(G_p, A^t[2])$, where $\delta: A^{t}(\Q_p) \to H^1(G_p, A^t[2])$ is the connecting map associated to this exact sequence. So the Tamagawa ratio $\mcT(A^t[2], A^t[2])$ is $1$, and thus $\mcTbnd(A^t[2]) = 1$ by irreducibility. We have now established the claim \eqref{eTamaComp}.

Thus, by equation \eqref{eTamaComp}, if $\{M_x : x \in X\}$ is a constant-module family with effectively equidistributed local conditions as in Definition \ref{defn:effective-constant}, then Theorem \ref{thm:main} implies the existence of $C, C' > 0$ such that
$$
\sum_{\substack{|t| \leq H \\ t \textup{ sqf.}}} \# \Sel_{2\phi_t} \, A^t = \sum_{x \in X_H} \# \Sel \, M_x \leq C \sum_{x \in X_H} 1 \leq C'H,
$$
as desired (indeed, we will soon see that in fact $\sum_{x \in X_H} 1 \sim cH$ for some $c > 0$). Hence it remains to check the three conditions in Definition \ref{defn:effective-constant}. 

We start with (1). Let $Q \leq H^c$ (for any fixed $c < 1/2$) be squarefree and coprime to the places in $\Vplac_0$, and let $(\msL_p)_{p \mid Q}$ with each $\msL_p \subseteq H^1(G_p, A[2])$. Our task is to give a good asymptotic formula for
\begin{equation}
\label{ePhiDistr}
\sum_{\substack{|t| \leq H, \ t \textup{ sqf.} \\ \msL_{tp} = \msL_p \text{ for all } p \mid Q}} \# \Sel_{\phi_t} \, A^t.
\end{equation}
Since $A^t$ is determined as an abelian variety over $\Q_p$ by the restriction of $t$ in $\Q_p^\ast/\Q_p^{\ast 2}$, the condition $\msL_{tp} = \msL_p$ is completely determined by the class of $t$ in $\Q_p^\ast/\Q_p^{\ast 2}$.

We shall now reduce the task \eqref{ePhiDistr} to counting certain cubic fields, for which we will ultimately rely on the counting result in \cite[Theorem 1.3]{BTT}. In order to do so, we start by recalling that $A[\phi] \cong \mathbb{F}_3(\chi_d)$ and hence $A^t[\phi_t] \cong \mathbb{F}_3(\chi_{dt})$. Now let $p \not \in \Vplac_0$ with $p \mid t$. We know that $I_p$ acts trivially on $\mathbb{F}_3(\chi_d)$, and therefore any generator of $I_p$ acts by inversion on $\mathbb{F}_3(\chi_{dt})$. Using the inflation--restriction exact sequence
$$
0 \rightarrow H^1(\hat{\Z}, \mathbb{F}_3(\chi_{dt})^{I_p}) \rightarrow H^1(G_p, \mathbb{F}_3(\chi_{dt})) \rightarrow H^1(I_p, \mathbb{F}_3(\chi_{dt}))^{\langle \Frob_p \rangle} \rightarrow 0,
$$
we conclude that $H^1(G_p, \mathbb{F}_3(\chi_{dt})) = 0$. This yields
$$
\delta((A')^t(\Q_p)/\phi_t(A^t(\Q_p))) \subseteq H^1(G_p, \mathbb{F}_3(\chi_{dt})) = 0,
$$
where $\delta$ is the connecting map associated to the exact sequence
$$
0 \rightarrow A^t[\phi_t] \rightarrow A^t \xrightarrow{\phi_t} (A')^t \rightarrow 0.
$$
Clearly, we also have $\delta((A')^t(\Q_p)/\phi_t(A^t(\Q_p))) = H^1_{\text{ur}}(G_p, \mathbb{F}_3(\chi_{dt}))$ for $p \not \in \Vplac_0 \cup \{q \mid t\}$. We conclude that for $p \not \in \Vplac_0$
\begin{equation}
\label{eLocalIsogenySelmer}
\delta((A')^t(\Q_p)/\phi_t(A^t(\Q_p))) =
\begin{cases}
0 &\text{if } p \mid t \\
H^1_{\text{ur}}(G_p, \mathbb{F}_3(\chi_{dt})) &\text{if } p \nmid t.
\end{cases}
\end{equation}
If we further fix the finitely many classes of $t$ in $\Q_p^\ast/\Q_p^{\ast 2}$ at the places $p \in \Vplac_0$, then the isomorphism class of the abelian variety $A^t/\Q_p$ is fixed for all $p \in \Vplac_0$. After fixing $t$ in this way, it follows from equation \eqref{eLocalIsogenySelmer} that $\# \Sel_{\phi_t} \, A^t$ counts the number of cubic fields $L$ with resolvent field $\Q(\sqrt{dt})$ with given local conditions at $\Vplac_0$ and the prime divisors of $Q$ and satisfying the condition that $p \mid \Delta_L$ implies $p \in \Vplac_0$ or $p^2 \nmid \Delta_L$. Therefore the desired effective distribution for equation \eqref{ePhiDistr} is a consequence of \cite[Theorem 1.3]{BTT}.

Part (2) now follows from the explicit formula for the leading constant in \cite[Theorem 1.3]{BTT}. Part (3) is a consequence of the fact that for all $p \not \in \Vplac_0$ and all $x = (t, \psi) \in X$ with $p \mid t$, the local conditions $\msL_{px}$ of the module $M_x = A^t[2]$ are determined by the splitting of $p$ in the extension $\Q(A[4])/\Q$.
\end{proof}

\end{document}